\documentclass[11pt,reqno]{article}
\usepackage[a4paper,margin=1in]{geometry}
\usepackage{amsmath,amssymb,amsthm,mathtools}
\usepackage{bm}
\usepackage{graphicx}
\usepackage{enumitem}
\usepackage[numbers,sort&compress]{natbib}
\usepackage{hyperref}
\usepackage{float}
\usepackage{booktabs}
\usepackage{multirow}

\hypersetup{colorlinks=true,linkcolor=blue,citecolor=blue,urlcolor=blue}

\newtheorem{theorem}{Theorem}[section]
\newtheorem{lemma}[theorem]{Lemma}

\theoremstyle{definition}
\newtheorem{definition}[theorem]{Definition}
\newtheorem{remark}[theorem]{Remark}
\newtheorem{assumption}[theorem]{Assumption}

\DeclareMathOperator{\diag}{diag}
\newcommand{\bu}{\bm{u}}
\newcommand{\bw}{\bm{w}}
\newcommand{\bv}{\bm{v}}
\newcommand{\bff}{\bm{f}}
\newcommand{\be}{\bm{e}}
\newcommand{\eps}{\varepsilon}
\newcommand{\Om}{\Omega}
\newcommand{\dt}{\delta t}
\newcommand{\R}{\mathbb{R}}
\newcommand{\N}{\mathbb{N}}
\newcommand{\norm}[1]{\lVert#1\rVert}

\newcommand{\Ck}{C_k}
\newcommand{\Ak}{A_k}
\newcommand{\Bk}{B_k}
\newcommand{\Dk}{D_k}
\newcommand{\Fk}{F_k}
\newcommand{\Reyn}{\mathrm{Re}}
\newcommand{\HS}{\cite{HuangShen2025}}
\DeclareMathOperator{\diver}{div}

\title{A spectral-vanishing-viscosity stabilization
of a higher-order consistent splitting scheme 
for the Navier--Stokes equations}

\author{M Nader Alhomsi\footnote{Department of Mathematics,  Central Michigan University, 
Mount Pleasant, MI 48858. Email: Alhom1n@cmich.edu}\,,
Akram Moustafa\footnote{Department of Computer Science,  Central Michigan University, 
Mount Pleasant, MI 48858. Email: moust1am@cmich.edu}\,,
Mohammad Al-Saqqa\footnote{Department of Mathematics,  Central Michigan University, 
Mount Pleasant, MI 48858. Email: alsaq1m@cmich.edu}\,,
 Jiahong Wu\footnote{Department of Mathematics,  University of Notre Dame, Notre Dame, IN 46556. Email: jwu29@nd.edu}\,
and Xiaoming Zheng\footnote{Department of Mathematics, Central Michigan University, Mount Pleasant, MI 48858. Email: zheng1x@cmich.edu}
}
\date{}
\begin{document}
\maketitle

\begin{abstract}
Huang and Shen~\cite{HuangShen2025} developed a novel class of high-order BDF--IMEX consistent-splitting schemes for the incompressible Navier--Stokes equations, giving the first rigorous stability and error analysis for a fully decoupled splitting scheme of temporal order higher than two. Extending their analysis from unit viscosity to arbitrary viscosity, this work reveals that the error upper bound coefficient contains inverse powers of the viscosity. Our numerical experiments further confirm that the scheme can break down at high Reynolds number. To prevent this failure, we stabilize the scheme by adding to the velocity update a symmetric positive-semidefinite spectral vanishing viscosity operator, built from the directionally applied Maday--Kaber--Tadmor kernel, which selectively damps the high, under-resolved modes at no additional asymptotic cost. We establish stability and error estimates for the stabilized scheme in which the spectral vanishing viscosity provides viscosity-independent coercive control of the high modes.
Three two-dimensional tests demonstrate the robustness and accuracy of the stabilized scheme. For a manufactured solution at $\Reyn=10^4$, the unstabilized scheme diverges for most time steps. The stabilization enlarges the range of stable time steps, and the stabilized scheme then converges at its design order for $k=2,3,4$ down to an accuracy floor set by the stabilization, which decreases algebraically with the resolution. For a perturbed Kovasznay flow, it accurately resolves the boundary layer at $\Reyn=10^4$ and drives the perturbation back to the steady state, while the unstabilized scheme blows up. For the Kelvin--Helmholtz instability problem, it reproduces the reference integral diagnostics throughout the reliable regime, whereas the unstabilized scheme produces spurious solutions or blows up. 
\end{abstract}


\section{Introduction}

The unsteady incompressible Navier--Stokes equations on a bounded domain
$\Om\subset\R^d$ ($d=2,3$),
\begin{equation}
\label{eq:NSE-intro}
\partial_t\bu + \bu\cdot\nabla\bu = -\nabla p + \nu\Delta\bu + \bff,
\qquad \diver\bu = 0,
\end{equation}
with boundary condition $\bu|_{\partial\Om}=\bm 0$
admit at most one strong solution under standard hypotheses on $\bff$ and
$\bu(0)$ \cite{Temam1984}, and the design of higher-order time-stepping schemes that are
both stable and accurate at large Reynolds number $\Reyn\sim 1/\nu$ remains an
ongoing challenge in computational fluid dynamics \cite{XuPasquetti2004,Pasquetti2006}.
Numerical methods for \eqref{eq:NSE-intro} fall broadly into fully
coupled mixed formulations
\citep{GiraultRaviart1979,BrezziFortin1991,ElmanSilvesterWathen2014},
which solve for the velocity and pressure simultaneously, and decoupled
strategies that split them at each step: projection and
pressure-/velocity-correction methods
\citep{Chorin1968,ELiu1995,KarniadakisIsraeliOrszag1991,Shen1992,TimmermansMinevVandeVosse1996,GuermondShenVelocity2003,GuermondShen2004,GuermondSalgado2011,GuermondMinevSalgado2012,OrszagIsraeliDeville1986,Prohl1997,Shen2012},
gauge methods \citep{ELiuGauge2003,WangLiu2000,NochettoPyo2005}, and
consistent splitting methods
\citep{GuermondShen2003,ShenYang2007,JohnstonLiu2004,WuHuangShen2022,HuangShen2023b,HuangShen2025};
see \citet{GuermondMinevShen2006} for an overview of the decoupled
approach. Fully coupled formulations require solving a large, indefinite
velocity--pressure saddle-point system at every time step, which is
computationally expensive and demands specialized solvers and
preconditioners. Splitting (or pressure-correction) schemes instead decouple
this system into a sequence of standard elliptic problems at each step, a
convection--diffusion solve for each velocity component and a Poisson solve
for the pressure, which can be handled by fast, well-established solvers and
scale well to large problems, accounting for their wide use in large-scale
simulation \citep{Chorin1968,KarniadakisIsraeliOrszag1991,GuermondMinevShen2006}.
Classical projection methods incur a well-known \emph{splitting error},
caused by an inconsistent artificial pressure boundary condition, that
limits their accuracy in the velocity $H^1$ norm and in the pressure
\citep{GuermondMinevShen2006,GuermondShen2004}. Consistent splitting
schemes \citep{GuermondShen2003,HuangShen2025} were introduced to remove
it. One central difficulty of numerical methods for turbulent flows, present even for fully coupled discretizations, is robustness, namely whether the coefficients in the error upper bound contain inverse powers of the viscosity \citep{GarciaArchillaJohnNovo2021, deFrutosGarciaArchillaJohnNovo2016, AlhomsiWuZheng2026}.

\citet{HuangShen2025} recently introduced a novel class of higher-order implicit-explicit (IMEX), consistent splitting schemes for \eqref{eq:NSE-intro} that employ Taylor-shifted backward differentiation formulas (BDF) of order
$k\in\{2,3,4\}$ followed by a curl--curl pressure Poisson update. The construction builds on their framework of generalized BDF/IMEX multipliers for parabolic problems \citep{HuangShen2024},  with a second-order such scheme implemented for the Navier--Stokes equations \citep{HuangShen2023b} and a perturbed Boussinesq system \cite{alhomsiGSAVBoussinesq2026}. Under the splitting,  the velocity step is implicit on the viscous term and explicit on the convection term (IMEX), while the pressure step is a decoupled discrete Poisson problem. These provide the first rigorous stability and convergence analysis, with optimal global-in-time error estimates in both two and three dimensions, for a fully decoupled splitting scheme of order higher than two for the Navier--Stokes equations. These estimates, however, are not robust in the inviscid limit: as shown in Theorem\,\ref{thm:main}, the constant in the error upper bound contains negative powers of the viscosity.

A closely related work, \citet{GarciaArchillaJohnNovo2025},  utilizes the classical BDF time stepping (the $\beta=1$ member of the Taylor-shifted family), a fully coupled and implicit velocity--pressure  discretization for the Navier--Stokes equations in a  finite-element setting.
They prove optimal-order error bounds in time (not in space) with inf-sup stable mixed elements,
relying, exactly as \cite{HuangShen2025}, on the G-stability of BDF as the central temporal tool and reaching comparably high order in time.
Most relevant for the present paper is the route to $\nu$-robustness: \citet{GarciaArchillaJohnNovo2025}
uses fully implicit skew-symmetric convection augmented with grad--div stabilization, which yields error constants
independent of inverse powers of $\nu$. By contrast, the estimate for the present splitting scheme
(Section~\ref{sec:HS}, Theorem~\ref{thm:main}) still retains the $\nu^{-5}$ degeneracy inherited from \cite{HuangShen2025}. The spectral-vanishing-viscosity term introduced in Section~\ref{sec:scheme} plays, for the spectral consistent splitting scheme, the same coercive stabilizing role that grad--div plays in the
coupled finite-element analysis of \citet{GarciaArchillaJohnNovo2025}.

Because the error bound for the splitting scheme in \cite{HuangShen2025} deteriorates as the viscosity approaches zero, and our numerical experiments confirm that the scheme can fail at high Reynolds numbers, we propose spectral vanishing viscosity (SVV) as a remedy.
The idea of SVV originates with \citet{Tadmor1989}, who introduced spectral vanishing
viscosity for the Fourier approximation of nonlinear scalar conservation laws as a
way to reconcile two competing demands: enough dissipation to
enforce the entropy condition and recover the physically relevant weak solution,
yet little enough that spectral accuracy is preserved. The mechanism is to activate an
artificial viscosity \emph{only} above a cut-off wavenumber $m_N$, with an
amplitude $\eps_N\to 0$ as $N\to\infty$. \citet{MadayKaberTadmor1993} carried the
construction to the non-periodic Legendre pseudo-spectral setting and supplied the
smooth kernel in common use since, which we adopt in \eqref{eq:MKTkernel}, together with
the scalings $\eps_N\sim N^{-1}$, $m_N\sim\sqrt N$. Convergence for
multidimensional conservation laws was established by
\citet{ChenDuTadmor1993,GuoMaTadmor2001}, and the method was extended to
Chebyshev discretizations by \citet{Andreassen1994}, who also gave the first
two-dimensional fluid application (waves in a stratified atmosphere). We refer to
\citet{Tadmor1998} for a survey of approximate solutions of nonlinear conservation laws.

The transfer of SVV from conservation laws to the incompressible Navier--Stokes
equations was made by \citet{KaramanosKarniadakis2000}, who observed that in an
under-resolved simulation a spectral method interprets a steep gradient as a
discontinuity, so that the same entropy-dissipation argument applies. They
formulated SVV for spectral/$hp$ elements and validated it on the Kovasznay flow
and on turbulent channel flow at $\Reyn_\tau=180$ and $395$. The method was then
developed systematically by Pasquetti, Xu and co-workers into a practical tool for
\emph{high Reynolds number} computation: \citet{PasquettiXu2002} and
\citet{XuPasquetti2004} showed that SVV stabilizes spectral-element computations
that are otherwise unstable, while preserving exponential convergence, and applied
it to the turbulent wake of a cylinder \citep{Xu2006,Pasquetti2006,RongXu2009},
with \citet{KirbySherwin2006} establishing the same for spectral/$hp$ elements.
Used as a subgrid model, SVV has since been applied to turbulent flows
at Reynolds numbers beyond the reach of direct simulation:
\citet{SeveracSerre2007} computed the three-dimensional turbulent rotor--stator
cavity at $\Reyn = 7\times10^{4}$--$7\times10^{5}$, matching both DNS and
experiment, \citet{MinguezPasquettiSerre2008} performed an SVV-LES of the flow
over the Ahmed body at $\Reyn\approx 7.7\times10^{5}$, and
\citet{ChenPasquettiXu2021} extended SVV to triangular spectral elements and
computed the backward-facing step at $\Reyn=10^{4}$ and $5\times 10^{4}$, reporting
that the unstabilized spectral element method is \emph{unstable} at the higher
Reynolds number while the SVV-stabilized method is not, the same dichotomy we
report in Section~\ref{sec:numerics}. \citet{MouraSherwinPeiro2016} give an
eigensolution analysis of SVV-stabilized advection--diffusion and caution that not
every kernel behaves as intended. The essential point that motivates the present
work is that spectral methods are much less numerically dissipative than low-order
methods, so the energy that should be dissipated at the grid scale
instead \emph{accumulates} there and causes the computation to break down, the
failure mode we document in Section~\ref{sec:numerics} for the bare scheme.

Two points from this literature bear directly on what follows. First, the
one-dimensional theory is well established, but the extension of the kernel to several
dimensions is \emph{not} canonical: \citet[\S3.2]{SeveracSerre2007} note that
``there is not a direct way'' to extend the one-dimensional definition, and
different authors have made different choices. We adopt the directional
(diagonal) form of \citet{SeveracSerre2007} and \citet{ChenPasquettiXu2021}; see
\eqref{eq:dirkernel}. Second, essentially all of the work above is
\emph{computational}: SVV is used as a stabilization or subgrid device and its
effect on the flow is assessed a posteriori. What is missing, and what this paper
supplies, is a convergence analysis of an SVV-stabilized \emph{higher-order splitting
scheme} in which the SVV term is shown to enter the energy estimate as a coercive
contribution, the role played by grad--div in the coupled finite-element
analysis of \citet{GarciaArchillaJohnNovo2025}.

The stabilization used here differs from these earlier constructions in one respect that matters for both the algorithm and the analysis. In the spectral-element and collocation implementations cited above, the kernel acts on the modal coefficients of the polynomial expansion. Here the kernel is applied in the simultaneous-diagonalization eigenbasis of the Legendre--Galerkin Laplacian, so that the SVV operator is a function of the discrete Laplacian, $S_N=-\eps_N Q_N\Delta$ with $Q_N$ commuting with $\Delta$ (Lemma~\ref{lem:svv-props}). Three consequences follow. First, the operator enters the velocity step implicitly, at the same shifted time level and through the same extrapolation $\Bk$ as the physical viscosity, and it only modifies the diagonal of the existing solver, so the pressure step and the per-step cost are unchanged. Second, because $Q_N$ commutes with $\Delta$, the multiplier and G-stability arguments of \cite{HuangShen2025} apply mode by mode to the SVV term (Lemma~\ref{lem:SVVcoerc}), which makes a stability and error analysis of the stabilized higher-order scheme possible. We are not aware of such an analysis for the earlier SVV formulations of the Navier--Stokes equations. Third, the kernel no longer acts on a basis in which smooth functions have spectrally decaying coefficients, which produces the algebraic accuracy floor quantified in Section~\ref{sec:conv}. The kernel and its directional use are taken from the literature. The contribution of this paper is this formulation, its analysis for arbitrary viscosity, and its numerical assessment.

We propose, analyze, and validate an SVV-stabilized variant of the \cite{HuangShen2025} BDF--IMEX schemes of orders $k=2,3,4$. Section~\ref{sec:HS} recalls the
scheme,
Section~\ref{sec:scheme} introduces the spectral-vanishing-viscosity
stabilization and its efficient implementation,
Section~\ref{sec:analysis} establishes the corresponding energy and error
estimates and identifies the source of the $\nu^{-5}$ factor, Section~\ref{sec:numerics} reports numerical experiments on a
manufactured-solution convergence test, the perturbed Kovasznay flow, and
the Kelvin--Helmholtz problem, and Section~\ref{sec:concl} collects the
conclusions.

\section{The consistent splitting scheme and its SVV stabilization}
\label{sec:consistentsplitting}
\subsection{The higher-order BDF/IMEX consistent splitting scheme in \cite{HuangShen2025} }
\label{sec:HS}

Let $\Om\subset\R^d$ ($d=2, 3$) be a bounded domain. We consider the Navier--Stokes initial--boundary
value problem
\begin{equation}\label{eq:NSE}
\partial_t\bu + \bu\cdot\nabla\bu - \nu\Delta\bu + \nabla p = \bff,
\quad \diver\bu = 0,
\quad \bu|_{\partial\Om} = \bm 0,
\quad \bu(0)=\bu_0,
\end{equation}
on $(0,T)\times\Om$. 
We denote by $(\cdot,\cdot)$ the $L^2$ inner product on $\Omega$, and by $\norm{\cdot}$, $\norm{\cdot}_1,\norm{\cdot}_2$  the norms of $L^2(\Omega)$, $H^1(\Omega)$,  and $H^2(\Omega)$, respectively.
The product spaces are  $\bm L^2(\Omega)=(L^2(\Omega))^d$, $\bm H^1_0(\Omega)=(H^1_0(\Omega))^d$, $\bm H^2(\Omega)= (H^2(\Omega))^d$, etc. 
In the theoretical analysis, we adopt the following assumptions.
\begin{assumption}\label{ass:smooth}
Let $\Om\subset\mathbb R^d$ ($d\le3$) be a connected bounded domain with $C^3$ boundary. 
The initial velocity satisfies $\bu^0\in\bm H^1_0(\Omega)\cap\bm H^2(\Omega)$ with $\nabla\cdot \bu^0=0$.
The velocity $\bu$  satisfies
$\bu\in L^\infty(0,T;\bm H^1_0(\Omega)\cap\bm H^2(\Omega))$, $\partial^k_t\bu\in L^\infty(0,T;\bm H^2(\Omega))$,
$\partial_t^{k+1}\bu\in L^\infty(0,T;\bm L^2(\Omega))$,
and $\partial_t^k(\bu\!\cdot\!\nabla\bu)\in L^2(0,T;\bm  L^2(\Omega))$. 
The pressure satisfies $p\in L^\infty(0,T;H^1(\Omega)/\R)$,
$\partial_t^k p\in L^2(0,T;H^1(\Omega))$.
The external force satisfies $\bff\in L^\infty(0,T;\bm L^2(\Omega))$.
\end{assumption}

The higher-order Taylor-shifted consistent splitting scheme of \cite{HuangShen2025} is given as follows. For each order $k\in\{2,3,4\}$, one chooses a
Taylor-shift parameter $\beta_k\in\N$ (the values $\beta_2=3,\beta_3=6,\beta_4=9$ are recommended). 
Set $t^n = n\dt$ and denote $\bm u^n$ and $p^n$ as the numerical solution at $t^n$. 
For $k=2,3,4$ the scheme reads: given $\bu^{n-k+1},\dots,\bu^n$ and
$p^{n-k+1},\dots,p^n$, $n\ge k-1$, 
find $\bu^{n+1}$ and $p^{n+1}$ satisfying
\begin{subequations}\label{eq:HS-scheme}
\begin{align}
\frac{\Ak(\bu^{n+1})}{\dt} &- \nu\Delta\Bk(\bu^{n+1}) + \nabla \Ck(p^n)
+ \Ck(\bu^n)\!\cdot\!\nabla \Ck(\bu^n) = \bff^{n+\beta_k},
\label{eq:HS-vel}\\
(\nabla p^{n+1},\nabla q) &= (\bff^{n+1}-\bu^{n+1}\!\cdot\!\nabla\bu^{n+1},\nabla q)
   - \nu(\nabla\times\nabla\times\bu^{n+1},\nabla q),
\quad\forall q\in H^1(\Om).
\label{eq:HS-pres}
\end{align}
\end{subequations}
Given a sequence $\{\phi^n\}$,  $A_k(\phi^{n+1})/\delta t$ is a $k$-th order approximation of $\partial_t \phi(t^{n+\beta_k})$ by using $\{\phi^{n+1-k}, \cdots, \phi^{n+1}\}$, $B_k(\phi^{n+1})$ is a $k$-th order implicit extrapolation to fit $\phi(t^{n+\beta_k})$ through $\{\phi^{n+2-k}$, $\cdots, \phi^{n+1}\}$, and $C_k$ is a $k$-th order explicit extrapolation to approximate $\phi(t^{n+\beta_k})$ via $\{\phi^{n+1-k}, \cdots, \phi^{n}\}$.
Their concrete forms are provided in \cite{HuangShen2025} and listed in Table\,\ref{tab:HS-coefs}.
\begin{table}[htbp]
\centering
\small
\setlength{\tabcolsep}{4pt}
\renewcommand{\arraystretch}{1.5}
\caption{The operators $A_k$, $B_k$, $C_k$ of the BDF--IMEX consistent-splitting scheme in \cite{HuangShen2025} }
\label{tab:HS-coefs}
\begin{tabular}{|c|l|l|l|}
\hline
 & $k=2\ (\beta_2=3)$ & $k=3\ (\beta_3=6)$ & $k=4\ (\beta_4=9)$\\
\hline
$A_k(\phi^{n+1})$
 & $\tfrac{5}{2}\phi^{n-1}-6\phi^{n}+\tfrac{7}{2}\phi^{n+1}$
 & $\begin{array}{@{}l@{}}-\tfrac{107}{6}\phi^{n-2}+59\phi^{n-1}\\[2pt]\quad{}-\tfrac{131}{2}\phi^{n}+\tfrac{73}{3}\phi^{n+1}\end{array}$
 & $\begin{array}{@{}l@{}}\tfrac{1691}{12}\phi^{n-3}-604\phi^{n-2}+975\phi^{n-1}\\[2pt]\quad{}-\tfrac{2108}{3}\phi^{n}+\tfrac{763}{4}\phi^{n+1}\end{array}$\\
\hline
$B_k(\phi^{n+1})$
 & $-2\phi^{n}+3\phi^{n+1}$
 & $15\phi^{n-1}-35\phi^{n}+21\phi^{n+1}$
 & $-120\phi^{n-2}+396\phi^{n-1}-440\phi^{n}+165\phi^{n+1}$\\
\hline
$C_k(\phi^{n})$
 & $-3\phi^{n-1}+4\phi^{n}$
 & $21\phi^{n-2}-48\phi^{n-1}+28\phi^{n}$
 & $-165\phi^{n-3}+540\phi^{n-2}-594\phi^{n-1}+220\phi^{n}$\\
\hline
\end{tabular}
\end{table}

\subsection{The SVV-stabilized scheme}
\label{sec:scheme}
Since the pressure equation~\eqref{eq:HS-pres} is simply a Poisson problem with a natural Neumann boundary condition, it is solved by a standard spectral-Galerkin Poisson solver~\citep{Shen1994, CanutoHussainiQuarteroniZang2006}.
Below we focus on the velocity solver. For ease of presentation, this section considers homogeneous Dirichlet boundary conditions and employs the Legendre Galerkin spectral method. The formulation for periodic and free-slip boundary conditions using a Fourier–trigonometric basis is provided in Appendix~\ref{app:fourier-trig}.

Let $X_N=\{\,v\in P_N(-1,1):v(\pm1)=0\,\}$ be the one-dimensional space of
polynomials of degree at most $N$ vanishing at the endpoints, of dimension
$M=N-1$. We represent it in the Legendre--Galerkin basis $\phi_j=L_j-L_{j+2}$
($L_j$ the Legendre polynomial of degree $j$), $j=0,\dots,N-2$, for which the
stiffness and mass matrices $S_v$ and $M_v$ of the $H^1_0$ and $L^2$ inner
products are sparse \citep{Shen1994}. Let $\bv_i\in\R^M$ be the
$M_v$-orthonormal solutions of the generalized eigenproblem
$S_v\bv_i=\mu_i M_v\bv_i$, and set $E=[\bv_0,\dots,\bv_{M-1}]$, so that
$E^\top M_v E=I$ and $E^\top S_v E=\Lambda:=\operatorname{diag}(\mu_i)$
\cite[Theorem~7.6.4]{HornJohnson2013}. The columns of $E$ are the
$\phi$-coordinates of the eigenfunctions $\psi_i=\sum_{m=0}^{M-1}E_{mi}\phi_m$,
which form the \emph{simultaneous-diagonalization basis} of $X_N$, with
$(\psi_i,\psi_j)=\delta_{ij}$ and $(\psi_i',\psi_j')=\mu_i\delta_{ij}$. The
matrix $E$ is the change of basis between the assembly basis $\{\phi_j\}$ and
the eigen-basis $\{\psi_i\}$ used in the solver below.

The discrete velocity space is the tensor product
$\bm V_N=\bigl(X_N\otimes X_N\bigr)^2\subset\bm H^1_0(\Om)$ on $\Om=(-1,1)^2$,
each component $X_N\otimes X_N$ spanned by the Legendre--Galerkin assembly basis
$\Phi_{ij}=\phi_i(x)\,\phi_j(y)$, $0\le i,j\le M-1$.
Let $\{\Psi_{ij}\}_{0\le i,j\le M-1}$ be the tensor eigen-basis of $X_N\otimes X_N$
(with $\Psi_{ij}=\psi_i(x)\,\psi_j(y)$), which is orthonormal and simultaneously
diagonalizes the mass and stiffness forms:
\[
(\Psi_{ij},\Psi_{i'j'})=\delta_{ii'}\delta_{jj'},\qquad
(\nabla\Psi_{ij},\nabla\Psi_{i'j'})=(\mu_i+\mu_j)\,\delta_{ii'}\delta_{jj'}.
\]
Equivalently, on $\bm V_N$, $-\Delta\Psi_{ij}=(\mu_i+\mu_j)\Psi_{ij}$, and, separating
directions, $(\partial_x\Psi_{ij},\partial_x\Psi_{i'j'})=\mu_i\delta_{ii'}\delta_{jj'}$
and $(\partial_y\Psi_{ij},\partial_y\Psi_{i'j'})=\mu_j\delta_{ii'}\delta_{jj'}$.

\medskip
The SVV term damps the high, under-resolved modes while leaving the
well-resolved low modes untouched. To build it, we first introduce the
one-dimensional Maday--Kaber--Tadmor SVV kernel $\widehat Q_\ell$ on $\{\psi_{\ell}\}$ of $X_N$ \citep{Tadmor1989,MadayKaberTadmor1993},
\begin{equation}\label{eq:MKTkernel}
\widehat Q_\ell =
\begin{cases}
0, & 0\le \ell\le m_N,\\[2pt]
\exp\bigl(-(\ell-M)^2/(\ell-m_N)^2\bigr), & m_N < \ell \le M-1,
\end{cases}
\end{equation}
with parameters
\begin{equation}\label{eq:SVVparams}
\eps_N = \frac{C_{\rm svv}}{M}, \qquad m_N = \lceil\sqrt M\,\rceil,
\qquad C_{\rm svv} = O(1).
\end{equation}
It vanishes on the low modes $\ell\le m_N\sim\sqrt M$ and rises monotonically to
$\widehat Q_{M-1}\to 1$ at the highest mode. The kernel acts on the coefficients in the
eigenbasis $\{\psi_\ell\}$ rather than on the Legendre coefficients. Its effect on the
accuracy for smooth solutions is quantified in Section~\ref{sec:conv}.

In two dimensions we apply the kernel \emph{directionally}, following
\citet{SeveracSerre2007} and \citet{ChenPasquettiXu2021}, so that each coordinate
carries its own one-dimensional kernel.

\begin{definition}[SVV operator]\label{def:svv}
Define the SVV operator $S_N:\bm V_N\to\bm V_N$ by its action on the eigen-basis,
extended linearly and applied componentwise to vector fields: for
$\bu=\sum_{ij}\widehat u_{ij}\Psi_{ij}$,
\begin{equation}\label{eq:SNdef}
S_N\Psi_{ij}:=\eps_N\bigl(\widehat Q_i\mu_i+\widehat Q_j\mu_j\bigr)\Psi_{ij},
\qquad
S_N\bu=\eps_N\sum_{ij}\bigl(\widehat Q_i\mu_i+\widehat Q_j\mu_j\bigr)\widehat u_{ij}\,\Psi_{ij}.
\end{equation}
The associated scalar multiplier $Q_N$, its square root, and the two directional
multipliers, all diagonal in $\{\Psi_{ij}\}$, are
\begin{equation}\label{eq:Qtilde}
Q_N\Psi_{ij}:=\widehat Q_{ij}\,\Psi_{ij},\quad
\sqrt{Q_N}\,\Psi_{ij}:=\sqrt{\widehat Q_{ij}}\,\Psi_{ij},\quad
\widehat Q_{ij}:=\frac{\widehat Q_i\mu_i+\widehat Q_j\mu_j}{\mu_i+\mu_j},
\end{equation}
\begin{equation}\label{eq:dirkernel}
\widehat Q^{\,x}\Psi_{ij}:=\widehat Q_i\,\Psi_{ij},\qquad
\widehat Q^{\,y}\Psi_{ij}:=\widehat Q_j\,\Psi_{ij},\qquad
\bm{\mathcal Q}_N:=\diag\bigl(\widehat Q^{\,x},\widehat Q^{\,y}\bigr).
\end{equation}
\end{definition}

\begin{lemma}[Properties of the SVV operator]\label{lem:svv-props}
The operator $S_N$ of Definition~\ref{def:svv} has the following properties.
\begin{enumerate}
\item[\textup{(P1)}] \emph{Strong form.} $S_N=-\eps_N\,Q_N\Delta$ on $\bm V_N$.
\item[\textup{(P2)}] \emph{Directional (divergence) form.}
  $S_N\bu=-\eps_N\diver\!\bigl(\bm{\mathcal Q}_N\nabla\bu\bigr)$ on $\bm V_N$, in the
  weak sense.
\item[\textup{(P3)}] \emph{Diagonal, self-adjoint, commuting.} $Q_N$ and
  $\sqrt{Q_N}$ are diagonal in $\{\Psi_{ij}\}$, hence self-adjoint, and commute with
  $\Delta$ on $\bm V_N$.
\item[\textup{(P4)}] \emph{Bounded symbol.} $\widehat Q_{ij}\in[0,1]$, hence
  $\norm{Q_N\bv}\le\norm{\bv}$ and $\norm{\sqrt{Q_N}\bv}\le\norm{\bv}$, and these
  bounds extend to all $\bv\in\bm L^2(\Om)$ by the truncated expansion.
\item[\textup{(P5)}] \emph{Vanishing on resolved modes.}
  $\widehat Q_{ij}=0\iff i\le m_N$ and $j\le m_N$.
\item[\textup{(P6)}] \emph{Positive semidefinite.}
  $(S_N\bu,\bu)=\eps_N\norm{\nabla\sqrt{Q_N}\,\bu}^2\ge0$; in particular $S_N$ is
  self-adjoint and positive semidefinite.
\end{enumerate}
\end{lemma}

The strong forms \textup{(P1)}--\textup{(P2)} of the SVV operator are proved in Appendix~\ref{app:svv-proof}.  Properties (P3)--(P6) are classical and go back to \citet{Tadmor1989},
\citet{MadayKaberTadmor1993}, and \citet{GuoMaTadmor2001}. The
two-dimensional directional (diagonal) kernel $\bm{\mathcal Q}_N$ is that of
\citet{SeveracSerre2007} and \citet{ChenPasquettiXu2021}.

The SVV-stabilized scheme replaces \eqref{eq:HS-vel} by
\begin{equation}\label{eq:SVV-vel}
\frac{\Ak(\bu^{n+1})}{\dt} - (\nu + \eps_N Q_N)\Delta\Bk(\bu^{n+1})
+ \nabla \Ck(p^n) + \Ck(\bu^n)\!\cdot\!\nabla \Ck(\bu^n) = \bff^{n+\beta_k},
\end{equation}
keeping the pressure step \eqref{eq:HS-pres} unchanged. The SVV contribution
$-\eps_N Q_N\Delta\Bk(\bu^{n+1})=S_N\Bk(\bu^{n+1})$ acts on $\Bk(\bu^{n+1})$, at
$t^{n+\beta_k}$ like the viscosity.
Note \eqref{eq:SVV-vel} is linear in $\bu^{n+1}$
and takes the form $\mathcal L\,\bu^{n+1}=\bm g$,
where
\begin{equation}\label{eq:Lop}
\mathcal L\,\bu^{n+1} := \frac{a_{k,k}}{\dt}\,\bu^{n+1}
- b_{k,k-1}\,(\nu + \eps_N Q_N)\,\Delta\bu^{n+1},
\end{equation}
where $a_{k,k}$ and $b_{k,k-1}$ are the coefficients of $\bu^{n+1}$ in
$\Ak$ and $\Bk$, respectively, and the right-hand side
$\bm g$ collects all other terms. 
In the eigen-basis $\{\Psi_{ij}\}$, 
\begin{equation}\label{eq:Ldiag}
(\widehat{\mathcal L u})_{ij}=\lambda_v(i,j)\,\widehat u_{ij},
\end{equation}
where $\widehat u_{ij}$ are the coefficients of $u$ in
$\{\Psi_{ij}\}$ and $\lambda_v(i,j)$ is given by
\begin{equation}\label{eq:lambda-SVV}
\lambda_v(i,j) = \frac{a_{k,k}}{\dt}
\;+\; b_{k,k-1}\,\Bigl[\nu\,(\mu_i+\mu_j)
      \;+\; \eps_N\bigl(\widehat Q_i\,\mu_i+\widehat Q_j\,\mu_j\bigr)\Bigr]
= \frac{a_{k,k}}{\dt}
\;+\; b_{k,k-1}\,\bigl(\nu + \eps_N\,\widehat Q_{ij}\bigr)\,(\mu_i+\mu_j),
\end{equation}
with $\widehat Q_{ij}$ as in \eqref{eq:Qtilde}. The directional kernel introduces no additional asymptotic computational cost: it changes only one diagonal entry.

Consequently the velocity step $\mathcal L\,\bu^{n+1}=\bm g$, applied to each
velocity component, is computed in four sub-steps that use the $\Phi_{ij}$
basis for assembly and the eigen-basis $\Psi_{ij}$ for inversion:
\begin{enumerate}
\item \emph{Assemble} the right-hand side $\bm g$ in the
  $\Phi_{ij}$ basis and obtain the coefficient
  matrix $G$.
\item \emph{Rotate} into the eigen-basis, $\widehat G = E^\top G\,E$.
\item \emph{Diagonal solve}: in $\{\Psi_{ij}\}$ the operator $\mathcal L$ is
  diagonal with entries $\lambda_v(i,j)$, so
  $\widehat U_{ij} = \widehat G_{ij}/\lambda_v(i,j)$.
\item \emph{Rotate back} to the $\Phi_{ij}$ basis, $U = E\,\widehat U\,E^\top$;
  the entries of $U$ are the coefficients of $\bu^{n+1}$.
\end{enumerate}
The per-step cost is therefore identical to the scheme without SVV, since the
stabilization only adds $\eps_N(\widehat Q_i\mu_i+\widehat Q_j\mu_j)$ to the
diagonal entry $\lambda_v(i,j)$. The directional kernel is thus free.

\subsection{Initialization of the multistep scheme}
\label{sec:init}

The $k$-step scheme \eqref{eq:HS-scheme} needs starting levels
$\bu^0,\dots,\bu^{k-1}$ and $p^0,\dots,p^{k-1}$. For a manufactured solution
these can be read off the exact fields, as in the exact-start runs of Section~\ref{sec:conv}. Otherwise,
as for the perturbed Kovasznay flow of Section~\ref{sec:kov}, they must be
generated to the design order without destabilizing the run. We use a
Richardson-extrapolated backward-Euler self-start
\citep{HairerWanner1996,Deuflhard1985}.

The building block is one backward-Euler substep, the $k=1$, $\beta_1=0$ member
of \eqref{eq:HS-scheme}, with the viscous and SVV terms implicit and the
convection and pressure gradient explicit. Advancing $(\bu^n,p^n)$ to
$(\bu^{n+1},p^{n+1})$ over a substep $\tau$,
\begin{equation}\label{eq:BEsub}
\frac{\bu^{n+1}-\bu^n}{\tau} - \nu\Delta\bu^{n+1} + S_N\bu^{n+1}
 = \bff(t^{n+1}) - (\bu^n\!\cdot\!\nabla)\bu^n - \nabla p^n,
\qquad \bu^{n+1}|_{\partial\Om}=0,
\end{equation}
after which $p^{n+1}$ follows from \eqref{eq:HS-pres}. By
\eqref{eq:lambda-SVV} this is a single diagonal solve with symbol
$1/\tau + \nu(\mu_i+\mu_j) + \eps_N(\widehat Q_i\mu_i+\widehat Q_j\mu_j)\ge 1/\tau>0$,
so the substep is $L$-stable and carries no viscous restriction, unlike a
Runge--Kutta self-start, which on the Gauss--Lobatto grid would require
$\tau\lesssim 1/(\nu N^4)$ \citep{WeidemanTrefethen1988}. Only the explicit
convective restriction of the main scheme remains, and refining the substep
relaxes it.

Backward Euler is first order, so its result at a fixed time admits an
asymptotic expansion in the substep $\tau$. For each level $j=1,\dots,k-1$ we
integrate $[0,t^j]$ with $\tau_\ell=\dt/2^\ell$ and $m_\ell=j\,2^\ell$ steps for
$\ell=0,\dots,L$, giving $\bu^{(j)}_\ell$, and combine them by the Neville
recursion \citep{HairerNorsettWanner1993}
\begin{equation}\label{eq:neville}
T^{(0)}_\ell=\bu^{(j)}_\ell,\qquad
T^{(r)}_\ell=T^{(r-1)}_\ell
   +\frac{T^{(r-1)}_\ell-T^{(r-1)}_{\ell-1}}{2^{r}-1},
\qquad r=1,\dots,L,\ \ell=r,\dots,L.
\end{equation}
Each column cancels one further term of the expansion, so $\bu^j:=T^{(L)}_L$
carries an error of order $\dt^{L+2}$, and $p^j$ is its consistent pressure from
\eqref{eq:HS-pres}. Taking $L=k-2$ makes every starting level accurate to order
$\dt^{k}$, the minimum for global order $k$. We use $L=k$, so that the starting
error stays below the temporal error at the finest $\dt$ and largest $N$
considered.

\section{Energy and error estimates for the SVV-stabilized scheme}
\label{sec:analysis}

\subsection{Main results}
\label{sec_mainresults}
The error analysis in \cite{HuangShen2025} is established only for the case $\nu=1$. We extend it to arbitrary viscosities $\nu>0$ and to the
SVV-stabilized scheme in Theorem~\ref{thm:main}. 

As in \cite{HuangShen2025}, the analysis is carried out for the time-discrete scheme with the
spatial variable kept continuous, under Assumption~\ref{ass:smooth}. In this setting $Q_N$
denotes a spectral multiplier that is diagonal in an $L^2$-orthonormal basis of eigenfunctions
of the Dirichlet Laplacian on $\Om$, with symbol in $[0,1]$, so that $Q_N$ and $\sqrt{Q_N}$ are
self-adjoint, bounded by one in $\bm L^2(\Om)$, and commute with $\Delta$ on
$\bm H^2(\Om)\cap\bm H^1_0(\Om)$. These properties, the continuous counterparts of
\textup{(P3)}, \textup{(P4)} and \textup{(P6)}, are the only ones used in the proofs. The
operator of Definition~\ref{def:svv} is the Legendre--Galerkin realization of such a multiplier
on $\Om=(-1,1)^2$, in which the eigenfunctions of the Laplacian are replaced by the discrete
eigenfunctions $\Psi_{ij}$. This fully discrete realization is the one used in the computations
of Section~\ref{sec:numerics}. Its spatial discretization error is not analyzed here.

Define $\be^i=\bu^i-\bu(t^i)$ and $e_p^i=p^i-p(t^i)$.
\begin{theorem}[Stability and error estimates for the SVV-stabilized scheme]\label{thm:main}
Suppose Assumption~\ref{ass:smooth} holds and
$\norm{\nabla\be^i}^2+\dt\norm{\Delta\be^i}^2\le C\dt^{2k}$ for $i=0,\dots,k-1$.
Set $C_0:=\max_{0\le t\le T}\norm{\nabla\bu(t)}+1$, $M_2:=\sup_t\norm{\bu(t)}_2$,
$M_f:=\sup_t\norm{\bff(t)}$, and
$\mathcal R_k:=\int_0^T\bigl(\norm{\partial_t^{k+1}\bu}^2+\norm{\partial_t^k\bu}_2^2
+\norm{\partial_t^k p}_1^2+\norm{\partial_t^k(\bu\!\cdot\!\nabla\bu)}^2\bigr)dt$.
Then, for any $\eps_N\ge0$ and $\dt$ such that the smallness condition
\begin{equation}\label{eq:smallness}
\bigl(C_R\,\dt^{2k}+\mathcal B\,D_{\rm svv}\bigr)\,\mathcal G\le 1
\end{equation}
holds, with $\mathcal B$, $C_R$, $\mathcal G$ and $D_{\rm svv}$ defined in \eqref{eq:svv-Ca2} and below it,
the SVV-stabilized scheme
\eqref{eq:SVV-vel}, \eqref{eq:HS-pres} satisfies the following.

\emph{Part 1 (stability).} 
\begin{eqnarray}\label{eq:H1-apriori}
\sup_{0\le j\le n+1}\norm{\nabla\bu^j} &\le& C_0,\\
\label{eq:svv-stab}
\sup_{m\le n}\norm{\nabla\bu^{m+1}}^2
&+& \nu\dt\sum_{i=0}^{n}\norm{\Delta\bu^{i+1}}^2
+\eps_N\dt\sum_{i=0}^{n}\norm{\sqrt{Q_N}\Delta\bu^{i+1}}^2
\le C_{a1},
\end{eqnarray}
where 
$C_{a1}=\frac{CC_0^6T}{\nu^3}+\frac{CTM_f^2}{\nu}+CC_0^2(1+\nu T)$.

\emph{Part 2 (error).}
\begin{equation}\label{eq:svv-err}
\begin{aligned}
&\norm{\nabla\be^{n+1}}^2+\nu\dt\sum_{i=0}^{n+1}\norm{\Delta\be^i}^2
+\frac{\dt}{\nu}\sum_{i=0}^{n+1}\norm{\nabla e_p^i}^2
+\frac{\eta_k\eps_N\dt}{2}\sum_{i=0}^{n+1}\norm{\sqrt{Q_N}\Delta\Ck(\be^i)}^2\\
&\qquad\le \bigl(C_R\,\dt^{2k}+\mathcal B\,D_{\rm svv}\bigr)\,\mathcal G,
\end{aligned}
\end{equation}
where the amplification $\mathcal B$, the data factor $C_R$, and the Gronwall factor $\mathcal G$ are
\begin{equation}\label{eq:svv-Ca2}
\begin{aligned}
\mathcal B&:=1+\frac{C_0^6T}{\nu^5}+\frac{TM_f^2}{\nu^3}+\frac{T(C_0^2+M_2^2)}{\nu}+\frac{C_0^2}{\nu^2}+\nu T,
\qquad C_R:=\frac{C\bigl((1+\nu^2)\mathcal R_k+\nu\bigr)}{\nu}\,\mathcal B,\\
\mathcal G&:=\exp\!\Bigl(\frac{CC_0^6T}{\nu^5}+\frac{CTM_f^2}{\nu^3}+\frac{CTM_2^2}{\nu}+\frac{CC_0^2T}{\nu}+\frac{CC_0^2}{\nu^2}\Bigr),
\end{aligned}
\end{equation}
and $D_{\rm svv}\le c\,\eps_N T\sup_{0\le t\le T}\norm{\sqrt{Q_N}\Delta\bu(t)}^2$
with $c$ depending only on $A_k,B_k$, and $C$ depends only on $\Om$ and $k$.
In particular $\eps_N=0$ gives $D_{\rm svv}=0$, removes the two SVV terms, and
recovers the $\nu$-dependent version of \cite[Thm.~4.1]{HuangShen2025}.
\end{theorem}

\begin{remark}\label{rem:smallness}
The condition \eqref{eq:smallness} is what closes the induction on
$\norm{\nabla\bu^n}\le C_0$ in the proof. For $\eps_N=0$ it reduces to $\dt$ sufficiently
small, as in \cite{HuangShen2025}. For $\eps_N>0$ the term $\mathcal B\,D_{\rm svv}\,\mathcal G$
does not decrease with $\dt$, so the condition also restricts the stabilization: since
$D_{\rm svv}\le c\,\eps_N T M_2^2$ by \textup{(P4)} and $\eps_N=C_{\rm svv}/M$, it holds once $N$
is sufficiently large. Because $\mathcal G$ grows like $\exp(C\nu^{-5})$, the required $N$ is
very large for small $\nu$. Like the constants themselves, this requirement reflects the
present technique of proof rather than the behaviour observed in Section~\ref{sec:numerics}.
\end{remark}

\begin{remark}\label{rem:scope}
It is worth saying precisely where these negative powers of $\nu$ come from,
since SVV does not remove them. In the stability estimate the second-order terms
are absorbed on all modes only by the viscous coercivity
$\eta_k\nu\dt\norm{\Delta\Ck(\bu^{i+1})}^2$, whose coefficient is $\nu$. The SVV
coercivity is available too, but only on the high modes, so it cannot carry the
low-mode part. The trilinear convection term \eqref{eq:S1conv} is therefore split
by Young's inequality with a weight tied to $\nu$, which leaves the free term
$C\dt\,\nu^{-3}\norm{\nabla\Ck(\bu^i)}^6$. This is the origin of the $\nu^{-3}$ in
$C_{a1}$. The pressure step \eqref{eq:S1press}--\eqref{eq:S1f} contributes
similarly. The hydrodynamic pressure gradient
$\norm{\Ck(\bff^i-\bu^i\!\cdot\!\nabla\bu^i)}$ stays $O(1)$ as $\nu\to0$, yet it is
absorbed by the same $\nu$-viscous coercivity and so is divided by $\nu$, and
through $\norm{\bu^j\!\cdot\!\nabla\bu^j}^2\le CC_0^3\norm{\Delta\bu^j}$ it feeds a
further $\nu^{-3}$. In the error estimate the Gronwall driver contains
$\norm{\Delta\bu^i}^2$, and $\dt\sum_i\norm{\Delta\bu^i}^2\le C_{a1}/\nu=O(\nu^{-4})$
by \eqref{eq:Ca1}. The exponent of the Gronwall factor is then
$\tfrac{C}{\nu}\cdot\tfrac{C_{a1}}{\nu}=O(\nu^{-5})$, which is the $\nu^{-5}$ in
$\mathcal G$. SVV cannot break this chain, because its kernel vanishes on the low
modes that the convection and pressure absorption must also cover, so it adds no
coercivity there and the $\nu$-tied Young weights stay unavoidable in this argument.
We do not claim that the resulting dependence on $\nu$ is sharp. A different technique of
proof might give a weaker dependence, and the bound by itself does not imply that the scheme
fails as $\nu\to0$. The breakdown of the unstabilized scheme reported in
Section~\ref{sec:numerics} is a numerical observation, separate from this bound.
\end{remark}

\subsection{Some lemmas}
\label{sec:somelemmas}

We collect here the notation and the lemmas used in the
proof.  

\begin{lemma}[Stokes-pressure estimate, \cite{LiuLiuPego2007}]
\label{lem:LLP-paper}
Fix a constant $\eps\in(0,\tfrac12)$. There exists $C>0$, depending only on
$\Om$ and $\eps$, such that for every $\bu\in\bm H^2(\Om)\cap\bm H^1_0(\Om)$,
\[
\norm{\nabla p_s(\bu)}^2
\le \bigl(\tfrac12+\eps\bigr)\norm{\Delta\bu}^2 + C\,\norm{\nabla\bu}^2,
\]
where $p_s(\bu)\in H^1/\R$ is the Stokes pressure defined by
$(\nabla p_s(\bu), \nabla q) = (\Delta\bu-\nabla\diver\bu,\,\nabla q)
=-(\nabla\times\nabla\times\bu,\,\nabla q)$ for all $q\in H^1(\Om)$.
\end{lemma}

\begin{lemma}[\cite{Temam1984}]
\label{lem:tri-paper}
For $d=2,3$ and $\bu,\bv,\bw\in\bm H^1_0(\Om)\cap\bm H^2(\Om)$,
\[
|(\bu\cdot\nabla\bv,\bw)|
\;\le\; c\,\norm{\bu}_1\,\norm{\bv}_1^{1/2}\norm{\bv}_2^{1/2}\,\norm{\bw},
\qquad
\norm{\bu\cdot\nabla\bu}^2\le c\,\norm{\bu}_1^3\norm{\bu}_2,
\]
where $c$ depends only on $\Om$.
\end{lemma}

The G-stability of the shifted BDF time derivative is obtained by applying
the Dahlquist G-stability theorem \citep{Dahlquist1978} to the multiplier
conditions of \cite[Lem.~3.1]{HuangShen2025}, and is recorded as
\cite[eq.~(3.37)]{HuangShen2025}.

\begin{lemma}[{G-stability of shifted BDF time derivative, \cite{HuangShen2025}}]
\label{lem:Gstab}
There exists a symmetric and positive definite matrix $G_k=(g_{i,j})\in\R^{k\times k}$, depending
only on $k$ and $\beta_k$, with smallest eigenvalue $\lambda_k^g>0$, such
that for any sequence $\{\phi^j\}\subset\bm H^2(\Om)\cap\bm H^1_0(\Om)$,
\[
\bigl(\Ak(\phi^{n+1}),-\Delta\Ck(\phi^{n+1})\bigr)
\ge \norm{\nabla\phi^{n+1}}_{G_k}^2 - \norm{\nabla\phi^{n}}_{G_k}^2,
\]
where
$\norm{\nabla\phi^{n+1}}_{G_k}^2:=\sum_{l,j=1}^{k}g_{l,j}\,\bigl(\nabla\phi^{n+1+l-k},\nabla\phi^{n+1+j-k}\bigr)$
denotes the $G_k$-weighted quadratic form of the gradient history; it satisfies
$\norm{\nabla\phi^{n+1}}_{G_k}^2\ge\lambda_k^g\,\norm{\nabla\phi^{n+1}}^2$.
\end{lemma}

The coercivity of the viscous term is obtained from the splitting $\Bk=\eta_k\Ck+\Dk+\Fk$, where the inner product involving $\Dk$ is treated by feeding the multiplier conditions of \cite[Lem.~3.2]{HuangShen2025} into the Dahlquist G-stability theorem \citep{Dahlquist1978}, while the one involving $\Fk$ is handled in Appendix~A of \cite{HuangShen2025}.
\begin{lemma}[{$\Bk$--$\Ck$ coercivity, \cite{HuangShen2025}}]
\label{lem:BkCk}
With $\eta_k=0.71$, there exist a symmetric and positive definite matrix $H_k=(h_{i,j})\in\R^{(k-1)\times(k-1)}$, a positive semidefinite quadratic form $U_k$, and 
$\kappa_k>0$ such that
\begin{align*}
\bigl(\Delta\Bk(\phi^{n+1}),\Delta\Ck(\phi^{n+1})\bigr)
&\ge \;\eta_k\norm{\Delta\Ck(\phi^{n+1})}^2 + \kappa_k\norm{\Delta\phi^{n+1}}^2 
+ \bigl(\mathcal H_{n+1}(\phi)-\mathcal H_n(\phi)\bigr)\\ 
&+ \bigl(\mathcal U_{n+1}(\phi)-\mathcal U_n(\phi)\bigr),
\end{align*}
where $\mathcal H_n(\phi):=\sum_{i,j=1}^{k-1} h_{i,j}\bigl(\Delta\phi^{n+1+i-k},\Delta\phi^{n+1+j-k}\bigr)$ and $\mathcal U_n(\phi):=U_k(\Delta\phi^n,\dots,\Delta\phi^{n+2-k})$.
\end{lemma}

\begin{lemma}[SVV coercivity]\label{lem:SVVcoerc}
For any sequence $\{\phi^j\}$ in $V_N$ (or, in the space-continuous setting of Section~\ref{sec_mainresults}, in
$\bm H^2(\Om)\cap\bm H^1_0(\Om)$),
\begin{align}
\bigl(\sqrt{Q_N}\Delta\Bk(\phi^{n+1}),\sqrt{Q_N}\Delta\Ck(\phi^{n+1})\bigr)
&\ge\; \eta_k\norm{\sqrt{Q_N}\Delta\Ck(\phi^{n+1})}^2
+ \kappa_k\norm{\sqrt{Q_N}\Delta\phi^{n+1}}^2 \notag\\
&+ \bigl(\mathcal H_{n+1}^Q(\phi)-\mathcal H_n^Q(\phi)\bigr)
+ \bigl(\mathcal U_{n+1}^Q(\phi)-\mathcal U_n^Q(\phi)\bigr),
\label{eq:SVVlem}
\end{align}
where, with the matrix $H_k$ and form $U_k$ from Lemma\,\ref{lem:BkCk},
\begin{align*}
\mathcal H_n^Q(\phi)&=\sum_{i,j=1}^{k-1}h_{i,j}\bigl(\sqrt{Q_N}\Delta\phi^{n+1+i-k},\,\sqrt{Q_N}\Delta\phi^{n+1+j-k}\bigr),\\
\mathcal U_n^Q(\phi)&=U_k\bigl(\sqrt{Q_N}\Delta\phi^n,\dots,\sqrt{Q_N}\Delta\phi^{n+2-k}\bigr).
\end{align*}
\end{lemma}

\begin{proof}[Proof of Lemma~\ref{lem:SVVcoerc}]
Recall from Section~\ref{sec:scheme} the eigen-basis
$\{\Psi_{ij}\}_{0\le i,j\le M-1}$ ($M=N-1$),
$\Psi_{ij}(x,y)=\psi_i(x)\psi_j(y)$, which $L^2$-orthonormalises $V_N$. For
$v\in V_N$ write $\widehat v_{ij}:=(v,\Psi_{ij})$ for the coordinate of $v$
along $\Psi_{ij}$, so that $v=\sum_{i,j}\widehat v_{ij}\,\Psi_{ij}$. By
Section~\ref{sec:scheme}, $Q_N$ is the spectral multiplier
\[
Q_N\Psi_{ij}=\widehat Q_{ij}\,\Psi_{ij},\qquad
\widehat Q_{ij}:=\frac{\widehat Q_i\,\mu_i+\widehat Q_j\,\mu_j}{\mu_i+\mu_j}
\in[0,1],
\]
cf.~\eqref{eq:Qtilde}; note that only $\widehat Q_{ij}\in[0,1]$ and its
diagonality are used below, so the argument is independent of the particular
two-dimensional kernel. It commutes with $\Delta$ and
$\sqrt{Q_N}\Psi_{ij}=\sqrt{\widehat Q_{ij}}\,\Psi_{ij}$. Expanding
$\Delta\Bk(\phi^{n+1})$ and $\Delta\Ck(\phi^{n+1})$ in $\{\Psi_{ij}\}$ and using
orthonormality,
\[
(\sqrt{Q_N}\Delta\Bk(\phi^{n+1}),\sqrt{Q_N}\Delta\Ck(\phi^{n+1}))
= \sum_{i,j=0}^{M-1} \widehat Q_{ij}\,
\bigl(\Delta\Bk(\phi^{n+1}),\Psi_{ij}\bigr)
\bigl(\Delta\Ck(\phi^{n+1}),\Psi_{ij}\bigr).
\]
The proof of Lemma~\ref{lem:BkCk} in \cite{HuangShen2025} uses only the
algebraic structure of $\Bk$ and $\Ck$, so the inequality holds mode by mode in
the diagonalization basis. Multiplying the mode-$(i,j)$ version of
Lemma~\ref{lem:BkCk} by the non-negative $\widehat Q_{ij}$ and summing over
$i,j$ yields \eqref{eq:SVVlem}. In the space-continuous setting of Section~\ref{sec_mainresults} the same argument
applies with the eigenfunctions of the Dirichlet Laplacian in place of $\Psi_{ij}$.
\end{proof}

\begin{lemma}\label{lem:Ck-uniform}
Denote $\Ck(\bu^j) = \sum_{q=0}^{k-1}c_{k,q}\bu^{j-k+1+q}$, $c_k=\max_{0\le q\le k-1}|c_{k,q}|$ and assume \\
$\sup_{j\le n}\norm{\nabla\bu^j} \le C_0$. Then
$\sup_{j\le n}\norm{\nabla\Ck(\bu^j)}\le k\,c_k\,C_0$.
\end{lemma}

\begin{proof}
The triangle inequality implies $\norm{\nabla\Ck(\bu^j)}\le \sum_q|c_{k,q}|\norm{\nabla\bu^{j-k+1+q}}
\le k c_k C_0$.
\end{proof}

\subsection{Proof of Theorem~\ref{thm:main}}
\label{sec_proof_thm_main}

We write $C$ for a generic positive constant depending only on $\Om$ and $k$,
and in particular not on $\nu$, $\dt$, or $n$. We use two tools from \HS. First,
the trilinear estimates for $b(\bu,\bv,\bw):=(\bu\!\cdot\!\nabla\bv,\bw)$,
\begin{align}
b(\bu,\bv,\bw)&\le C\,\norm{\bu}_1\,\norm{\bv}_2\,\norm{\bw},\qquad
b(\bu,\bv,\bw)\le C\,\norm{\bu}_2\,\norm{\bv}_1\,\norm{\bw}.\label{eq:tri2}
\end{align}
Second, the discrete Gronwall lemma (\HS, Lemma~2.1): if
$a_n,b_n,c_n,d_n\ge0$ and $C_G,\tau>0$ satisfy, for all $m\ge1$,
\[
a_m+\tau\sum_{n=1}^{m}b_n\le\tau\sum_{n=0}^{m-1}a_n d_n+\tau\sum_{n=0}^{m-1}c_n+C_G,
\]
then
\[
a_m+\tau\sum_{n=1}^{m}b_n\le\exp\Bigl(\tau\sum_{n=0}^{m-1}d_n\Bigr)\Bigl(\tau\sum_{n=0}^{m-1}c_n+C_G\Bigr).
\]

\begin{proof}[Proof of Theorem~\ref{thm:main}]
We prove by induction on $n$. Assuming
\begin{equation}\label{eq:IH}
\norm{\nabla\bu^i}\le C_0,\qquad i=0,\dots,n
\end{equation}
(which holds for $i=0$), we shall prove $\norm{\nabla\bu^{n+1}}\le C_0$.

\bigskip
\noindent\textbf{Step 1.\quad Stability: proof of Part~1.}

\medskip
We take the inner product of \eqref{eq:SVV-vel}, written at level $i+1\le n$, with
$-\dt\,\Delta \Ck(\bu^{i+1})$, and estimate the resulting terms one at a time.
The time-derivative, convection, pressure, and forcing terms are exactly as in
the bare scheme; the combined operator $-(\nu+\eps_N Q_N)\Delta\Bk$ contributes,
beyond the viscous pairing, the SVV pairing treated in \eqref{eq:S1svv}.
For the time-derivative term, Lemma~\ref{lem:Gstab} gives
\begin{equation}\label{eq:S1time}
\bigl(\Ak(\bu^{i+1}),-\Delta \Ck(\bu^{i+1})\bigr)
\ge \norm{\nabla\bu^{i+1}}_{G_k}^2-\norm{\nabla\bu^{i}}_{G_k}^2.
\end{equation}
For the viscous term, Lemma~\ref{lem:BkCk} (with $\phi=\bu$) yields
\begin{equation}\label{eq:S1visc}
\begin{aligned}
\nu\dt\bigl(\Delta \Bk(\bu^{i+1}),\Delta \Ck(\bu^{i+1})\bigr)
&\ge \eta_k\nu\dt\norm{\Delta \Ck(\bu^{i+1})}^2+\nu\kappa_k\dt\norm{\Delta\bu^{i+1}}^2\\
&\quad+\nu\dt\bigl(\mathcal H_{i+1}(\bu)-\mathcal H_i(\bu)\bigr)
+\nu\dt\bigl(\mathcal U_{i+1}(\bu)-\mathcal U_i(\bu)\bigr).
\end{aligned}
\end{equation}
For the SVV term, Lemma~\ref{lem:SVVcoerc} gives
\begin{equation}\label{eq:S1svv}
\begin{aligned}
\eps_N\dt\bigl(Q_N\Delta \Bk(\bu^{i+1}),\Delta \Ck(\bu^{i+1})\bigr)
&\ge \eps_N\eta_k\dt\norm{\sqrt{Q_N}\Delta \Ck(\bu^{i+1})}^2
+\eps_N\kappa_k\dt\norm{\sqrt{Q_N}\Delta\bu^{i+1}}^2\\
&\quad+\eps_N\dt\bigl(\mathcal H^Q_{i+1}(\bu)-\mathcal H^Q_i(\bu)\bigr)
+\eps_N\dt\bigl(\mathcal U^Q_{i+1}(\bu)-\mathcal U^Q_i(\bu)\bigr),
\end{aligned}
\end{equation}
a nonnegative coercive contribution to the left-hand side.
For the convection term, we use Lemma\,\ref{lem:tri-paper} with
$\bu=\bv=\Ck(\bu^i)$, the bound $\norm{\Ck(\bu^i)}_2\le C\norm{\Delta \Ck(\bu^i)}$,
and Young's inequality with weight $\delta_1 \nu$, to obtain
\begin{equation}\label{eq:S1conv}
\dt\bigl|\bigl(\Ck(\bu^i)\!\cdot\!\nabla \Ck(\bu^i),\Delta \Ck(\bu^{i+1})\bigr)\bigr|
\le \frac{C\dt}{\nu^3}\norm{\nabla \Ck(\bu^i)}^6
+\delta_1\nu\dt\Bigl(\norm{\Delta \Ck(\bu^i)}^2+\norm{\Delta \Ck(\bu^{i+1})}^2\Bigr).
\end{equation}
For $\bv\in\mathbf H^2(\Om)$,
the Stokes pressure $p_s(\bv)$ satisfies the variational identity
(Theorem~1 of \cite{LiuLiuPego2007})
\begin{equation}\label{eq:LLPid}
(\nabla p_s(\bv),\nabla q)=-(\nabla\times\nabla\times\bv,\,\nabla q),\qquad\forall q\in H^1(\Om).
\end{equation}
Inserting \eqref{eq:LLPid} into the pressure step \eqref{eq:HS-pres} written at
level $i$ expresses the discrete pressure as
\begin{equation}\label{eq:S1pid}
(\nabla p^i,\nabla q)=\bigl(\bff^i-\bu^i\!\cdot\!\nabla\bu^i,\,\nabla q\bigr)
+\nu\,(\nabla p_s(\bu^i),\nabla q),\qquad\forall q\in H^1(\Om).
\end{equation}
Applying the explicit extrapolation $\Ck$ to \eqref{eq:S1pid} and then choosing
$q=\Ck(p^i)$ yields
\begin{equation}\label{eq:S1prep}
\norm{\nabla \Ck(p^i)}\le \norm{\Ck(\bff^i-\bu^i\!\cdot\!\nabla\bu^i)}
+\nu\,\norm{\nabla p_s(\Ck(\bu^i))}.
\end{equation}
We compute the inner product between $\nabla \Ck(p^i)$ and $-\Delta \Ck(\bu^{i+1})$ via Cauchy-Schwarz inequality and \eqref{eq:S1prep}. The forcing--convection
part is treated by Young's inequality with weight $\delta_\alpha\nu$. For the
Stokes-pressure part, Young's inequality with weight $\delta_\gamma\nu$ gives
\[
\nu \|\nabla p_s(\Ck(\bu^i))\| \,
\|\Delta \Ck(\bu^{i+1})\|
\le \frac{\nu}{4\delta_\gamma}\norm{\nabla p_s(\Ck(\bu^i))}^2
+\delta_\gamma\nu\norm{\Delta \Ck(\bu^{i+1})}^2,
\]
into which the Stokes-pressure estimate of Lemma~\ref{lem:LLP-paper} (with $\bu=\Ck(\bu^i)$) is
inserted, i.e., $\norm{\nabla p_s(\Ck(\bu^i))}^2 \le(\tfrac12+\eps)\norm{\Delta \Ck(\bu^i)}^2
+C\norm{\nabla \Ck(\bu^i)}^2$. Altogether, we obtain
\begin{equation}\label{eq:S1press}
\begin{aligned}
\dt\bigl|\bigl(\nabla \Ck(p^i),\Delta \Ck(\bu^{i+1})\bigr)\bigr|
&\le \frac{C\dt}{\nu}\norm{\Ck(\bff^i-\bu^i\!\cdot\!\nabla\bu^i)}^2
+\frac{\nu(1+2\eps)}{8\delta_\gamma}\dt\norm{\Delta \Ck(\bu^i)}^2\\
&\quad+\frac{C\nu\dt}{\delta_\gamma}\norm{\nabla \Ck(\bu^i)}^2
+(\delta_\alpha+\delta_\gamma)\nu\dt\norm{\Delta \Ck(\bu^{i+1})}^2.
\end{aligned}
\end{equation}
The first term on the right involves $\Ck$, a bounded combination of the
levels $j=i+1-k,\dots,i$, so it is controlled level by level. For each such $j$,
$\norm{\bu^j\!\cdot\!\nabla\bu^j}^2\le C\norm{\bu^j}_1^3\norm{\bu^j}_2
\le C C_0^3\norm{\Delta\bu^j}$ by \eqref{eq:IH} for $j\le i\le n$, and
$\norm{\Ck(\bff^i)}\le CM_f$. Hence
\begin{equation}\label{eq:S1f}
\frac{C\dt}{\nu}\norm{\Ck(\bff^i-\bu^i\!\cdot\!\nabla\bu^i)}^2
\le \frac{C\dt}{\nu}\,M_f^2+\frac{C C_0^3\dt}{\nu}\sum_{j=i+1-k}^{i}\norm{\Delta\bu^j}.
\end{equation}
Finally, the right-hand side of \eqref{eq:HS-vel} contributes
\begin{equation}\label{eq:S1rhs}
\dt\bigl(\bff^{\,i+\beta_k},-\Delta \Ck(\bu^{i+1})\bigr)
\le \frac{C\dt}{\nu}\,M_f^2+\delta_f\,\nu\dt\norm{\Delta \Ck(\bu^{i+1})}^2.
\end{equation}

Summing \eqref{eq:S1time}, \eqref{eq:S1visc}, \eqref{eq:S1svv}, \eqref{eq:S1conv}, \eqref{eq:S1press}, and \eqref{eq:S1rhs} for $i=k-1,\dots,m$, the time, viscous, and SVV contributions combine into the total energy
\[
\mathcal E^{\,i+1}:=\norm{\nabla\bu^{i+1}}_{G_k}^2
+\nu\dt\,\mathcal H_{i+1}(\bu)+\nu\dt\,\mathcal U_{i+1}(\bu)
+\eps_N\dt\,\mathcal H^Q_{i+1}(\bu)+\eps_N\dt\,\mathcal U^Q_{i+1}(\bu),
\]
whose increments telescope to $\mathcal E^{\,m+1}-\mathcal E^{\,k-1}$.  Using $\sum_i\dt\le T$ on the forcing, the summed bounds read
\begin{equation}\label{eq:S1summed}
\begin{aligned}
&\mathcal E^{m+1} - \mathcal E^{k-1} +
\Bigl[\eta_k - \tfrac{1+2\eps}{8\delta_\gamma}-(2\delta_1+\delta_\alpha+\delta_\gamma+\delta_f)\Bigr]
\nu\dt\!\sum_{i=k-1}^{m}\!\norm{\Delta\Ck(\bu^{i+1})}^2
\\
&
+\kappa_k\,\nu\dt\!\sum_{i=k-1}^{m}\!\norm{\Delta\bu^{i+1}}^2
+\eps_N\eta_k\dt\!\sum_{i=k-1}^{m}\!\norm{\sqrt{Q_N}\Delta\Ck(\bu^{i+1})}^2
+\eps_N\kappa_k\dt\!\sum_{i=k-1}^{m}\!\norm{\sqrt{Q_N}\Delta\bu^{i+1}}^2\\
\le & 
\frac{C\dt}{\nu^3}\sum_{i=k-1}^{m}\norm{\nabla\Ck(\bu^i)}^6
+\frac{C\nu\dt}{\delta_\gamma}\sum_{i=k-1}^{m}\norm{\nabla\Ck(\bu^i)}^2
+\frac{CT M_f^2}{\nu}
+\frac{CC_0^3}{\nu}\,\dt\!\sum_{i=k-1}^{m}\sum_{j=i+1-k}^{i}\norm{\Delta\bu^j}.
\end{aligned}
\end{equation}
 As
$\mathcal H_{m+1}(\bu),\mathcal U_{m+1}(\bu),\mathcal H^Q_{m+1}(\bu),\mathcal U^Q_{m+1}(\bu)\ge0$, the final level satisfies
$\mathcal E^{\,m+1}\ge\lambda_k^g\norm{\nabla\bu^{m+1}}^2$, so these endpoint
terms are discarded, while $\mathcal E^{\,k-1}$ depends only on the first
$k-1$ steps. By the hypotheses on the startup values and the smoothness of the
exact solution, $\norm{\nabla\bu^i}\le C_0$ for $i\le k-1$, so the startup contribution $M_0$, which collects $\mathcal E^{\,k-1}$ and the terms at the levels $i\le k-1$ produced by the index shifts in the sums above, satisfies $M_0\le CC_0^2$. We choose $\delta_1,\delta_\alpha,\delta_f,\eps$ small and $\delta_\gamma$ suitable so
that
\begin{equation}\label{eq:coerc}
\eta_k-\frac{1+2\eps}{8\delta_\gamma}-(2\delta_1+\delta_\alpha+\delta_\gamma+\delta_f)\ge 0
\end{equation}
(corresponding to condition (4.21) of \HS). This is feasible because
$\frac{1}{8\delta_\gamma}+\delta_\gamma$ has minimum $\frac{1}{\sqrt{2}}\approx0.7071$,
attained at $\delta_\gamma=\frac{1}{2\sqrt{2}}$, which lies just below $\eta_k=0.71$.
The margin $\eta_k-\frac{1}{\sqrt2}\approx 2.9\times10^{-3}$ is small and must also
absorb $2\delta_1+\delta_\alpha+\delta_f$ and the excess
$\frac{2\eps}{8\delta_\gamma}$, so these parameters are taken correspondingly small.
The resulting nonnegative term $\nu\dt\norm{\Delta \Ck(\bu^j)}^2$ is then dropped.
 Using \eqref{eq:IH},
$\norm{\nabla \Ck(\bu^i)}^6\le CC_0^6$ and $\norm{\nabla \Ck(\bu^i)}^2\le CC_0^2$
(the latter, with $\delta_\gamma$ bounded away from $0$ and $m\dt\le T$,
producing the $CC_0^2\nu T$ term), and absorbing the first-power term
(the stencil sum in \eqref{eq:S1f} re-indexes to a single sum, up to the
factor $k$ absorbed in $C$) by
\[
\frac{CC_0^3}{\nu}\,\dt\sum_{i=k-1}^{m}\norm{\Delta\bu^i}
\le \frac{\kappa_k}{2}\,\nu\dt\sum_{i=k-1}^{m}\norm{\Delta\bu^i}^2
+\frac{CC_0^6T}{\nu^3},
\]
we obtain the $\nu$-explicit form of (4.22)--(4.23) of \HS, which is Part~1,
\eqref{eq:svv-stab}:
\begin{equation}\label{eq:Ca1}
\begin{aligned}
&\sup_{m\le n}\norm{\nabla\bu^{m+1}}^2+\nu\dt\sum_{i=0}^{n}\norm{\Delta\bu^{i+1}}^2
+\eps_N\dt\sum_{i=0}^{n}\norm{\sqrt{Q_N}\Delta\bu^{i+1}}^2
\\
\le& \frac{CC_0^6T}{\nu^3}+\frac{CTM_f^2}{\nu}+CC_0^2(1+\nu T)=C_{a1}.
\end{aligned}
\end{equation}
In particular $\dt\sum_i\norm{\Delta\bu^i}^2\le C_{a1}/\nu=O(\nu^{-4})$.

\bigskip
\noindent\textbf{Step 2.\quad Error estimate for $\norm{\nabla\be^{n+1}}$.}

\medskip
Subtracting from \eqref{eq:HS-vel} the identity satisfied by the exact
solution yields the error equation (\HS, (4.24)): with
$\be^i:=\bu^i-\bu(t^i)$ and $e_p^i:=p^i-p(t^i)$,
\begin{equation}\label{eq:err}
\frac{\Ak(\be^{i+1})}{\dt}-(\nu+\eps_N Q_N)\Delta \Bk(\be^{i+1})+\nabla \Ck(e_p^i)
+\Ck(\bu^i)\!\cdot\!\nabla \Ck(\bu^i)-\Ck(\bu(t^i))\!\cdot\!\nabla \Ck(\bu(t^i))
=P_k^i+Q_k^i+R_k^i+S_k^i+T_k^i,
\end{equation}
where $P_k^i,Q_k^i,R_k^i,S_k^i$ are the pressure, viscous, time, and convection
truncation errors of \HS, (4.25)--(4.28), and
$T_k^i:=\eps_N Q_N\Delta \Bk(\bu(t^{i+1}))$ is the SVV truncation error, which has
no continuous counterpart and satisfies
$\norm{T_k^i}\le\eps_N\norm{\Delta\Bk(\bu(t^{i+1}))}\le C\eps_N\sup_t\norm{\Delta\bu(t)}$.
Taking the $L^2$ inner product of
\eqref{eq:err} with $-\dt\,\Delta \Ck(\be^{i+1})$, the time-derivative and
viscous terms are estimated exactly as in \eqref{eq:S1time}--\eqref{eq:S1visc},
with $\bu$ replaced by $\be$. The SVV operator contributes the coercive term,
bounded below by Lemma~\ref{lem:SVVcoerc} (the analogue of \eqref{eq:S1svv} with
$\bu$ replaced by $\be$),
\begin{equation}\label{eq:S2svv}
\begin{aligned}
\eps_N\dt\bigl(Q_N\Delta\Bk(\be^{i+1}),\Delta\Ck(\be^{i+1})\bigr)
&\ge \eps_N\eta_k\dt\,\norm{\sqrt{Q_N}\Delta\Ck(\be^{i+1})}^2
+\eps_N\dt\bigl(\mathcal H^Q_{i+1}(\be)-\mathcal H^Q_i(\be)\bigr)\\
&\quad+\eps_N\dt\bigl(\mathcal U^Q_{i+1}(\be)-\mathcal U^Q_i(\be)\bigr),
\end{aligned}
\end{equation}
with $\mathcal H^Q(\be),\mathcal U^Q(\be)\ge0$. Since $\sqrt{Q_N}$ is
self-adjoint, Young's inequality gives
\begin{equation}\label{eq:S2Tk}
\begin{aligned}
\dt\,\bigl|(T_k^i,\Delta\Ck(\be^{i+1}))\bigr|
&=\eps_N\dt\,\bigl|(\sqrt{Q_N}\Delta\Bk(\bu(t^{i+1})),\,\sqrt{Q_N}\Delta\Ck(\be^{i+1}))\bigr|\\
&\le \frac{\eps_N\dt}{2\eta_k}\,\norm{\sqrt{Q_N}\Delta\Bk(\bu(t^{i+1}))}^2
+\frac{\eta_k}{2}\,\eps_N\dt\,\norm{\sqrt{Q_N}\Delta\Ck(\be^{i+1})}^2.
\end{aligned}
\end{equation}
The last term is absorbed into the coercivity of \eqref{eq:S2svv}, leaving
$\tfrac{\eta_k}{2}\eps_N\dt\norm{\sqrt{Q_N}\Delta\Ck(\be^{i+1})}^2$ on the left,
and summation with
$\norm{\sqrt{Q_N}\Delta\Bk(\bu(t^{i+1}))}\le C\sup_t\norm{\sqrt{Q_N}\Delta\bu(t)}$
bounds the new free term by the data constant
\begin{equation}\label{eq:Csvv}
D_{\rm svv}:=\frac{\eps_N\dt}{2\eta_k}\sum_{i=k-1}^{n}\norm{\sqrt{Q_N}\Delta\Bk(\bu(t^{i+1}))}^2
\le c\,\eps_N T\sup_t\norm{\sqrt{Q_N}\Delta\bu(t)}^2,
\end{equation}
added to the data of the master estimate \eqref{eq:master}.

The convective term is split into
\begin{equation}\label{eq:split}
\Ck(\bu^i)\!\cdot\!\nabla \Ck(\bu^i)-\Ck(\bu(t^i))\!\cdot\!\nabla \Ck(\bu(t^i))
=\Ck(\be^i)\!\cdot\!\nabla \Ck(\bu^i)+\Ck(\bu(t^i))\!\cdot\!\nabla \Ck(\be^i).
\end{equation}
Each piece is bounded with one inequality in \eqref{eq:tri2}. With
$\norm{\Ck(\bu^i)}_2\le C\norm{\Delta \Ck(\bu^i)}$ and Young's inequality with
weight $\delta_2\nu$,
\begin{equation}\label{eq:S2conv}
\begin{aligned}
\dt\bigl|\bigl(\Ck(\be^i)\!\cdot\!\nabla \Ck(\bu^i)&+\Ck(\bu(t^i))\!\cdot\!\nabla \Ck(\be^i),\,
-\Delta \Ck(\be^{i+1})\bigr)\bigr|\\
&\le \frac{C\dt}{\nu}\norm{\nabla \Ck(\be^i)}^2\Bigl(\norm{\Delta \Ck(\bu^i)}^2+\norm{\Ck(\bu(t^i))}_2^2\Bigr)
+\delta_2\nu\dt\norm{\Delta \Ck(\be^{i+1})}^2 .
\end{aligned}
\end{equation}

The pressure error is treated through the Stokes pressure, as in \HS, (4.34)--(4.39):
from the error form of \eqref{eq:HS-pres},
\begin{equation}\label{eq:S2prep}
\bigl(\nabla \Ck(e_p^i),\nabla q\bigr)
=\bigl(\Ck(\bu(t^i)\!\cdot\!\nabla\bu(t^i)-\bu^i\!\cdot\!\nabla\bu^i),\nabla q\bigr)
+\nu\bigl(\nabla p_s(\Ck(\be^i)),\nabla q\bigr),\qquad\forall q\in H^1(\Om).
\end{equation}
Taking $q=\Ck(e_p^i)$ gives, as in \HS, (4.36),
\[
\norm{\nabla \Ck(e_p^i)}\le\norm{\Ck\bigl(\bu(t^i)\!\cdot\!\nabla\bu(t^i)-\bu^i\!\cdot\!\nabla\bu^i\bigr)}
+\nu\norm{\nabla p_s(\Ck(\be^i))}.
\]
Using $\bu(t^i)\!\cdot\!\nabla\bu(t^i)-\bu^i\!\cdot\!\nabla\bu^i
=-\be^i\!\cdot\!\nabla\bu^i-\bu(t^i)\!\cdot\!\nabla\be^i$, and the Sobolev and
Poincar\'e inequalities give
\[
\norm{\Ck\bigl(\bu(t^i)\!\cdot\!\nabla\bu(t^i)-\bu^i\!\cdot\!\nabla\bu^i\bigr)}^2
\le C\!\!\sum_{j=i+1-k}^{i}\!\!\bigl(\norm{\nabla\be^j}^2\norm{\Delta\bu^j}^2+\norm{\bu(t^j)}_2^2\norm{\nabla\be^j}^2\bigr).
\]
We then apply Cauchy--Schwarz inequality and Young's inequalities of weights $\propto\nu$, and
Lemma~\ref{lem:LLP-paper} for the Stokes pressure to yield
\begin{equation}\label{eq:S2press}
\begin{aligned}
& \dt\bigl|\bigl(\nabla \Ck(e_p^i),-\Delta \Ck(\be^{i+1})\bigr)\bigr|
\le \frac{C\dt}{\nu}\sum_{j=i+1-k}^{i}\norm{\nabla\be^j}^2
      \Bigl(\norm{\Delta\bu^j}^2+\norm{\bu(t^j)}_2^2\Bigr)\\
&\qquad+\frac{\nu(1+2\eps)}{8\delta_\gamma}\,\dt\,\norm{\Delta \Ck(\be^i)}^2
       +\frac{C\nu\dt}{\delta_\gamma}\,\norm{\nabla \Ck(\be^i)}^2
 +(\delta_\alpha+\delta_\gamma)\,\nu\dt\,\norm{\Delta \Ck(\be^{i+1})}^2 .
\end{aligned}
\end{equation}

For the four truncation errors $P_k^i,Q_k^i,R_k^i,S_k^i$ (\HS, (4.25)--(4.28):
pressure, viscous, time, convection), Taylor expansion about $t^{i+\beta_k}$ gives, after summation,
\begin{equation}\label{eq:trunc}
\frac{\dt}{\nu}\sum_i\Bigl(\norm{P_k^i}^2+\norm{Q_k^i}^2+\norm{R_k^i}^2+\norm{S_k^i}^2\Bigr)
\le \frac{C(1+\nu^2)}{\nu}\dt^{2k}\mathcal R_k,
\end{equation}
where the factor $\nu^2$ is only from
$Q_k^i= -\nu \Delta {\bf u}(t^{i+\beta_k}) + \nu\Delta \Bk({\bf u}(t^{i+1}))$.

Combining the above, dropping the nonnegative telescoping and dissipation terms,
keeping the coefficient of $\nu\dt\norm{\Delta \Ck(\be^j)}^2$ positive as in
\eqref{eq:coerc}, and summing for $i=k-1,\dots,n$, we obtain the master estimate.
The conversion of the extrapolated quantities $\norm{\Delta \Ck(\be^i)}^2$ into the true
error norms $\norm{\Delta\be^i}^2$ is supplied by the $\Bk=\eta_k\Ck+D_k+\Fk$ decomposition:
the $\Fk$ coercivity term $\nu\kappa_k\dt\norm{\Delta\be^{i+1}}^2$ of Lemma~\ref{lem:BkCk},
exactly as in \HS, (4.40)--(4.46), absorbs the remaining $\norm{\Delta \Ck(\be^i)}^2$
contributions and leaves a positive multiple of $\nu\dt\sum_i\norm{\Delta\be^i}^2$, absorbed into $C$.
We thus reach the analogue of \HS, (4.46):
\begin{equation}\label{eq:master}
\begin{aligned}
& \norm{\nabla\be^{n+1}}^2+\nu\dt\sum_{i=k}^{n+1}\norm{\Delta\be^i}^2
+\frac{\eta_k\eps_N\dt}{2}\sum_{i=k}^{n+1}\norm{\sqrt{Q_N}\Delta\Ck(\be^i)}^2
\le \frac{C}{\nu}\dt\sum_{i=k-1}^{n}\norm{\nabla\be^i}^2
     \bigl(\norm{\Delta\bu^i}^2+\norm{\bu(t^i)}_2^2\bigr)\\
&\quad+\frac{C}{\nu}\dt\sum_{i=k-1}^{n}\norm{\nabla\Ck(\be^i)}^2
     \bigl(\norm{\Delta\Ck(\bu^i)}^2+\norm{\Ck(\bu(t^i))}_2^2+1\bigr)
     +\frac{C(1+\nu^2)}{\nu}\dt^{2k}\mathcal R_k+M_0^{\rm err}+D_{\rm svv},
\end{aligned}
\end{equation}
with $M_0^{\rm err}\le C\dt^{2k}$ and the Gronwall driver
\begin{equation*}
d_i=\norm{\Delta\bu^i}^2+\norm{\bu(t^i)}_2^2
+\!\!\sum_{q=0}^{\min\{k-1,\,n-i\}}\!\!\bigl(\norm{\Delta\Ck(\bu^{i+q})}^2+\norm{\Ck(\bu(t^{i+q}))}_2^2+1\bigr).
\end{equation*}

Note \eqref{eq:master} is in the form
required by the discrete Gronwall lemma, with $\tau=\dt$, $a_i=\norm{\nabla\be^i}^2$,
$b_i=\nu \norm{\Delta\be^i}^2$, weights $\tfrac{C}{\nu}d_i$, data
$\tau\sum c_i=\tfrac{C(1+\nu^2)}{\nu}\dt^{2k}\mathcal R_k$, and constant
$C_G=M_0^{\rm err}+D_{\rm svv}$ (the nonnegative SVV coercive term on the left of
\eqref{eq:master} rides through the lemma unchanged). By Step~1, \eqref{eq:Ca1}, and the regularity of the exact
solution,
\begin{equation}\label{eq:sumd}
\begin{aligned}
\dt\sum_{i=k-1}^{n}\frac{C}{\nu}\,d_i
&\le \frac{C}{\nu}\Bigl(\dt\sum_{i=k-1}^{n}\bigl(\norm{\Delta\bu^i}^2+\norm{\Delta \Ck(\bu^i)}^2\bigr)
   +\dt\sum_{i=k-1}^{n}\bigl(\norm{\bu(t^i)}_2^2+\norm{\Ck(\bu(t^i))}_2^2\bigr)+T\Bigr)\\
&\le \frac{C}{\nu}\Bigl(\frac{C_{a1}}{\nu}+TM_2^2+T\Bigr)
\le \frac{CC_0^6T}{\nu^5}+\frac{CTM_f^2}{\nu^3}+\frac{CTM_2^2}{\nu}+\frac{CC_0^2T}{\nu}+\frac{CC_0^2}{\nu^2}.
\end{aligned}
\end{equation}
the $\nu^{-5}$ arising from $\tfrac{1}{\nu}\cdot\tfrac{C_{a1}}{\nu}$ with
$C_{a1}=O(\nu^{-3})$. The Gronwall lemma then yields the analogue of \HS, (4.48):
\begin{equation}\label{eq:gron}
\begin{aligned}
\norm{\nabla\be^{n+1}}^2
& +\nu\dt\sum_{i=0}^{n+1}\norm{\Delta\be^i}^2
+\frac{\eta_k\eps_N\dt}{2}\sum_{i=0}^{n+1}\norm{\sqrt{Q_N}\Delta\Ck(\be^i)}^2
\le \Bigl[\frac{C(1+\nu^2)\mathcal R_k}{\nu}\dt^{2k}+M_0^{\rm err}+D_{\rm svv}\Bigr]\,\mathcal G\\
&=:C_{a2}^{\rm vel}\dt^{2k}+D_{\rm svv}\,\mathcal G,
\qquad \mathcal G:=\exp\!\Bigl(\frac{CC_0^6T}{\nu^5}+\frac{CTM_f^2}{\nu^3}+\frac{CTM_2^2}{\nu}+\frac{CC_0^2T}{\nu}+\frac{CC_0^2}{\nu^2}\Bigr).
\end{aligned}
\end{equation}

To close the induction, the triangle inequality and
$\norm{\nabla\bu(t^{n+1})}\le C_0-1$ give
\[
\norm{\nabla\bu^{n+1}}\le\norm{\nabla\bu(t^{n+1})}+\norm{\nabla\be^{n+1}}
\le (C_0-1)+\sqrt{C_{a2}^{\rm vel}\,\dt^{2k}+D_{\rm svv}\,\mathcal G}\le C_0
\]
provided
$C_{a2}^{\rm vel}\dt^{2k}+D_{\rm svv}\mathcal G\le1$, which is implied by the smallness condition
\eqref{eq:smallness} because $C_{a2}^{\rm vel}\le C_R\,\mathcal G$ and $\mathcal B\ge1$. This completes the induction.

\bigskip
\noindent\textbf{Step 3.\quad Error estimate for the pressure.}

\medskip
Setting $q=e_p^i$ in the error form of \eqref{eq:HS-pres} and using Lemma~\ref{lem:LLP-paper}
and the Sobolev inequality, as in \HS, (4.51), yield
\begin{equation}\label{eq:S3}
\norm{\nabla e_p^i}^2\le C\norm{\nabla\be^i}^2\Bigl(\norm{\Delta\bu^i}^2+\norm{\bu(t^i)}_2^2\Bigr)+2\nu^2\norm{\Delta\be^i}^2+C\nu^2\norm{\nabla\be^i}^2,\quad 1\le i\le n+1.
\end{equation}
Multiplying by $\dt/\nu$, summing $i=1,\dots,n+1$, and using \eqref{eq:Ca1},
\eqref{eq:gron} together with $\dt\sum_i\norm{\Delta\bu^i}^2\le C_{a1}/\nu$,
\begin{equation}\label{eq:S3sum}
\frac{\dt}{\nu}\sum_{i=1}^{n+1}\norm{\nabla e_p^i}^2\le C_{a2}^{\rm press}\dt^{2k}+\mathcal B\,D_{\rm svv}\,\mathcal G,
\qquad C_{a2}^{\rm press}\le C\,\mathcal B\,C_{a2}^{\rm vel},
\end{equation}
with $\mathcal B$ as in \eqref{eq:svv-Ca2}. The Stokes-pressure step multiplies the
SVV data term by the same factor $\mathcal B$ that multiplies $C_R$, because the
first term on the right of \eqref{eq:S3} inherits $D_{\rm svv}\mathcal G$ from
\eqref{eq:gron} through $\sup_i\norm{\nabla\be^i}^2$.

Adding \eqref{eq:gron} and \eqref{eq:S3sum} gives Part~2, \eqref{eq:svv-err}, with
amplification $\mathcal B$, data factor $C_R$, and Gronwall factor $\mathcal G$ as in
\eqref{eq:svv-Ca2}. The velocity part contributes $D_{\rm svv}\mathcal G$ and the
pressure part $\mathcal B\,D_{\rm svv}\mathcal G$, whose sum is bounded by
$2\mathcal B\,D_{\rm svv}\mathcal G$ since $\mathcal B\ge1$, and the factor $2$ is
absorbed into $D_{\rm svv}$.
Setting $\eps_N=0$ annihilates the coercive SVV term on the left and $D_{\rm svv}$
on the right, recovering the $\nu$-dependent version of
\cite[Thm.~4.1]{HuangShen2025}.
\end{proof}

\section{Numerical experiments}
\label{sec:numerics}
In this section, we present several two-dimensional numerical experiments to demonstrate the stability and accuracy of the proposed SVV-stabilized higher-order consistent splitting scheme with spectral spatial discretizations. In Examples 1 and 2, we employ the Legendre--Galerkin method \cite{GuermondShen2003,Shen1994}. On the domain $(-1,1)^2$, the velocity is approximated in the tensor-product space $X_N\otimes X_N$, with $X_N$ defined in Section~\ref{sec:scheme}, while the pressure is approximated in $Y_N\otimes Y_N$, where $Y_N=\operatorname{span}\{L_m:m=0,1,\ldots,N\}$ and $L_m$ denotes the Legendre polynomial of degree $m$. In Example 3, we employ a Fourier--trigonometric spatial discretization, whose details are provided in Appendix~\ref{app:fourier-trig}.

\subsection{Example 1: convergence test}
\label{sec:conv}

We verify that the scheme attains its design temporal order $k$, and the SVV term does not degrade that order where the temporal error dominates. We use the manufactured solution of \cite[Example~2]{HuangShen2025} on
$\Om=(-1,1)^2$,
\[
\bu=\bigl(\sin(2\pi y)\sin^2(\pi x),\,-\sin(2\pi x)\sin^2(\pi y)\bigr)\sin t,
\qquad p=\cos(\pi x)\sin(\pi y)\sin t.
\]
Figure~\ref{fig:eg1-flow} shows the velocity field and the vorticity of this solution at $t=1$.
\begin{figure}[H]
\centering
\includegraphics[width=6.5cm]{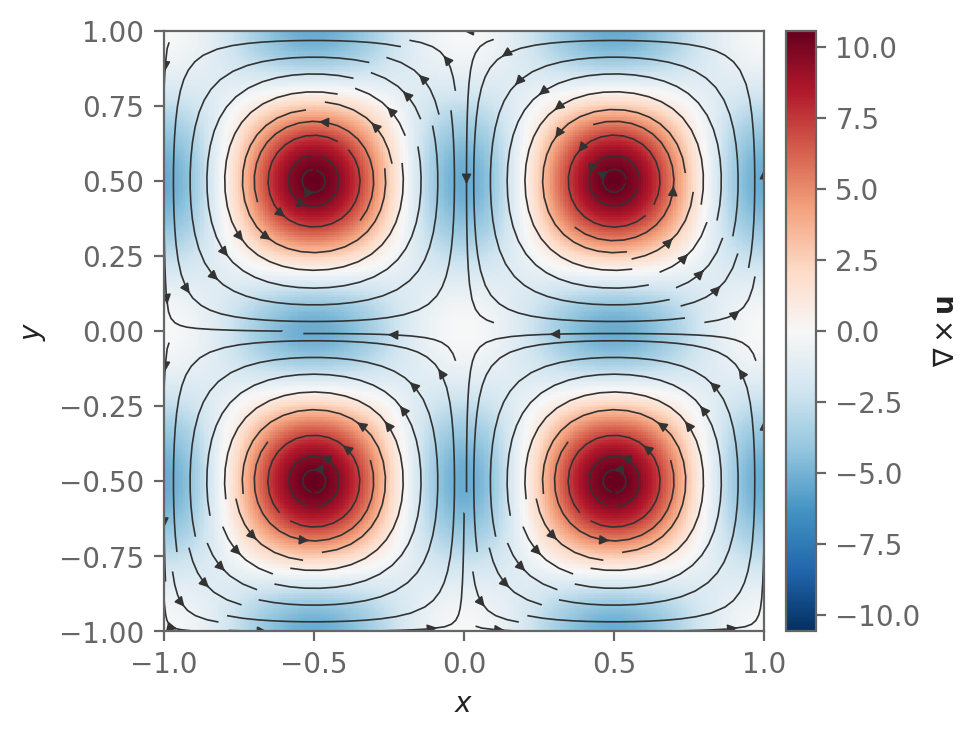}
\caption{Example 1. Streamlines and vorticity $\nabla\times\bu$ of the manufactured solution at $t=1$.}
\label{fig:eg1-flow}
\end{figure}

We take $N=128$ modes in each direction so that the spatial error is negligible
against the temporal error, integrate to $T=1$, and report the relative $L^2$
errors $E_u=\norm{\bu^n-\bu(t^n)}/\norm{\bu(t^n)}$ and
$E_p=\norm{p^n-p(t^n)}/\norm{p(t^n)}$ at $T=1$, comparing the bare scheme
($C_{\rm svv}=0$) with the SVV-stabilized scheme ($C_{\rm svv}=1$,
$\eps_N=C_{\rm svv}/M$, kernel \eqref{eq:MKTkernel}). We consider two viscosities,
$\nu=10^{-3}$ and $\nu=10^{-4}$, and generate the first $k$ levels by the
Richardson-BE self-start, except where the exact solution is used for comparison.

When $\nu=10^{-3}$, the bare scheme achieves the optimal convergence rates, as shown in Tables~\ref{tab:conv-nu3-k2}--\ref{tab:conv-nu3-k4}. For $k=2$, the SVV-stabilized scheme produces results nearly identical to those of the bare scheme, whereas for $k=3$ and $k=4$, the velocity error saturates at approximately $10^{-4}$, as observed in Table~\ref{tab:conv-nu3-k3} at $\delta t=0.00625$ and in Table~\ref{tab:conv-nu3-k4} at $\delta t=0.0125$ and $0.00625$. This saturation is caused by the SVV term $\varepsilon_N Q_N\Delta B_k({\bf u}^{n+1})$ in scheme~\eqref{eq:SVV-vel}, which contributes the error term $\mathcal B D_{\rm svv}$ in \eqref{eq:svv-err}, where $D_{\rm svv}\le c\,\eps_N T\sup_t\norm{\sqrt{Q_N}\Delta\bu(t)}^2$. The size of this term is explained by the basis in which the kernel acts. The eigenfunctions $\psi_i$ approximate the sine modes of the Dirichlet Laplacian, and a smooth function that vanishes at $x=\pm1$ but whose second derivative does not, such as the factor $\sin^2(\pi x)$ of the present solution, has coefficients $\widehat f_i$ that decay only like $i^{-3}$ in this basis. At $N=128$ we find $i^3|\widehat f_i|\approx10.5$ for $10\le i\le100$. The SVV term therefore damps a small but non-negligible part of the exact solution, carried by the modes on which the kernel is appreciable. A linear, mode-by-mode estimate that balances the SVV forcing $\eps_N(\widehat Q_i\mu_i+\widehat Q_j\mu_j)\widehat u_{ij}$ against the damping rate $\nu(\mu_i+\mu_j)+\eps_N(\widehat Q_i\mu_i+\widehat Q_j\mu_j)$ predicts a relative velocity error at $T=1$ of $1.04\times10^{-4}$ for $\nu=10^{-3}$ and $1.55\times10^{-4}$ for $\nu=10^{-4}$, in close agreement with the saturation levels $1.06\times10^{-4}$ and $1.65\times10^{-4}$ observed for $k=4$ in Tables~\ref{tab:conv-nu3-k4} and~\ref{tab:blowup-nu4}. Because this coefficient decay is algebraic, the floor decreases only algebraically, not spectrally, as $N$ increases. It is independent of $\dt$ and remains below the temporal discretization error over most of the $\dt$ range in Tables~\ref{tab:conv-nu3-k2}--\ref{tab:conv-nu3-k4}, becoming visible only for the smallest time steps when $k=3$ and $k=4$.

\begin{table}[H]
\small
\centering
\caption{$\nu=10^{-3}$, $N=128$, temporal convergence for $k=2$ ($\beta_2=3$) at $T=1$. }
\label{tab:conv-nu3-k2}
\begin{tabular}{c cccc cccc}
\hline
 & \multicolumn{4}{c}{$E_u$} & \multicolumn{4}{c}{$E_p$}\\
\cline{2-5}\cline{6-9}
$\dt$ & bare & order & SVV & order & bare & order & SVV & order\\
\hline
$0.1000$ & $1.63\times10^{-1}$ & --     & $1.62\times10^{-1}$ & --     & $1.15\times10^{-1}$ & --     & $1.15\times10^{-1}$ & --\\
$0.0500$ & $5.09\times10^{-2}$ & $1.67$ & $5.09\times10^{-2}$ & $1.67$ & $3.83\times10^{-2}$ & $1.59$ & $3.82\times10^{-2}$ & $1.59$\\
$0.0250$ & $1.46\times10^{-2}$ & $1.80$ & $1.46\times10^{-2}$ & $1.80$ & $1.12\times10^{-2}$ & $1.77$ & $1.12\times10^{-2}$ & $1.77$\\
$0.0125$ & $3.90\times10^{-3}$ & $1.90$ & $3.89\times10^{-3}$ & $1.90$ & $3.03\times10^{-3}$ & $1.89$ & $3.03\times10^{-3}$ & $1.89$\\
$0.00625$ & $1.00\times10^{-3}$ & $1.96$ & $1.01\times10^{-3}$ & $1.95$ & $7.82\times10^{-4}$ & $1.95$ & $7.83\times10^{-4}$ & $1.95$\\
\hline
\end{tabular}
\end{table}

\begin{table}[H]
\small
\centering
\caption{$\nu=10^{-3}$, $N=128$, temporal convergence for $k=3$ ($\beta_3=6$) at $T=1$. }
\label{tab:conv-nu3-k3}
\begin{tabular}{c cccc cccc}
\hline
 & \multicolumn{4}{c}{$E_u$} & \multicolumn{4}{c}{$E_p$}\\
\cline{2-5}\cline{6-9}
$\dt$ & bare & order & SVV & order & bare & order & SVV & order\\
\hline
$0.1000$ & $1.31\times10^{-1}$ & --     & $1.30\times10^{-1}$ & --     & $8.69\times10^{-2}$ & --     & $8.68\times10^{-2}$ & --\\
$0.0500$ & $3.23\times10^{-2}$ & $2.01$ & $3.23\times10^{-2}$ & $2.02$ & $2.22\times10^{-2}$ & $1.97$ & $2.22\times10^{-2}$ & $1.97$\\
$0.0250$ & $5.40\times10^{-3}$ & $2.58$ & $5.39\times10^{-3}$ & $2.58$ & $3.85\times10^{-3}$ & $2.53$ & $3.85\times10^{-3}$ & $2.53$\\
$0.0125$ & $7.37\times10^{-4}$ & $2.87$ & $7.43\times10^{-4}$ & $2.86$ & $5.28\times10^{-4}$ & $2.87$ & $5.28\times10^{-4}$ & $2.87$\\
$0.00625$ & $9.45\times10^{-5}$ & $2.96$ & $1.42\times10^{-4}$ & $2.39$ & $6.73\times10^{-5}$ & $2.97$ & $6.90\times10^{-5}$ & $2.94$\\
\hline
\end{tabular}
\end{table}

\begin{table}[H]
\small
\centering
\caption{$\nu=10^{-3}$, $N=128$, temporal convergence for $k=4$ ($\beta_4=9$) at $T=1$.  }
\label{tab:conv-nu3-k4}
\begin{tabular}{c cccc cccc}
\hline
 & \multicolumn{4}{c}{$E_u$} & \multicolumn{4}{c}{$E_p$}\\
\cline{2-5}\cline{6-9}
$\dt$ & bare & order & SVV & order & bare & order & SVV & order\\
\hline
$0.10000$ & $1.65\times10^{-2}$ & --     & $1.65\times10^{-2}$ & --     & $1.18\times10^{-2}$ & --     & $1.18\times10^{-2}$ & --\\
$0.05000$ & $4.21\times10^{-3}$ & $1.97$ & $5.32\times10^{-3}$ & $1.63$ & $3.21\times10^{-3}$ & $1.88$ & $3.22\times10^{-3}$ & $1.88$\\
$0.02500$ & $8.11\times10^{-4}$ & $2.37$ & $5.60\times10^{-4}$ & $3.25$ & $4.54\times10^{-4}$ & $2.82$ & $4.24\times10^{-4}$ & $2.92$\\
$0.01250$ & $4.43\times10^{-5}$ & $4.19$ & $1.14\times10^{-4}$ & $2.29$ & $3.48\times10^{-5}$ & $3.70$ & $3.53\times10^{-5}$ & $3.59$\\
$0.00625$ & $3.08\times10^{-6}$ & $3.85$ & $1.06\times10^{-4}$ & $0.11$ & $2.43\times10^{-6}$ & $3.84$ & $1.10\times10^{-5}$ & $1.68$\\
\hline
\end{tabular}
\end{table}

To check the dependence on the spatial resolution, we repeat the test with $N=256$. Table~\ref{tab:conv-nu3-N256} and Figure~\ref{fig:conv-N128-N256} report the velocity errors. The errors of the bare scheme coincide with those of $N=128$ to the digits shown, except for $k=4$ at $\dt=0.025$ and $\dt=0.00625$, so the reported errors are temporal. For the stabilized scheme the saturation level for $k=4$ decreases from $1.06\times10^{-4}$ to $1.64\times10^{-5}$, in agreement with the value $1.60\times10^{-5}$ given by the linear estimate above, and it is no longer visible for $k=3$. The floor therefore persists at higher resolution but decreases, by a factor of about $6.4$ when $N$ is doubled.

\begin{table}[H]
\small
\centering
\caption{$\nu=10^{-3}$, $N=256$. Relative $L^2$ velocity error $E_u$ at $T=1$ for the bare and the SVV-stabilized schemes.}
\label{tab:conv-nu3-N256}
\begin{tabular}{c c cc cc}
\hline
$k$ & $\dt$ & bare & order & SVV & order\\
\hline
\multirow{5}{*}{$2$}
 & $0.1$ & $1.63\times10^{-1}$ & -- & $1.62\times10^{-1}$ & --\\
 & $0.05$ & $5.09\times10^{-2}$ & $1.67$ & $5.09\times10^{-2}$ & $1.67$\\
 & $0.025$ & $1.46\times10^{-2}$ & $1.80$ & $1.46\times10^{-2}$ & $1.80$\\
 & $0.0125$ & $3.90\times10^{-3}$ & $1.90$ & $3.89\times10^{-3}$ & $1.90$\\
 & $0.00625$ & $1.00\times10^{-3}$ & $1.96$ & $1.00\times10^{-3}$ & $1.96$\\
\hline
\multirow{5}{*}{$3$}
 & $0.1$ & $1.31\times10^{-1}$ & -- & $1.31\times10^{-1}$ & --\\
 & $0.05$ & $3.23\times10^{-2}$ & $2.02$ & $3.23\times10^{-2}$ & $2.02$\\
 & $0.025$ & $5.40\times10^{-3}$ & $2.58$ & $5.40\times10^{-3}$ & $2.58$\\
 & $0.0125$ & $7.37\times10^{-4}$ & $2.87$ & $7.36\times10^{-4}$ & $2.87$\\
 & $0.00625$ & $9.45\times10^{-5}$ & $2.96$ & $9.57\times10^{-5}$ & $2.94$\\
\hline
\multirow{5}{*}{$4$}
 & $0.1$ & $1.65\times10^{-2}$ & -- & $1.65\times10^{-2}$ & --\\
 & $0.05$ & $4.21\times10^{-3}$ & $1.97$ & $4.20\times10^{-3}$ & $1.97$\\
 & $0.025$ & $5.50\times10^{-4}$ & $2.94$ & $5.49\times10^{-4}$ & $2.94$\\
 & $0.0125$ & $4.43\times10^{-5}$ & $3.63$ & $4.70\times10^{-5}$ & $3.54$\\
 & $0.00625$ & $3.36\times10^{-6}$ & $3.72$ & $1.64\times10^{-5}$ & $1.52$\\
\hline
\end{tabular}
\end{table}

\begin{figure}[H]
\centering
\includegraphics[width=\textwidth]{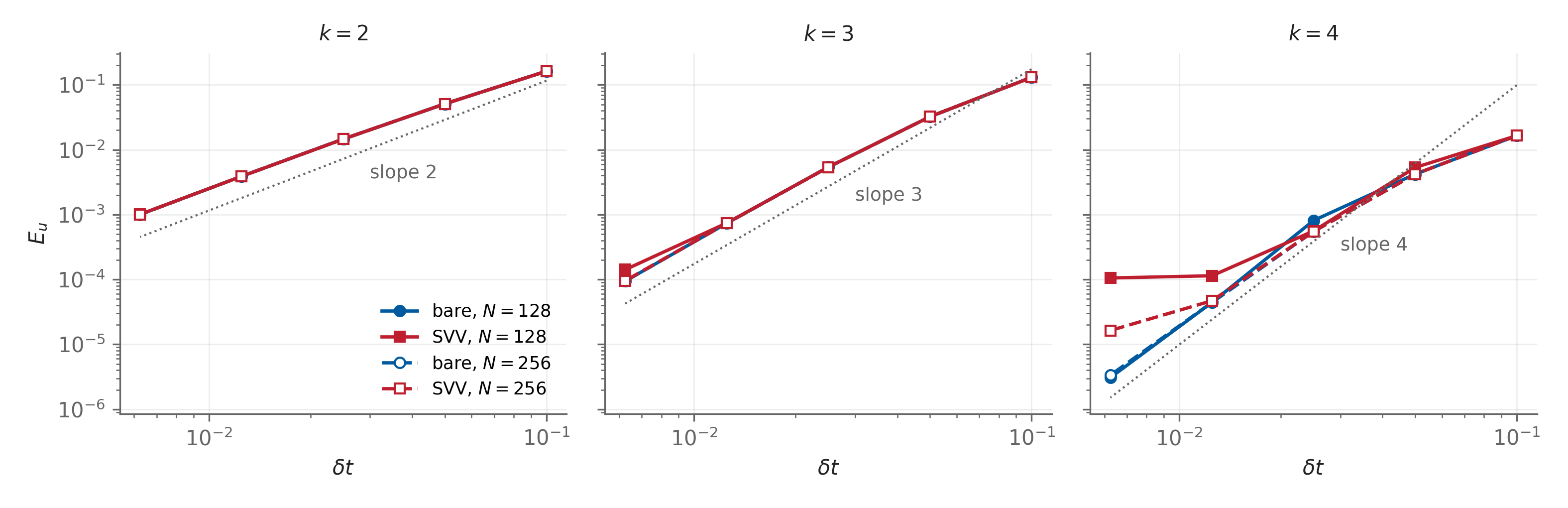}
\caption{Example 1 at $\nu=10^{-3}$. Relative velocity error $E_u$ at $T=1$ against the time step for $k=2,3,4$, for the bare and the SVV-stabilized schemes with $N=128$ and $N=256$. The dotted lines have slope $k$.}
\label{fig:conv-N128-N256}
\end{figure}

At $\nu=10^{-4}$ ($\Reyn=10^4$), the bare scheme is no longer stable for most values of $\delta t$, as shown in Table~\ref{tab:blowup-nu4}. The table reports the relative velocity error $E_u$ at $T=1$ for $k=2,3,4$, using two initialization procedures for the bare scheme: the exact solution and Richardson-BE initialization.
The bare scheme diverges under \emph{both} initialization procedures, with the error reaching $10^{14}$--$10^{25}$ or overflowing to a non-finite value for most choices of $\dt$; only a few isolated cases with sufficiently small $\dt$ remain stable. Because the exact initialization supplies the true solution at the first $k$ time levels and the computation still blows up, the instability is attributable to the bare scheme itself at this Reynolds number rather than to the initialization procedure. Since it disappears at the smallest time step with $N$ fixed, it acts as a viscosity-dependent restriction on the time step. In contrast, the SVV-stabilized scheme with the Richardson-BE initialization yields stable and convergent solutions for every tested $\dt$ except the largest step of the fourth-order scheme ($\dt=0.05$), with the attainable accuracy limited by the SVV saturation level $\approx 10^{-4}$  discussed above. At the smallest time step, where the bare scheme is stable, it is more accurate than the stabilized scheme for $k=3,4$, which is the price of this floor. For comparison, the $k=2$ block also lists the second-order GSAV consistent splitting scheme of \citet{AlhomsiWuZheng2026}: it remains bounded but loses accuracy at the intermediate time steps and converges only at the smallest one. The pressure error exhibits the same behavior and is therefore omitted.

\begin{table}[H]
\small
\centering
\caption{$\nu=10^{-4}$ ($\Reyn=10^4$), $N=128$. Relative $L^2$ velocity error $E_u$ at $T=1$ under the bare scheme with the exact and Richardson-BE initializations and the SVV-stabilized scheme with Richardson-BE initialization.  ``NaN'' denotes a non-finite (overflowed) value. The SAV column is the second-order
GSAV consistent splitting scheme of \citet{AlhomsiWuZheng2026} ($\beta_2=4$, exact start), listed in the $k=2$ block only since it is a second-order method. The last column is
the observed order of the SVV velocity error.}
\label{tab:blowup-nu4}
\begin{tabular}{c c cccc c}
\hline
$k$ & $\dt$ & bare, exact start & bare, Richardson-BE & SAV & SVV, Richardson-BE & order\\
\hline
\multirow{5}{*}{$2$}
 & $0.1000$ & $1.66\times10^{-1}$ & $1.66\times10^{-1}$ & $2.54\times10^{-1}$ & $1.77\times10^{-1}$ & --\\
 & $0.0500$ & $6.28\times10^{-2}$ & $2.15\times10^{2}$  & $1.15\times10^{-1}$ & $7.62\times10^{-2}$ & $1.22$\\
 & $0.0250$ & $1.40\times10^{1}$  & $7.09\times10^{25}$ & $1.27\times10^{0}$  & $1.49\times10^{-2}$ & $2.35$\\
 & $0.0125$ & $7.14\times10^{-1}$ & $4.32\times10^{14}$ & $9.92\times10^{-1}$ & $3.99\times10^{-3}$ & $1.90$\\
 & $0.00625$ & $1.03\times10^{-3}$ & $1.03\times10^{-3}$ & $1.71\times10^{-3}$ & $1.04\times10^{-3}$ & $1.94$\\
\hline
\multirow{5}{*}{$3$}
 & $0.1000$ & $1.34\times10^{-1}$ & $1.42\times10^{-1}$ & -- & $1.35\times10^{-1}$ & --\\
 & $0.0500$ & $1.52\times10^{-1}$ & $1.81\times10^{16}$ & -- & $1.29\times10^{-1}$ & $0.07$\\
 & $0.0250$ & $3.47\times10^{25}$ & $\mathrm{NaN}$      & -- & $5.90\times10^{-3}$ & $4.45$\\
 & $0.0125$ & $4.64\times10^{20}$ & $\mathrm{NaN}$      & -- & $7.78\times10^{-4}$ & $2.92$\\
 & $0.00625$ & $9.81\times10^{-5}$ & $3.47\times10^{-4}$ & -- & $1.92\times10^{-4}$ & $2.02$\\
\hline
\multirow{4}{*}{$4$}
 & $0.05000$ & $4.75\times10^{-3}$ & $4.17\times10^{2}$  & -- & $7.39\times10^{1}$   & --\\
 & $0.02500$ & $1.76\times10^{16}$ & $\mathrm{NaN}$      & -- & $1.04\times10^{-2}$ & $12.79$\\
 & $0.01250$ & $1.20\times10^{14}$ & $\mathrm{NaN}$      & -- & $1.70\times10^{-4}$ & $5.94$\\
 & $0.00625$ & $1.29\times10^{-5}$ & $1.18\times10^{-5}$ & -- & $1.65\times10^{-4}$ & $0.04$\\
\hline
\end{tabular}
\end{table}

The same test with $N=256$ is reported in Table~\ref{tab:blowup-nu4-N256}. The bare scheme now diverges for every $\dt\le0.05$, including the smallest step, at which it was stable for $N=128$. The stabilized scheme also diverges at the intermediate time steps. It converges again at $\dt=0.00625$, where the errors are $1.03\times10^{-3}$, $1.03\times10^{-4}$ and $3.17\times10^{-5}$ for $k=2,3,4$. The last value is the SVV floor, for which the linear estimate gives $3.06\times10^{-5}$. In this test the stabilization with $\eps_N=C_{\rm svv}/M$ therefore enlarges the range of stable time steps but does not remove the restriction on the time step, which becomes more severe as $N$ increases while $\eps_N$ decreases.

\begin{table}[H]
\small
\centering
\caption{$\nu=10^{-4}$ ($\Reyn=10^4$), $N=256$. Relative $L^2$ velocity error $E_u$ at $T=1$ for the bare and the SVV-stabilized schemes, both with Richardson-BE initialization. ``NaN'' denotes a non-finite (overflowed) value.}
\label{tab:blowup-nu4-N256}
\begin{tabular}{c c c c}
\hline
$k$ & $\dt$ & bare & SVV\\
\hline
\multirow{5}{*}{$2$}
 & $0.1$ & $1.66\times10^{-1}$ & $1.68\times10^{-1}$\\
 & $0.05$ & $4.50\times10^{1}$ & $1.33\times10^{0}$\\
 & $0.025$ & $\mathrm{NaN}$ & $7.84\times10^{1}$\\
 & $0.0125$ & $\mathrm{NaN}$ & $1.05\times10^{-2}$\\
 & $0.00625$ & $\mathrm{NaN}$ & $1.03\times10^{-3}$\\
\hline
\multirow{5}{*}{$3$}
 & $0.1$ & $1.34\times10^{-1}$ & $1.35\times10^{-1}$\\
 & $0.05$ & $4.69\times10^{17}$ & $3.94\times10^{2}$\\
 & $0.025$ & $\mathrm{NaN}$ & $1.27\times10^{21}$\\
 & $0.0125$ & $\mathrm{NaN}$ & $2.55\times10^{-2}$\\
 & $0.00625$ & $\mathrm{NaN}$ & $1.02\times10^{-4}$\\
\hline
\multirow{5}{*}{$4$}
 & $0.1$ & $1.67\times10^{-2}$ & $1.68\times10^{-2}$\\
 & $0.05$ & $5.38\times10^{35}$ & $1.39\times10^{4}$\\
 & $0.025$ & $\mathrm{NaN}$ & $6.09\times10^{84}$\\
 & $0.0125$ & $\mathrm{NaN}$ & $3.16\times10^{6}$\\
 & $0.00625$ & $\mathrm{NaN}$ & $3.17\times10^{-5}$\\
\hline
\end{tabular}
\end{table}

\subsection{Example 2: perturbed Kovasznay problem}
\label{sec:kov}
\subsubsection{Problem setup and long time behavior}
We consider the steady Kovasznay flow \citep{Kovasznay1948} on
$\Om=(-1,1)^2$,
\begin{align}\label{eq:Kov}
\bu_{Kov} &= \bigl(\,1-e^{\lambda x}\cos(2\pi y)\,,\;
\tfrac{\lambda}{2\pi}e^{\lambda x}\sin(2\pi y)\,\bigr),
&p_{Kov} &= \tfrac{1}{2}(1-e^{2\lambda x}),
\end{align}
with $\lambda = \tfrac{\Reyn}{2} - \sqrt{\tfrac{\Reyn^2}{4}+4\pi^2}$, as the base flow. Its graph with $\Reyn=10^4$ is shown in Fig.\,\ref{fig:eg2_Kovasanayflow}.
We add a small divergence-free perturbation
\begin{equation}\label{eq:bump}
\bw_0(x,y) = A_{\rm pert}\Bigl(-4y(1-x^2)^2(1-y^2),\;
4x(1-x^2)(1-y^2)^2\Bigr),
\end{equation}
with $A_{\rm pert}=10^{-2}$, as the initial perturbation. We solve
\eqref{eq:NSE} with $\bu(0,\cdot)=\bu_{Kov}+\bw_0$ with the Dirichlet boundary condition $\bu=\bu_{Kov}$ on $\partial\Om$. We therefore subtract the steady solution and solve the evolution equation for the perturbation 
\begin{equation}
\bw=\bu-\bu_{Kov},
\end{equation}
which satisfies $\bw|_{\partial\Om}=0$.
\begin{figure}[h]
\centering
\includegraphics[width=5cm, height=4.5cm]{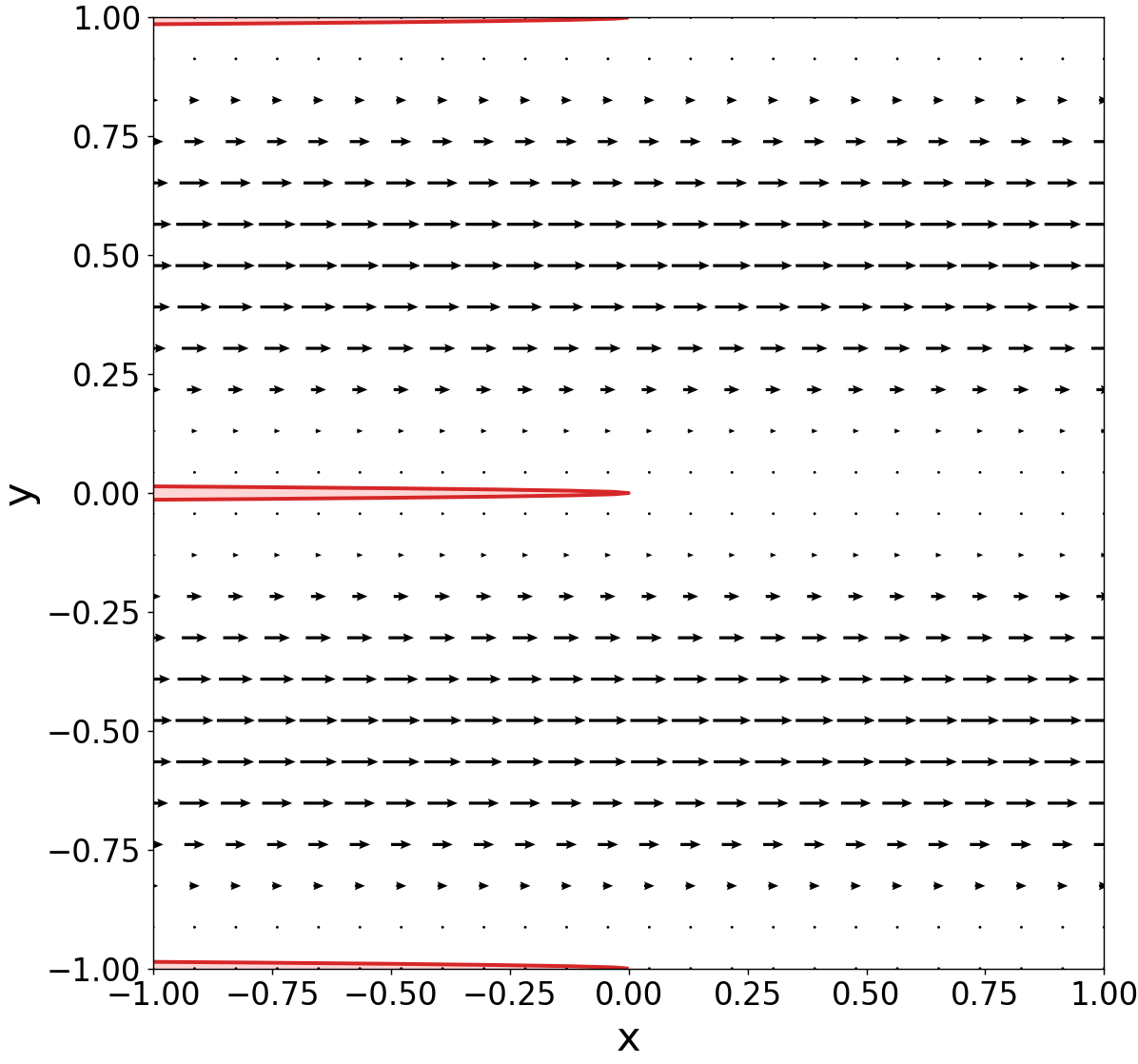}
\caption{Example 2. Kovasznay flow $\bu_{Kov}$ at $\Reyn=10^4$. The red region is where the first component $u_{{Kov},1}$ is negative.}
\label{fig:eg2_Kovasanayflow}
\end{figure}

The spatial discretization is the Legendre--Galerkin method
\citep{Shen1994} with $N=128$ polynomial degree per direction in all
production runs. The time step is $\dt=10^{-4}$ throughout. SVV
parameters are $C_{\rm svv}=1$, $m_N=\lceil\sqrt M\rceil=12$,  unless otherwise specified.

We test the consistent splitting scheme for $\Reyn=10^2, 10^3, 10^4$, using $k=2$ and $\dt=10^{-4}$, both with and without SVV. For $N=128$ and $\Reyn=10^3$, we also test $k=2,3,4$ with $\dt=10^{-4}$ and $10^{-5}$. The results are nearly identical to those obtained with $k=2$ and $\dt=10^{-4}$, indicating that the spatial discretization error dominates in this setting. Therefore, below we report results for different spatial resolutions while fixing $k=2$ and $\dt=10^{-4}$.

Figure~\ref{fig:eg2_timehistory}[a] shows $\norm{\bw(t)}$ at $\Reyn=10^2,10^3,10^4$ at $N=128$, with and without SVV. The SVV results indicate that the perturbation decays exponentially and the long-time solution returns to the Kovasznay flow. These curves are insensitive to the resolution: the SVV results for $N=128,256,512,1024$ coincide at each Reynolds number, as shown for $\Reyn=10^4$ in Figure~\ref{fig:eg2_timehistory}[b].
At $\Reyn=10^2$ the two schemes agree to all reported digits, so no stabilization is needed. At $\Reyn=10^3$ the scheme without SVV reaches a minimum near $t\approx 40$ and then grows to $1.2\times10^{-1}$ at $t=1000$, and at $\Reyn=10^4$ it blows up near $t\approx 10$. SVV thus improves stability and accuracy together, removing the spurious growth and the blow-up and driving the solution to the steady Kovasznay flow.
\begin{figure}[h]
\centering
\includegraphics[width=5cm, height=4.5cm]{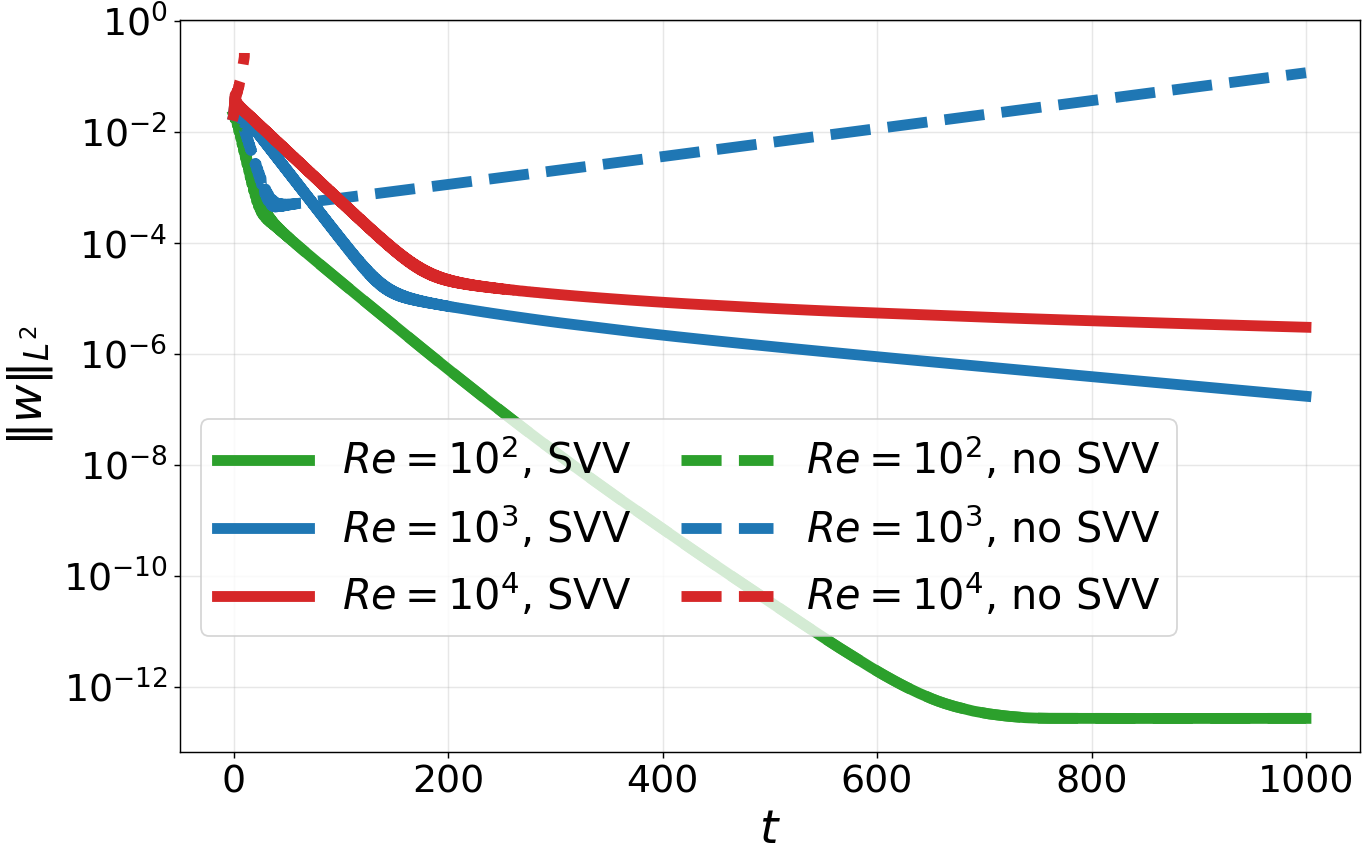}[a]
\includegraphics[width=5cm, height=4.5cm]{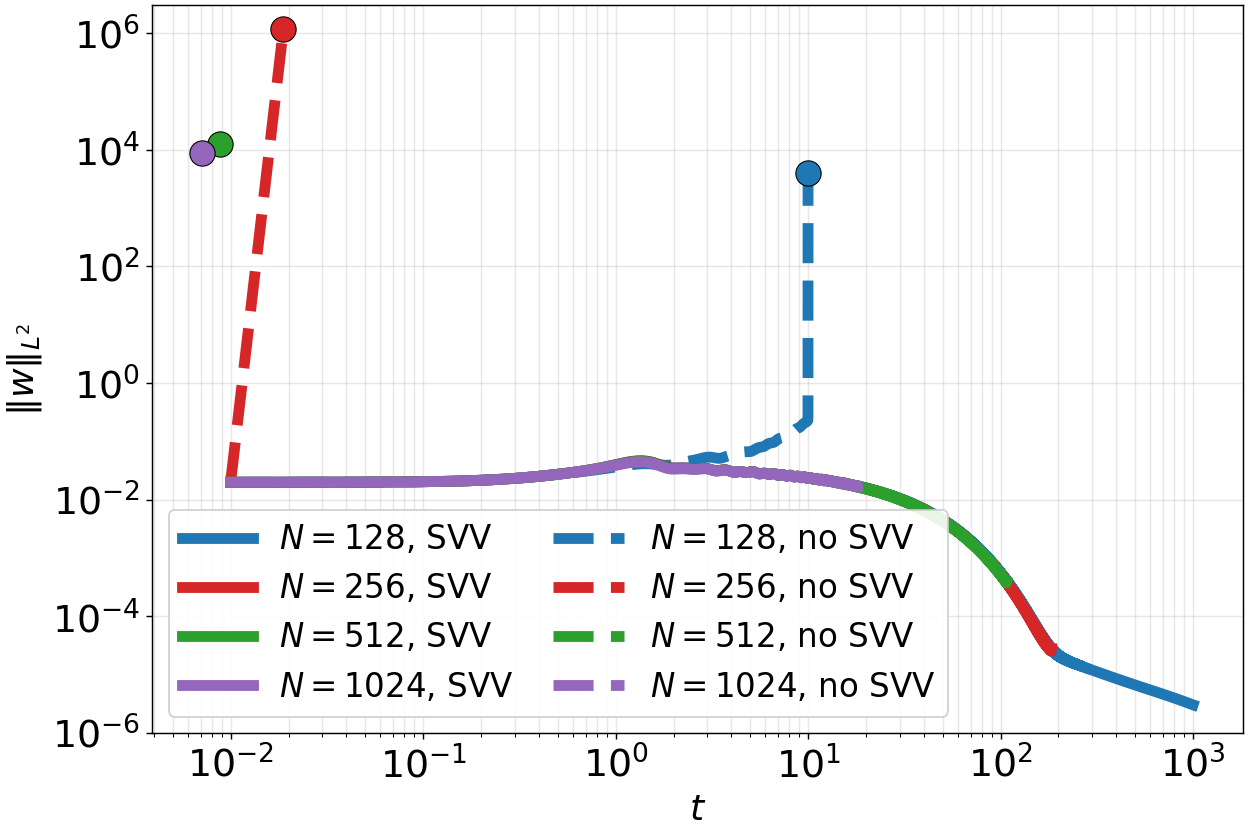}[b]
\caption{Example 2. [a]:  $\norm{\bw(t)}_{L^2(\Om)}$ vs time for $\Reyn=10^2, 10^3, 10^4$. The curves for $Re=10^2$ with SVV and without SVV coincide. 
[b]: $\norm{\bw(t)}_{L^2(\Om)}$ vs time for $\Reyn=10^4$ with different mesh resolutions. The solid curves (SVV results) coincide. The dots are where the scheme blows up.}
\label{fig:eg2_timehistory}
\end{figure}

\subsubsection{Boundary layer}

The main difficulty in this problem is the boundary layer at the outflow wall $x=1$ and the oscillations associated with it.  The total velocity satisfies $\bu=\bu_{Kov}$ on $\partial\Om$, and because $u_{Kov,1}(1,y)>0$ for every $y$ (see Fig.\,\ref{fig:eg2_Kovasanayflow}),  the fluid at $x=1$ flows out of the domain. The perturbation satisfies $\bw=\bm 0$ on $\partial\Om$. It is therefore carried toward $x=1$ by the base flow and forced to vanish there, and a boundary layer forms. Figure~\ref{fig:vorticityfulldomain} shows the perturbation vorticity $\nabla\times\bw$ over the full domain, with the region of interest located near the right wall. Figure~\ref{fig:vorticitysmallregion} provides a magnified view of this region at different mesh resolutions.
\begin{figure}[H]
\centering
\includegraphics[width=10cm, height=4.5cm]{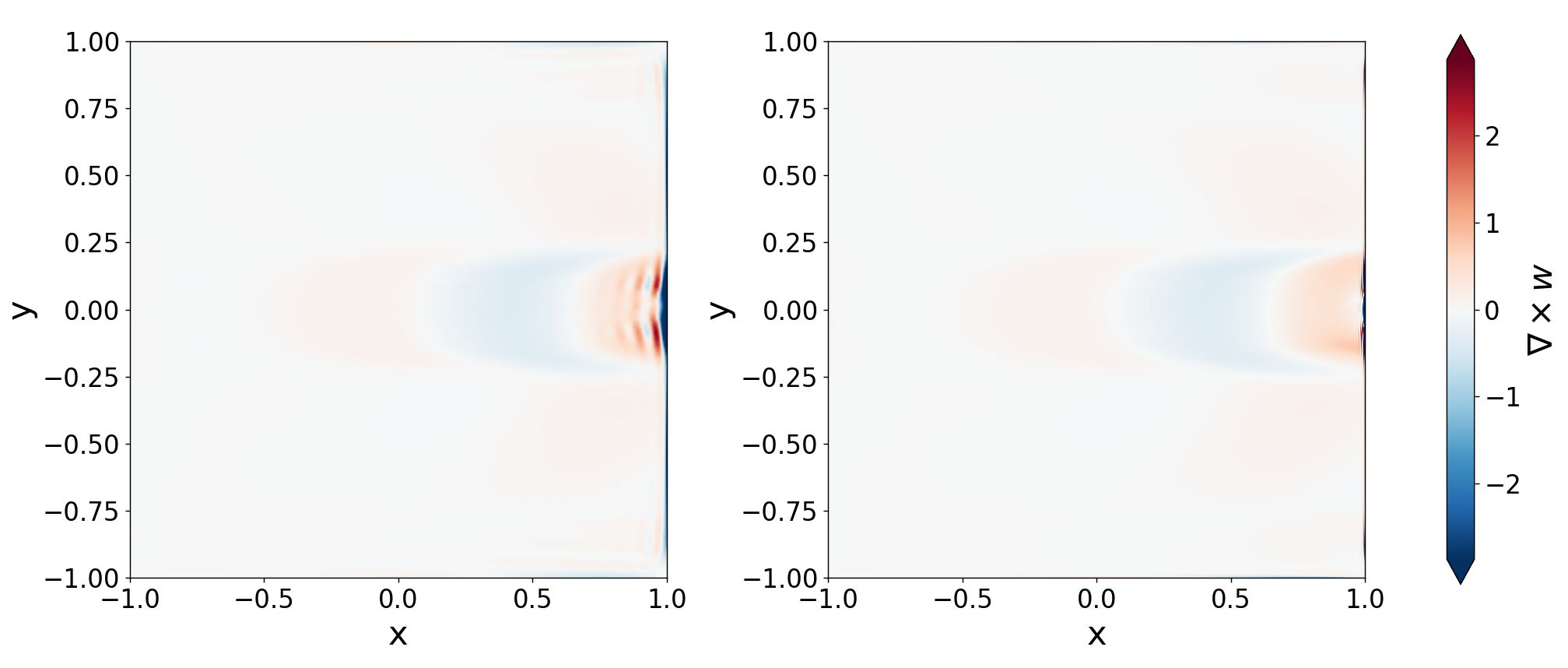}
\caption{Example 2 perturbation vorticity fields $\nabla\times {\bw}$ at $t=10$ for $\Reyn=10^4$, $N=128$ (left) and $N=1024$ (right).}
\label{fig:vorticityfulldomain}
\end{figure}

\begin{figure}[H]
\centering
\includegraphics[width=12cm, height=4cm]{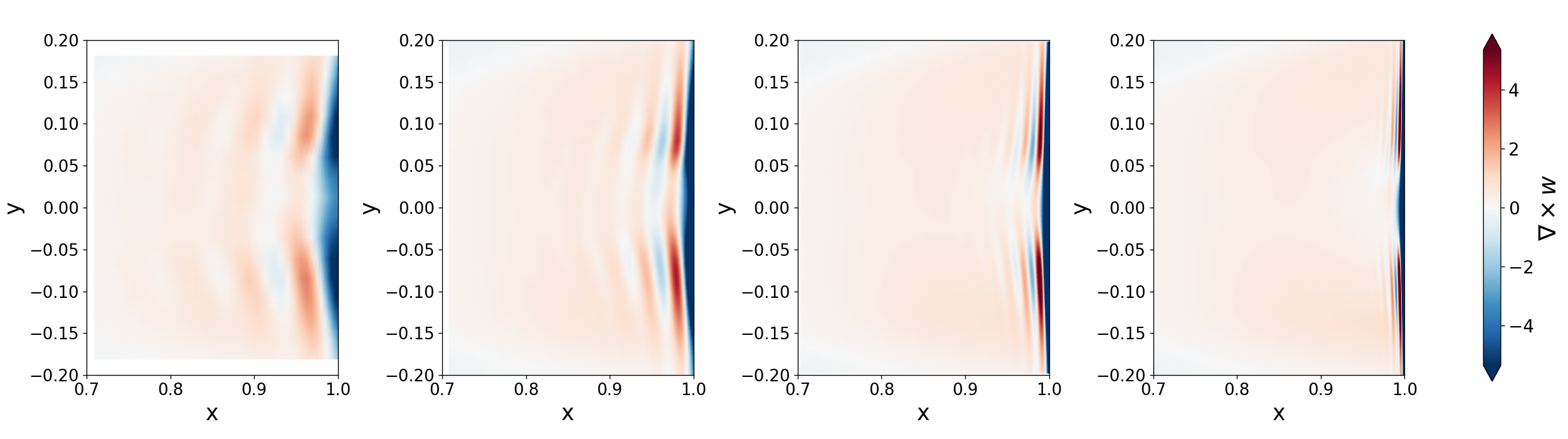}
\caption{Example 2 perturbation vorticity $\nabla\times {\bw}$ near right boundary at $t=10$ for $\Reyn=10^4$. From left to right: $N=128, 256, 512, 1024$.}
\label{fig:vorticitysmallregion}
\end{figure}

Fig.\,\ref{fig:boundarylayer} shows the profiles of the perturbation quantities $\bw=(w_1,w_2)$, $\nabla\times\bw$ and $p-p_{Kov}$ along $y=0$ for $\Reyn=10^4$.
The boundary layer thickness, distance from the outflow wall to the peak of $w_2$ along $y=0$, converges to  $1.3\times10^{-2}$  under mesh refinement. This layer contains $7$, $13$, $27$ and $53$ Legendre--Gauss--Lobatto points for $N=128$, $256$, $512$ and $1024$, respectively.
\begin{figure}[H]
\centering
\includegraphics[width=14cm, height=6cm]{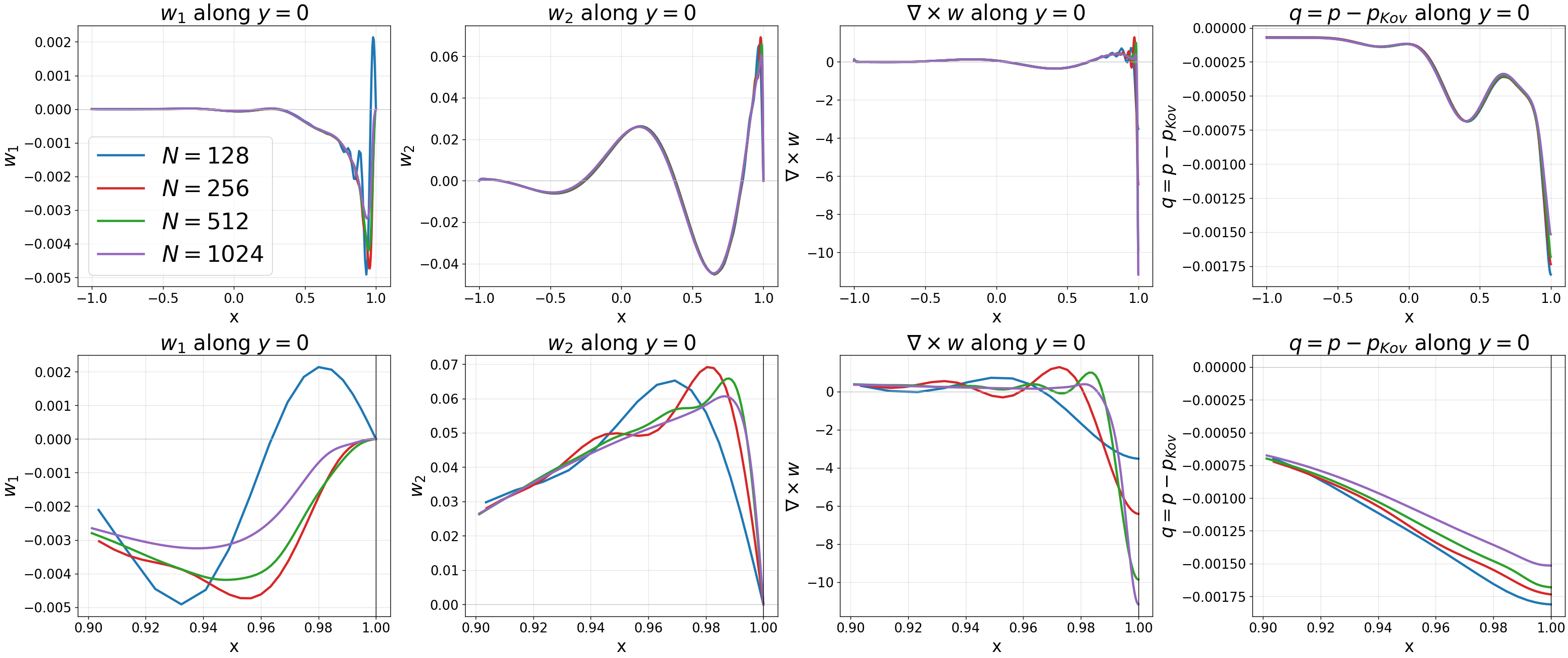}[a]
\caption{Example 2 boundary layer for $\Reyn=10^4$: the perturbation quantities, $w_1$, $w_2$, $\nabla\times {\bw}$, and $q=p-p_{Kov}$ along $y=0$ at $t=10$. First row: the interval $x\in [-1,1]$. Second row: the small interval $x\in[0.9,1]$. }
\label{fig:boundarylayer}
\end{figure}

Since no exact solution is available, we take the $\Reyn=10^4$,  $N=1024$ run as the reference and report the relative error
$\norm{\bw_N-\bw_{1024}}/\norm{\bw_{1024}}$
at $t=10$, both over the whole domain and over the outflow region $[0.8,1]\times[-0.2,0.2]$.  Table
\ref{tab:conv-space-Re} collects the results, which shows convergence at a rate below first order, limited by the boundary layer.
\begin{table}[H]
\small
\centering
\caption{Example 2. Spatial convergence of the perturbation velocity at $t=10$
for $\Reyn=10^4$, measured against the $N=1024$ run.}
\label{tab:conv-space-Re}
\begin{tabular}{c cc cc}
\hline
 & \multicolumn{2}{c}{whole domain} & \multicolumn{2}{c}{$[0.8,1]\times[-0.2,0.2]$}\\
\cline{2-3}\cline{4-5}
$N$ & relative error & order & relative error & order\\
\hline
$128$ & $1.45\times10^{-1}$ & --     & $3.35\times10^{-1}$ & --\\
$256$ & $9.44\times10^{-2}$ & $0.62$ & $2.12\times10^{-1}$ & $0.66$\\
$512$ & $5.17\times10^{-2}$ & $0.87$ & $1.12\times10^{-1}$ & $0.92$\\
\hline
\end{tabular}
\end{table}

We follow the layer in time and across Reynolds number. Along $y=0$ we
record the peak of $|w_2|$ near the outflow wall, its distance from the wall
$1-x_{\rm peak}$, the wall derivative $\partial_x w_2(1,0)$, and two ratios,
\begin{equation}\label{eq:layer-measures}
d_{\rm slope}=\frac{\max|w_2|}{\left|\partial_x w_2(1,0)\right|},
\qquad
R_{\nabla}=\frac{\left|\partial_x w_2(1,0)\right|}
                {\max_{|x|<0.8}\left|\partial_x w_2\right|}.
\end{equation}
Both are independent of the size of the perturbation, since numerator and
denominator carry it equally. Because $w_1$ vanishes along $x=1$, its
tangential derivative there is zero and $\partial_x w_2(1,0)$ is the boundary
vorticity. Fig.\,\ref{fig:layerevolution} shows these quantities for
$\Reyn=10^2$, $10^3$ and $10^4$ at $N=256$. The perturbation oscillates with a
period close to $0.73$ at all three Reynolds numbers and the wall derivative
changes sign on every cycle, so the two ratios are evaluated at the crest of
each cycle, where the layer is strongest.

The perturbation decays exponentially in every case and the decay slows as the
Reynolds number rises. After a transient ending near $t=10$ the three layer
measures settle onto constants and hold them while the amplitude falls by five
to seven decades, so the layer keeps its shape as the solution returns to the
Kovasznay flow and only its amplitude decreases. For $\Reyn=10^3$ and $10^4$ we follow the measures only up to $t\approx193$ and $t\approx235$, beyond which $\norm{\bw}_{L^2(\Om)}$ levels off while the layer amplitude keeps decaying and the peak-location measures no longer describe the outflow layer. The thickness orders cleanly
with the Reynolds number: taking medians over $t>20$, the peak of $w_2$ sits
$5.0\times10^{-2}$, $2.9\times10^{-2}$ and $1.9\times10^{-2}$ from the wall for
$\Reyn=10^2$, $10^3$ and $10^4$, and $d_{\rm slope}$ follows the same ordering
with $1.86$, $1.32$ and $1.07\times10^{-2}$. The ratio $R_\nabla$ stays between
$11$ and $21$ throughout, so the wall gradient exceeds the largest interior
gradient by more than an order of magnitude at every Reynolds number, but it is
not monotone in $\Reyn$ and we do not read an ordering from it. This resolution
is needed for the comparison: at $N=128$ the curves for $\Reyn=10^3$ and
$10^4$ coincide in all three layer measures, which reflects the grid rather
than the flow.

\begin{figure}[H]
\centering
\includegraphics[width=15cm]{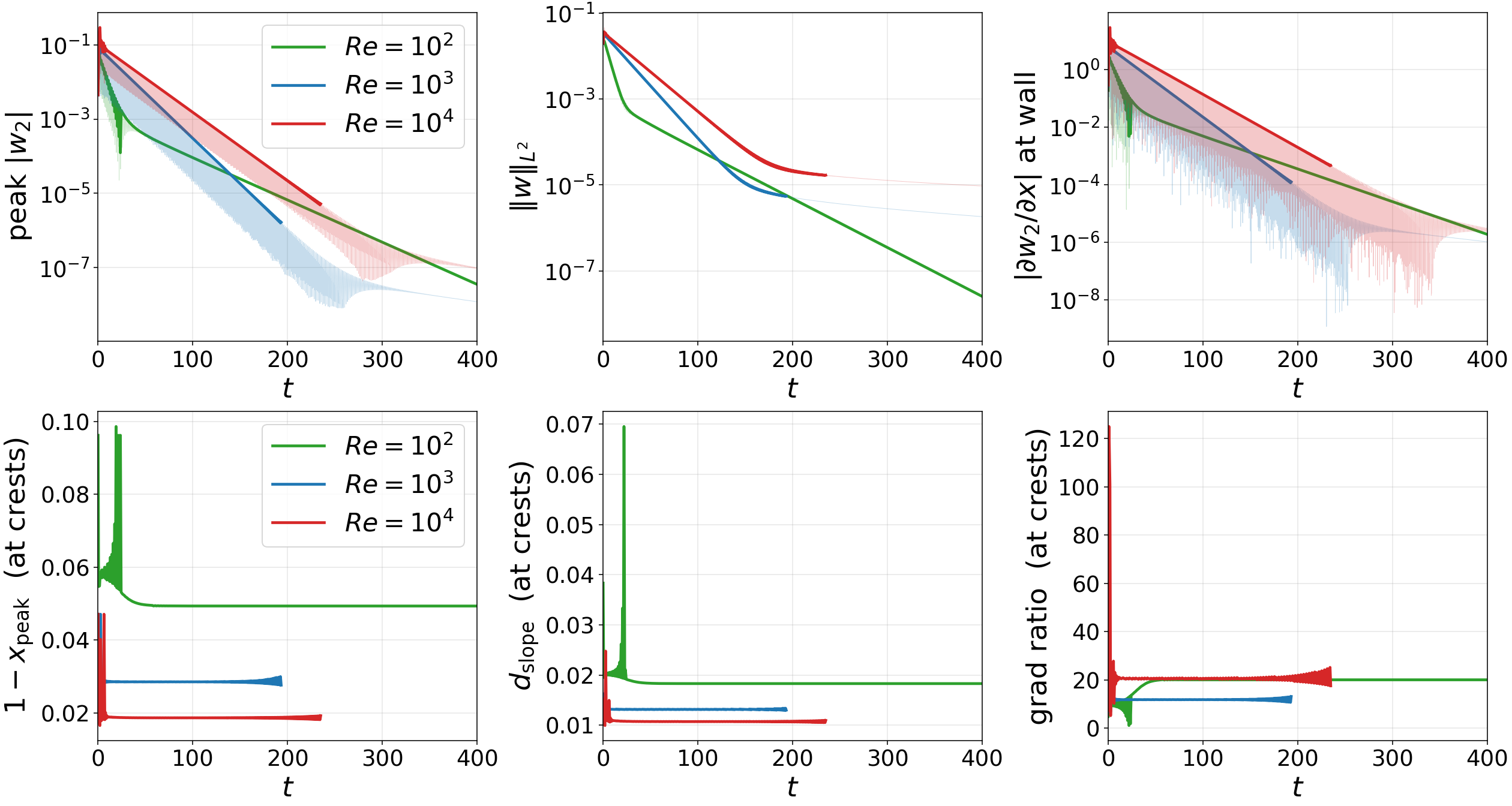}
\caption{Example 2. Outflow layer along $y=0$ at $N=256$ for $\Reyn=10^2$
(green), $10^3$ (blue) and $10^4$ (red). Upper row: peak of $|w_2|$, the
whole-domain norm $\norm{\bw}_{L^2(\Om)}$, and the wall derivative
$|\partial_x w_2(1,0)|$, on a logarithmic scale. Lower row: the distance
$1-x_{\rm peak}$ and the two ratios of \eqref{eq:layer-measures}. The faint
lines are the full recorded signal at every step. The bold lines trace the
crest of each oscillation cycle, and follow the signal itself where it no
longer oscillates.}
\label{fig:layerevolution}
\end{figure}

\subsubsection{Comparison with a FEM/Newton scheme}
We close this example by comparing the cost of reaching $T=1$ at $\Reyn=10^4$ with a
finite-element/Newton solver, analysed in \cite{GarciaArchillaJohnNovo2025} and applied
to this problem in \cite{AlhomsiWuZheng2026}, which uses Taylor--Hood $P_2/P_1$ elements
with grad-div stabilization, fully implicit BDF-2 in time, and a Newton solve at every
step, both codes run at $\dt=10^{-4}$. Table~\ref{tab:cost} reports the cost. The two
codes are not comparable at equal $N$, since $N$ counts polynomial degree for the spectral
scheme and mesh cells for the finite-element one, so we compare at a similar number of
unknowns. There the spectral scheme is about $42$ times cheaper near $0.6$ to $0.8$ million
degrees of freedom and about $29$ times cheaper near $2.4$ to $3.1$ million, taking $85$\,ms
per step against $1792$\,ms and using $0.37$\,GB against $54.5$\,GB. Equal degrees of
freedom is not equal error, so these figures measure the cost of carrying a given number of
unknowns rather than the cost of reaching a given accuracy.

\begin{table}[H]
\small
\centering
\caption{Example 2. Cost of integrating to $T=1$ at $\Reyn=10^4$ with
$\dt=10^{-4}$, for the present Legendre-Galerkin SVV scheme and a
Galerkin--Newton Taylor--Hood $P_2/P_1$ solver.  Memory is the total over
the job: the spectral runs are a single shared-memory process, and the
finite-element runs are the number of ranks times the largest resident set size
among them.}
\label{tab:cost}
\begin{tabular}{l c r r r r r r r}
\hline
method & $N$ & velocity dof & total dof & wall (s) & cores & core-h
       & ms/step & memory (GB)\\
\hline
spectral & $128$  & $32\,258$    & $48\,899$    & $46.1$   & $8$  & $0.10$  & $4.6$    & $0.11$\\
         & $256$  & $130\,050$   & $196\,099$   & $204.7$  & $8$  & $0.46$  & $20.5$   & $0.11$\\
         & $512$  & $522\,242$   & $785\,411$   & $853.8$  & $8$  & $1.90$  & $85.4$   & $0.37$\\
         & $1024$ & $2\,093\,058$& $3\,143\,683$& $6555.3$ & $8$  & $14.57$ & $655.5$  & $1.10$\\
\hline
Newton   & $128$  & $132\,098$   & $148\,739$   & $1856.5$ & $16$ & $8.25$  & $185.7$  & $17.0$\\
         & $256$  & $526\,338$   & $592\,387$   & $17917.1$& $16$ & $79.63$ & $1791.7$ & $54.5$\\
         & $512$  & $2\,101\,250$& $2\,364\,419$& $23945.8$& $64$ & $425.70$& $2394.6$ & $215.7$\\
\hline
\end{tabular}
\end{table}

To weigh cost against accuracy we compare the two solutions at the point where they are
hardest to compute. Figure~\ref{fig:specvsnewton} shows the perturbation vorticity near the
outflow wall and the profile of $w_2$ along $y=0$ at $t=1$, for the spectral scheme at
$N=1024$ and the finite-element solver at $N=512$, two runs of comparable size, $3.1$ and
$2.4$ million unknowns, costing $14.6$ and $425.7$ core-hours. Away from the wall the two
agree closely, the interior lobes of $w_2$ near $x=-0.1$ and $x=0.75$ differing by a few
percent, so both solvers capture the transported perturbation equally well.

The difference is confined to the layer. The spectral peak of $w_2$ reaches
$1.81\times10^{-1}$ against $1.06\times10^{-1}$ and the peak vorticity $62.9$ against
$25.2$, while both place the peak at nearly the same distance from the wall,
$1.33\times10^{-2}$ and $1.56\times10^{-2}$, so it is the height of the layer that is
missed, not its position. The reason is visible in the markers of the right panel. The
Legendre--Gauss--Lobatto nodes are spaced $7.0\times10^{-6}$ apart at the wall against
$3.9\times10^{-3}$ for the uniform lattice, a factor of $560$, so the finite-element cut
carries about four points across the layer where the spectral cut carries several dozen.
The vorticity fields show the same thing, the spectral solution resolving a train of
alternating bands upstream of the wall sheet that the finite-element solution does not
represent at all.

\begin{figure}[H]
\centering
\includegraphics[width=7cm]{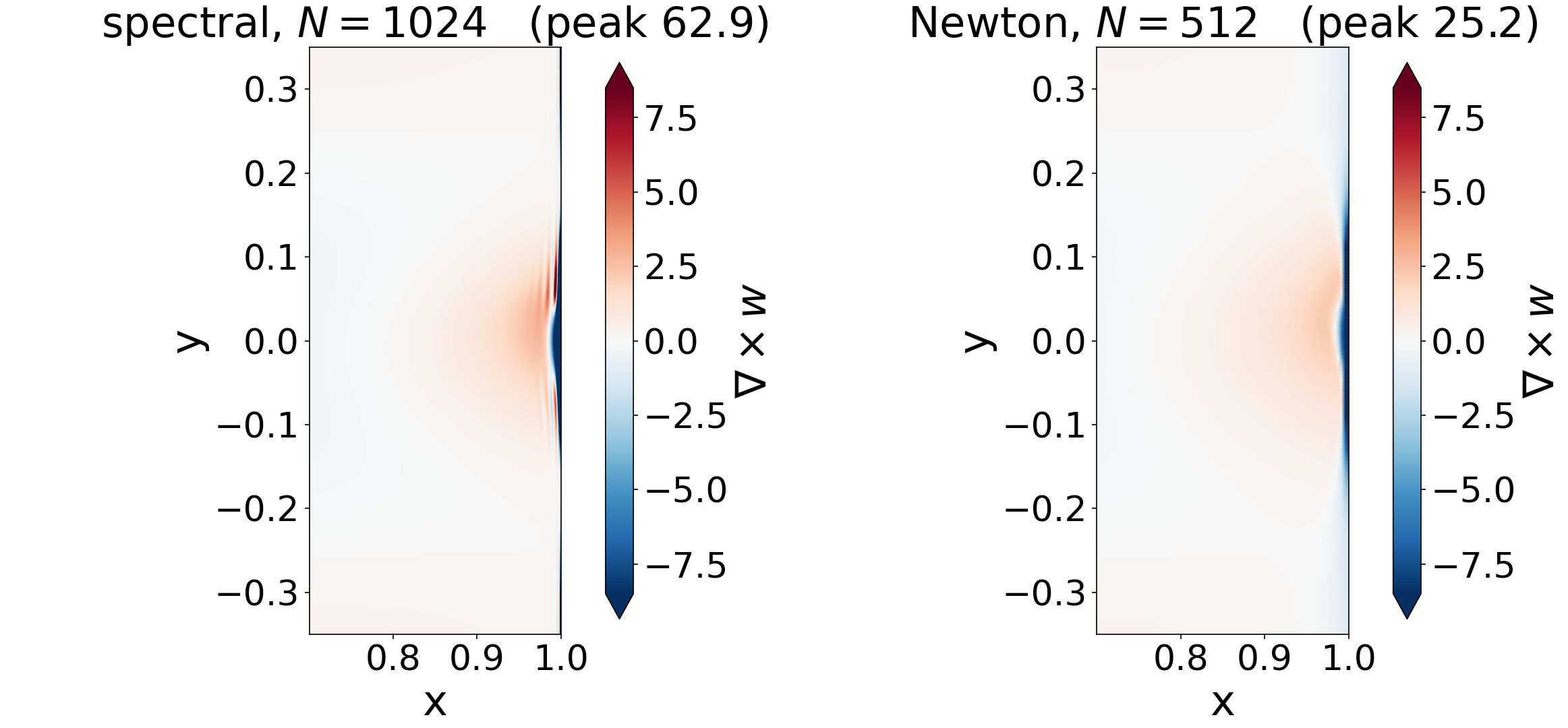}\quad
\includegraphics[width=7cm]{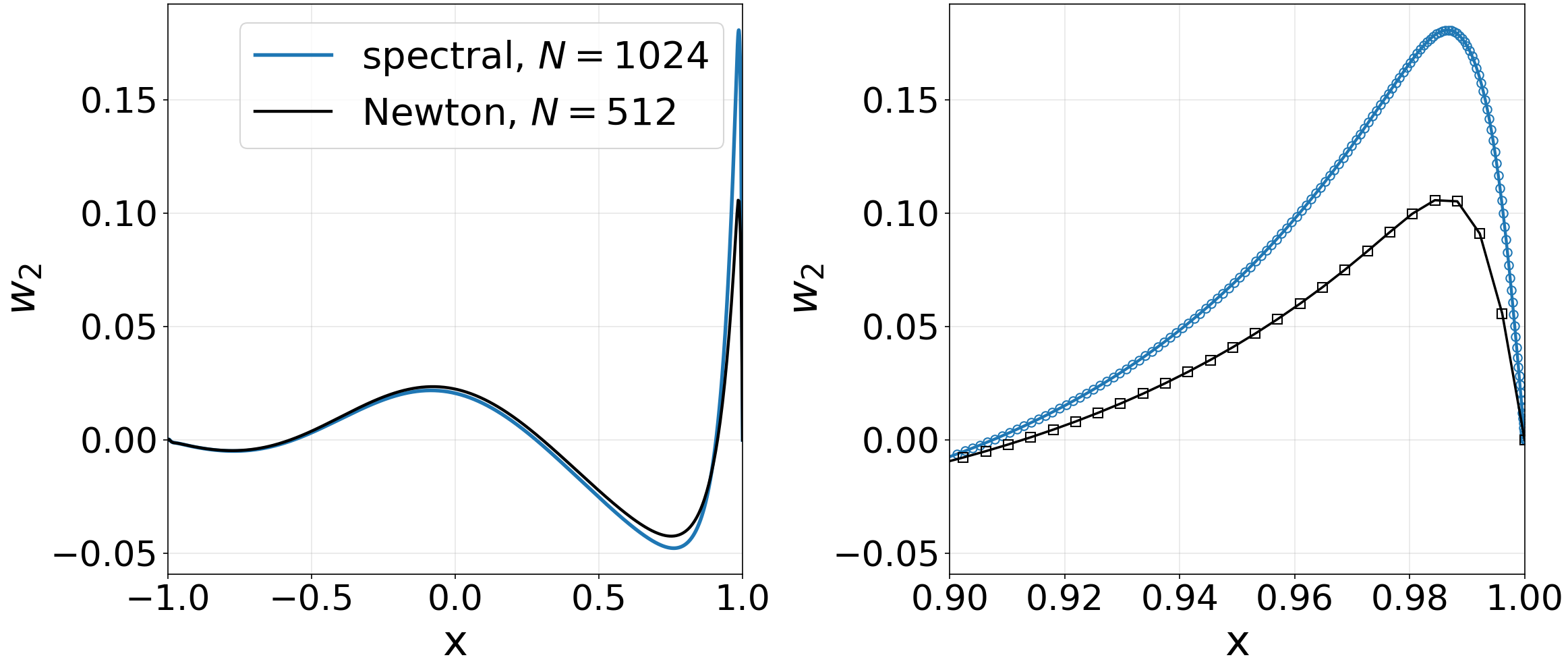}
\caption{Example 2 at $\Reyn=10^4$ and $t=1$, comparing the spectral scheme with
$N=1024$ against the Galerkin--Newton solver with $N=512$. The two left panels show
the perturbation vorticity $\nabla\times\bw$ near the outflow wall for the two methods. The two
right panels show $w_2$ along $y=0$ over the whole interval and in a zoom near the
outflow wall.}
\label{fig:specvsnewton}
\end{figure}

\subsection{Example 3: Kelvin--Helmholtz problem}
\label{sec:KH}
The two-dimensional Kelvin--Helmholtz instability is a classical shear-flow problem for incompressible solvers at high Reynolds number. We adopt the specific configuration in \cite{Schroeder2019KH,OlshanskiiRebholz2020}. On the unit square
$\Om=(0,1)^2$  the velocity $\bu=(u_1,u_2)$ and pressure $p$ solve the incompressible Navier--Stokes equations
\eqref{eq:NSE-intro}. The boundary conditions are periodic in $x$ and free-slip at the walls
$y=0,1$, the latter imposed as $\bu\cdot\bm n=0$ and
$\bigl(-\nu\,\nabla\bu\cdot\bm n\bigr)\times\bm n=\bm 0$. The free-slip condition is equivalent to 
\begin{equation}
\label{KH_bc}
u_2=0 \quad\text{and}\quad \partial_y u_1=0 \qquad\text{on } y\in\{0,1\}.
\end{equation}
The pressure is also periodic in $x$ direction and satisfies 
$\partial_y p = 0$  on $y\in\{0,1\}$ by using the relation $\frac{\partial p}{\partial n}=-\bm n\cdot ((\bu\!\cdot\!\nabla)\bu+\nu\,\nabla\times\nabla\times\bu)$ and \eqref{KH_bc}.
The spatial discretization of this problem is described in Appendix\,\ref{app:fourier-trig}.

The initial condition is defined by
\begin{equation}
u^0_1=u_\infty\tanh\!\left(\dfrac{2y-1}{\delta_0}\right)+c_n\,\partial_y\psi(x,y),\qquad
u^0_2=-c_n\,\partial_x\psi(x,y).
\end{equation}
where $\delta_0=\tfrac{1}{28}$ is the initial vorticity thickness, $u_\infty=1$
is a reference velocity, $c_n=10^{-3}$ is a noise/scaling factor, and
$
\psi(x,y)=u_\infty\bigl(\cos(8\pi x)+\cos(20\pi x)\bigr) \,\exp(-(y-0.5)^2/\delta_0^2)
$.

The Reynolds number is defined by
$Re=\delta_0\,u_\infty/\nu=1/(28\nu)$, and $\nu$ is determined by selecting $Re$. We take $\Reyn=10^3$ and $10^4$. As documented in \cite{Schroeder2019KH,OlshanskiiRebholz2020}, the ensuing roll-up of the layer into vortices and their successive pairing is strongly sensitive to small perturbations, which makes the problem a demanding test of the numerical stability.

Figure~\ref{fig:KH_vort_Re_compare} shows the early evolution of the vorticity field $\nabla\times\bu$ for the SVV-stabilized scheme with $k=4$ and $N=512$, comparing $\Reyn=10^3$ (top row) with $\Reyn=10^4$ (bottom row) on the common colour scale of \cite{Schroeder2019KH}. At $t=0$ the two rows are identical, since the initial shear layer is independent of $\Reyn$. As the layer rolls up into the two primary vortices, the two Reynolds numbers separate: at $\Reyn=10^4$ the braids connecting the cores stay thin and sharply defined and the vortices retain a large vorticity magnitude ($\min\nabla\times\bu\approx-57$ at $t=7$), whereas at $\Reyn=10^3$ the stronger viscous diffusion smears the braids and weakens the cores ($\min\nabla\times\bu\approx-44$ at $t=7$). At both Reynolds numbers the SVV-stabilized run is free of the spurious high-wavenumber oscillations that render the bare scheme unstable at this resolution, and the roll-up proceeds cleanly, consistent with the reference computations of \cite{Schroeder2019KH,OlshanskiiRebholz2020}.

\begin{figure}[H]
\centering
\includegraphics[width=\textwidth]{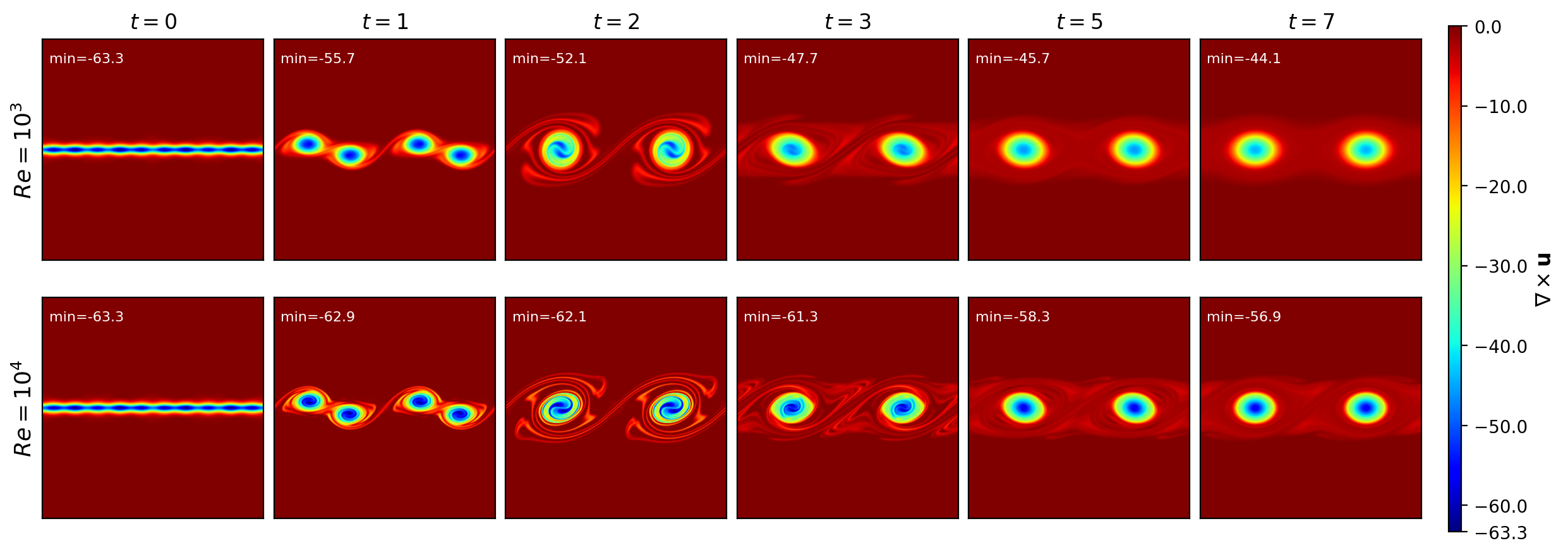}
\caption{Kelvin--Helmholtz problem: vorticity $\nabla\times\bu$ for the SVV-stabilized scheme with $k=4$ and $N=512$ at $t=0,1,2,3,5,7$ (left to right). Top row: $\Reyn=10^3$; bottom row: $\Reyn=10^4$. Both rows share the colour scale of \cite{Schroeder2019KH}, $\nabla\times\bu\in[-63.3,\,0]$; the value printed in each panel is the minimum vorticity.}
\label{fig:KH_vort_Re_compare}
\end{figure}

To expose what the SVV term suppresses, Figure~\ref{fig:KH_N128_svv} shows the vorticity at the coarsest resolution $N=128$ and $t=2$, well within the reliable regime $t\le 7$. The three left panels of each row use the SVV-stabilized scheme with $k=2,3,4$ and are almost indistinguishable from one another, and from the $N=256$ and $N=512$ SVV results, so that with SVV even the coarse grid already captures the correct roll-up at both Reynolds numbers. The three right panels repeat the computation with SVV switched off. At $\Reyn=10^3$ small spurious oscillations appear around the vortices, while at $\Reyn=10^4$ they contaminate the entire field: the bare scheme injects grid-scale noise whose amplitude ($\min\nabla\times\bu\approx-106$, far below the physical range) overwhelms the true vorticity. The SVV term removes exactly this under-resolved, high-wavenumber content while leaving the resolved vortices intact.

\begin{figure}[H]
\centering
\includegraphics[width=\textwidth]{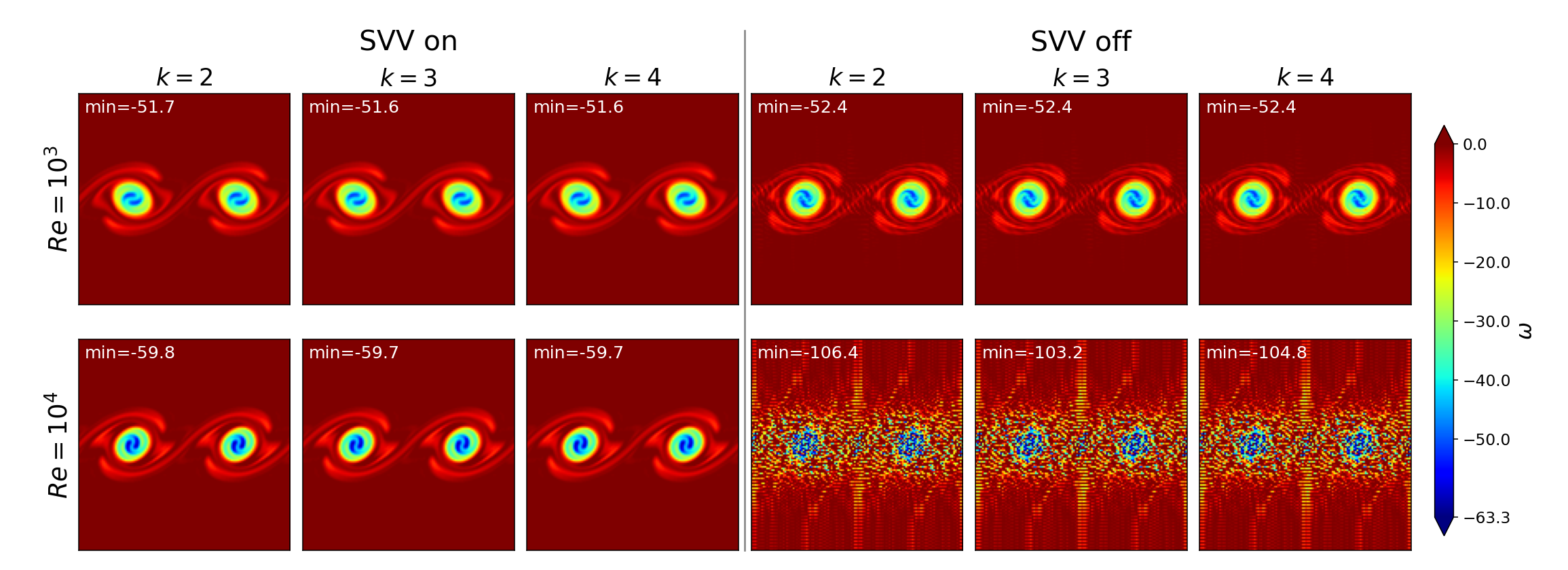}
\caption{Kelvin--Helmholtz problem at the coarse resolution $N=128$ and $t=2$ (within the reliable regime): vorticity $\nabla\times\bu$ for $\Reyn=10^3$ (top row) and $\Reyn=10^4$ (bottom row). In each row the left three panels use SVV with $k=2,3,4$ and the right three switch SVV off. With SVV the $k=2,3,4$ results coincide; without SVV spurious high-wavenumber oscillations appear, severely so at $\Reyn=10^4$. Common colour scale of \cite{Schroeder2019KH}, $\nabla\times\bu\in[-63.3,\,0]$ (values below $-63.3$ are clipped to the end colour); the number in each panel is the minimum vorticity.}
\label{fig:KH_N128_svv}
\end{figure}

This behaviour is consistent across the parameters we tested. With SVV active the coarse ($N=128$), intermediate ($N=256$) and fine ($N=512$) fields agree to plotting accuracy throughout the reliable regime and for every order $k=2,3,4$, so the stabilized scheme delivers a resolution- and order-robust solution rather than one that must be chased with mesh refinement. Without SVV the picture is the opposite: the bare scheme is only marginally usable at $\Reyn=10^3$ and is already polluted by grid-scale oscillations at $\Reyn=10^4$, with the contamination growing as the Reynolds number increases and the physical scales become finer. The spectral vanishing viscosity therefore acts precisely where it is needed, on the unresolved high modes, and is what makes the higher-order consistent splitting scheme dependable in the demanding high-Reynolds-number regime that motivates this study.

Figure~\ref{fig:KH_grid_N512} compares the two Reynolds numbers at the fixed resolution $N=512$, using the SVV-stabilized scheme with $k=2,3,4$, at the two later times $t=12$ and $t=20$. By this stage the shear layer has rolled up into a pair of co-rotating primary vortices which, as described in \cite{Schroeder2019KH}, subsequently pair into ever larger structures until a single vortex remains, an inverse transfer of energy from small to large scales that is characteristic of two-dimensional flow and whose final pairing time is extremely sensitive to perturbations. This sensitivity is now visible in the computed fields: at these late times the three temporal orders no longer coincide. At $\Reyn=10^3$ (top row) the $k=2$ field is markedly different from the $k=3$ and $k=4$ fields, which remain close to one another; this pattern holds at both $t=12$ and $t=20$, where the $k=3$ and $k=4$ vortices nearly coincide while the $k=2$ vortex is displaced. At $\Reyn=10^4$ (bottom row) all three orders differ at $t=12$, and although by $t=20$ each has collapsed onto a single large vortex, its location differs from one order to the next. This divergence is not a failure of the scheme but the expected signature of the chaotic late-time pairing anticipated above: the infinitesimal differences in temporal truncation error between the orders are amplified by the flow's extreme sensitivity, so that the phase and timing of the pairing---rather than its qualitative character---vary with $k$. The two Reynolds numbers also differ markedly in texture: at $\Reyn=10^3$ the strong viscous diffusion has erased the thin braids and left smooth, nearly circular vortices, whereas at $\Reyn=10^4$ the cores stay compact and intense, the connecting braids remain thin and sheet-like, and fine secondary filaments survive. What is robust across $k$ is therefore stability rather than the detailed late-time field: at every order the SVV-stabilized run stays well resolved and free of grid-scale oscillations throughout, whereas the corresponding runs without SVV blow up at these resolutions, consistent with the instability of the bare scheme documented above.

\begin{figure}[H]
\centering
\includegraphics[width=\textwidth]{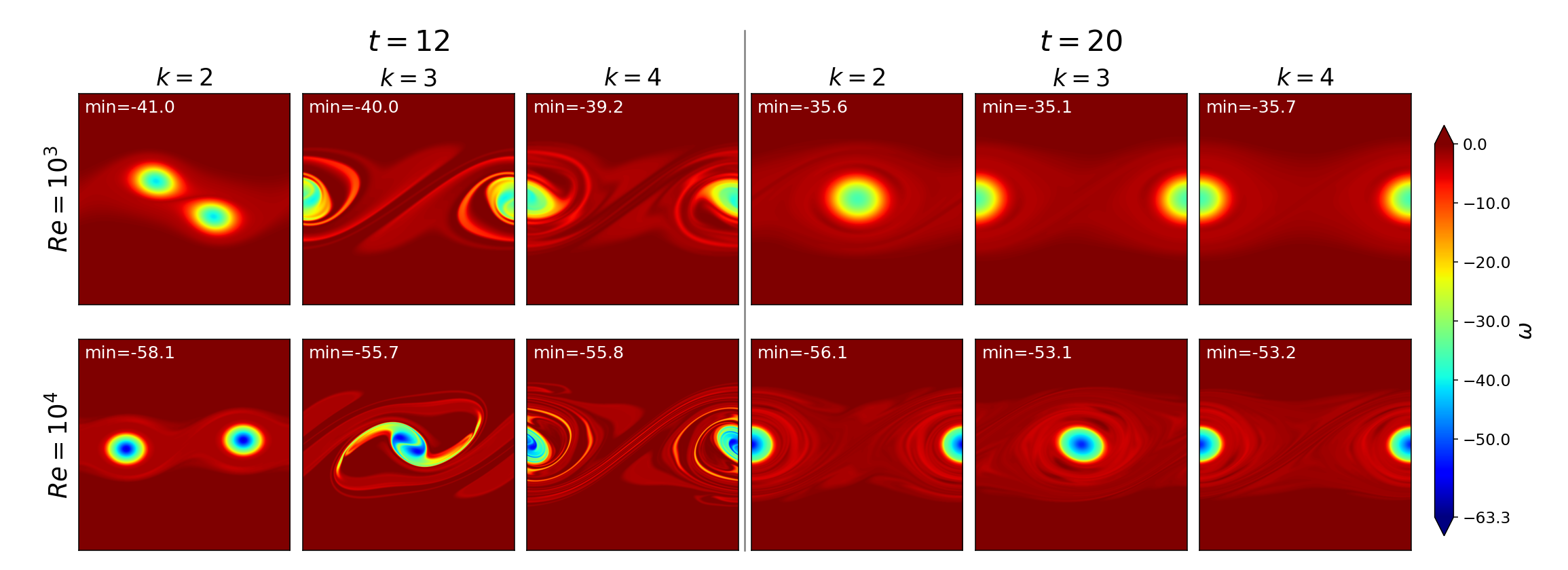}
\caption{Kelvin--Helmholtz problem at $N=512$ with the SVV-stabilized scheme: vorticity $\nabla\times\bu$ for $\Reyn=10^3$ (top row) and $\Reyn=10^4$ (bottom row). In each row the left three panels are $t=12$ and the right three are $t=20$, with $k=2,3,4$ within each group. Common colour scale of \cite{Schroeder2019KH}, $\nabla\times\bu\in[-63.3,\,0]$; the number in each panel is the minimum vorticity.}
\label{fig:KH_grid_N512}
\end{figure}

Finally, we validate the stabilized scheme against the reference solutions of
\cite{Schroeder2019KH}, whose integral time series are available for $\Reyn=10^3$
and $\Reyn=10^4$. We monitor four integral diagnostics. Writing $\mathbf{u}=(u_1,u_2)$
for the velocity, $\omega=\nabla\times\mathbf{u}=\partial_x u_2-\partial_y u_1$ for the
scalar vorticity, and $\langle\,\cdot\,\rangle$ for the average over the periodic
$x$-direction, these are the kinetic energy ($K(t)$), enstrophy ($E(t)$), palinstrophy ($P(t)$), and vorticity
thickness ($\delta_\omega(t)$),
\begin{equation}\label{eq:KH-diagnostics}
\begin{aligned}
K(t)&=\tfrac12\int_\Omega |\mathbf{u}|^2\,d\mathbf{x},
&\qquad
E(t)&=\tfrac12\int_\Omega \omega^2\,d\mathbf{x},\\[4pt]
P(t)&=\tfrac12\int_\Omega |\nabla\omega|^2\,d\mathbf{x},
&\qquad
\delta_\omega(t)&=\frac{2u_\infty}{\displaystyle\max_{y}\bigl|\partial_y\langle u_1\rangle(y,t)\bigr|},
\end{aligned}
\end{equation}
with $\delta_\omega(0)=\delta_0=1/28$. Here $K$ and $E$ measure the total flow energy
and the mean-square vorticity, $P$ measures the mean-square vorticity gradient and
hence the degree of filamentation, and $\delta_\omega$ is the shear-layer (vorticity)
thickness built from the $x$-averaged streamwise profile. Figure~\ref{fig:KH_diagnostics}
overlays the reference of \cite{Schroeder2019KH} (black; solid for $\Reyn=10^3$, dashed
for $\Reyn=10^4$) on these four quantities computed at $N=512$ for $k=2,3,4$, with and
without SVV. The reference data extend to the scaled time $\bar t=400$, i.e.
$t=\bar t\,\delta_0/u_\infty=14.3$, marked by the vertical line. Through the reliable
regime ($t\le 7$) the SVV results are indistinguishable from the reference for every order $k$ and
both Reynolds numbers, most strikingly in the kinetic energy, which the SVV curves track
to plotting accuracy. Beyond the roll-up the curves separate only in the \emph{timing} of the final vortex
pairing. The palinstrophy peak, the enstrophy step and the thickness jump occur at
slightly different times, a manifestation of the extreme sensitivity of that pairing
documented in \cite{Schroeder2019KH}. The kinetic energy, being insensitive to the
pairing, continues to agree throughout. The bare scheme, by contrast, departs from the
reference and, for the under-resolved configurations, blows up. This shows as the
enstrophy and palinstrophy excursions of the ``no SVV'' curves, severe at $\Reyn=10^4$,
where the $k=2$ run diverges near $t\approx4$.

\begin{figure}[H]
\centering
\includegraphics[width=0.49\textwidth]{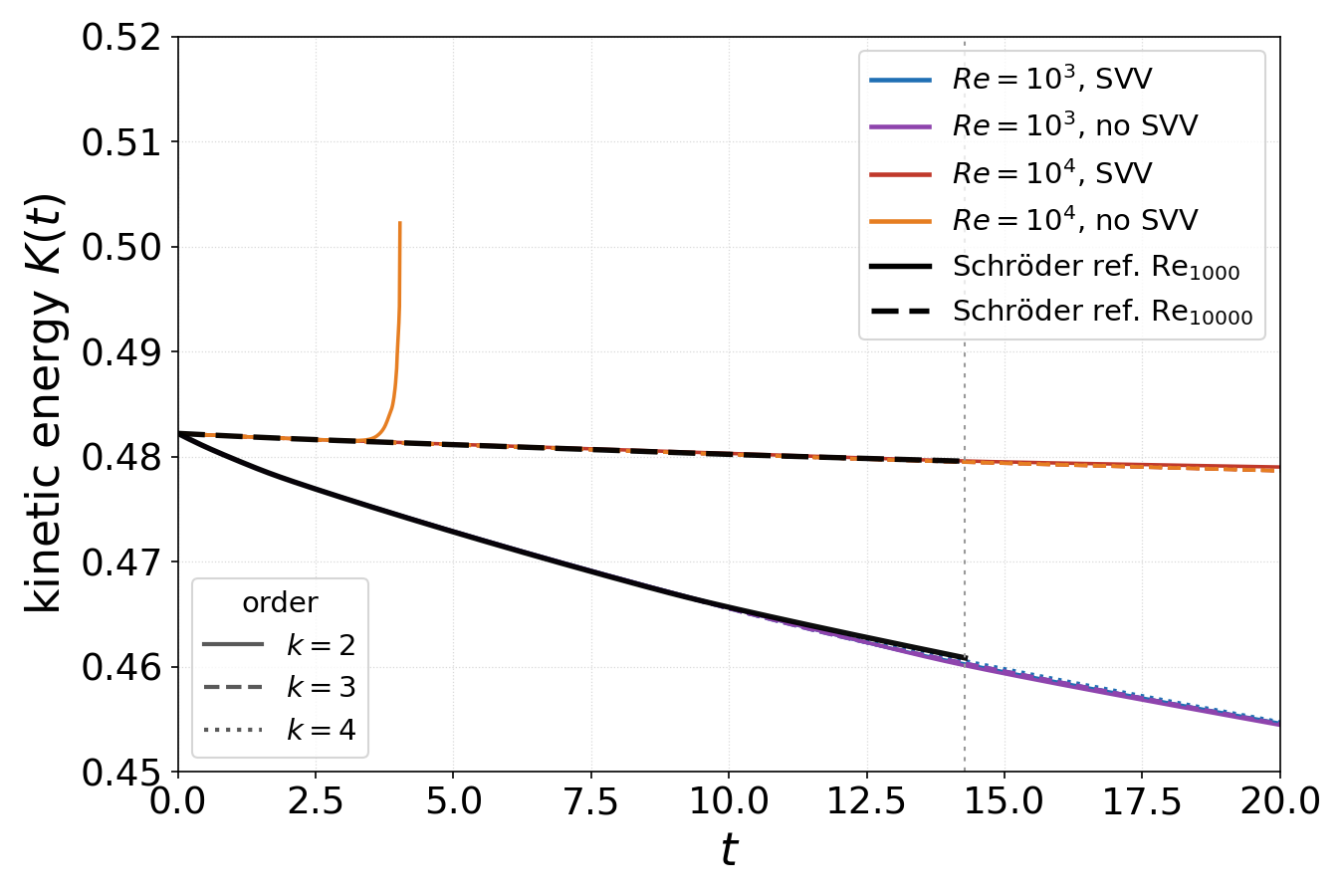}
\includegraphics[width=0.49\textwidth]{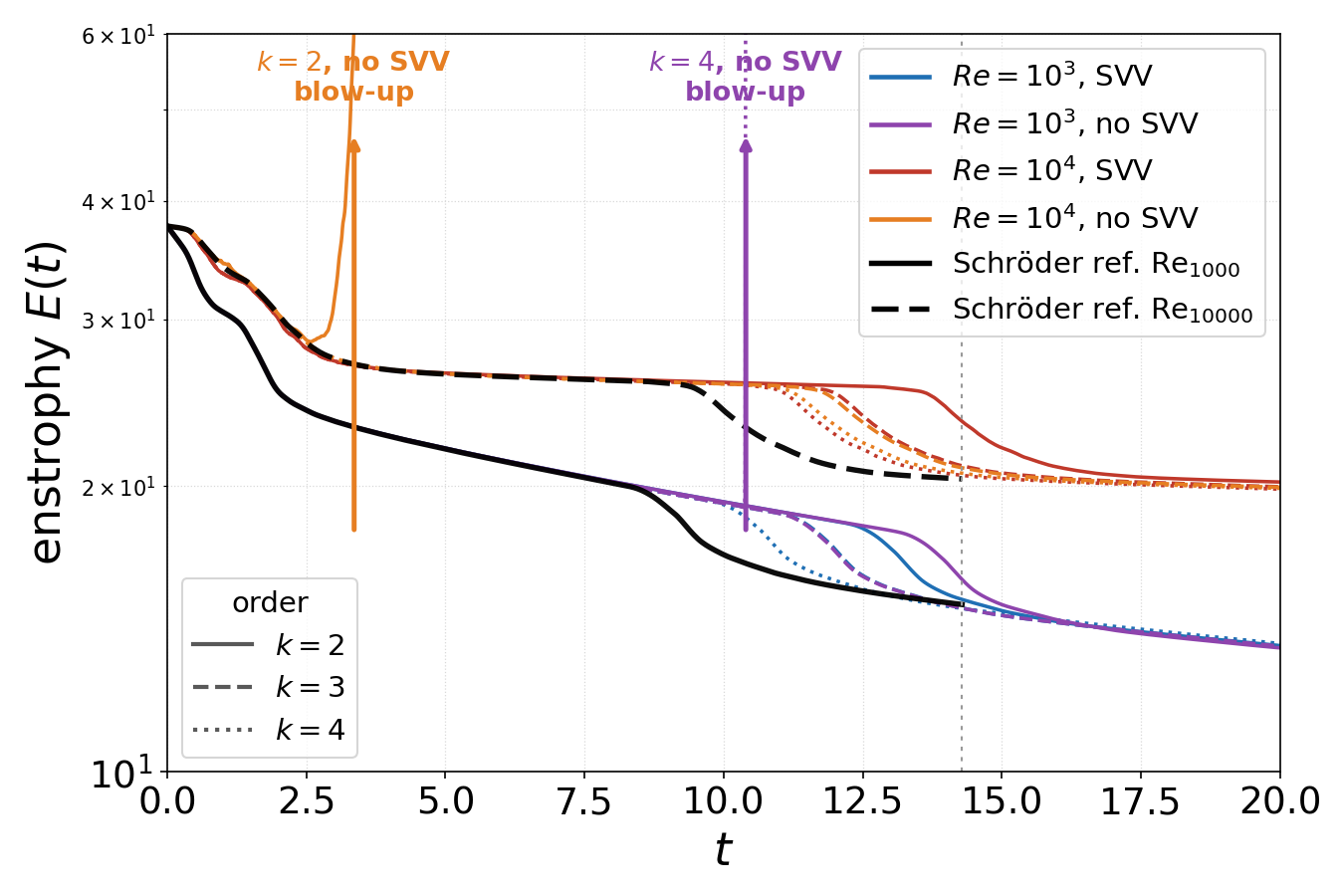}\\[1ex]
\includegraphics[width=0.49\textwidth]{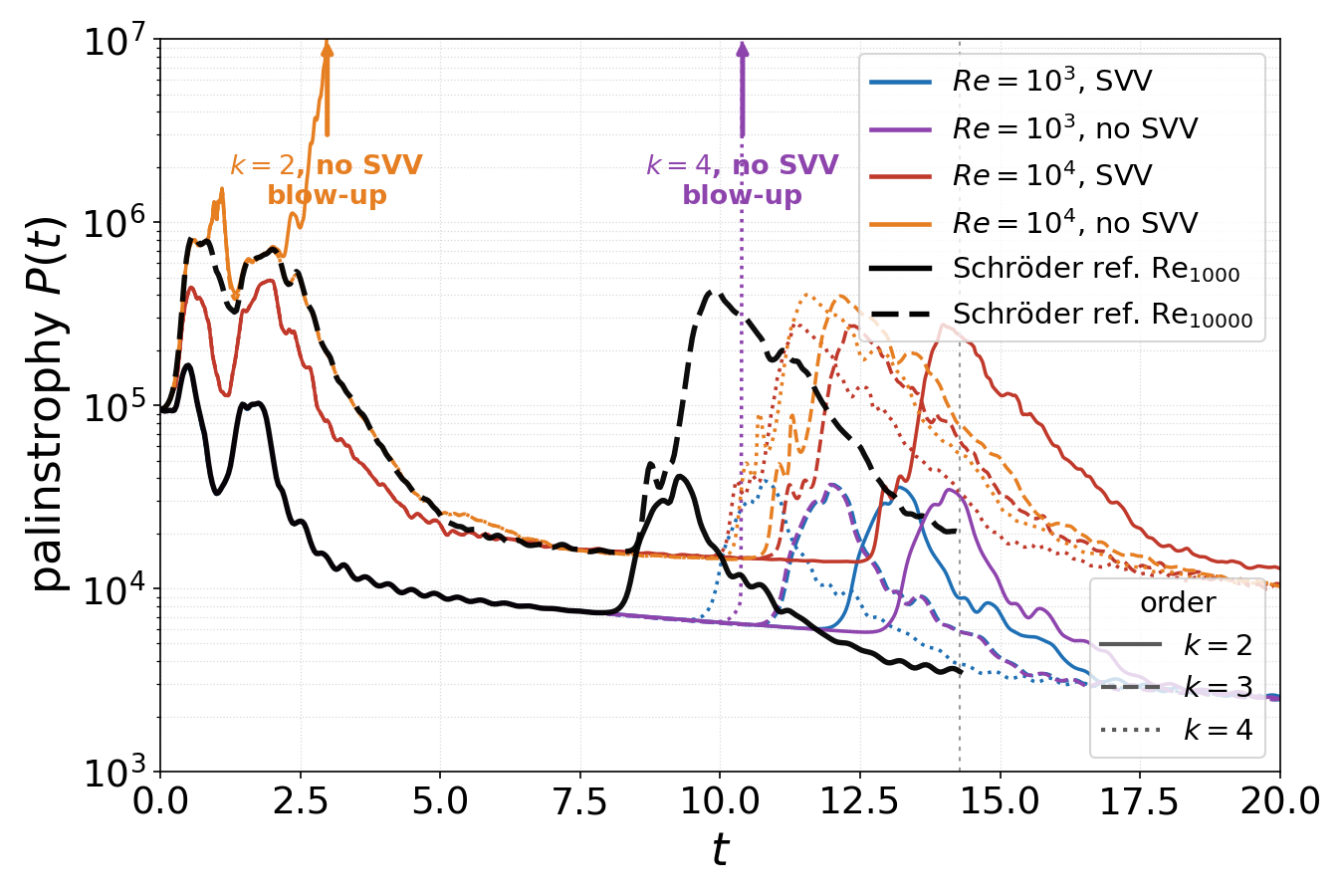}
\includegraphics[width=0.49\textwidth]{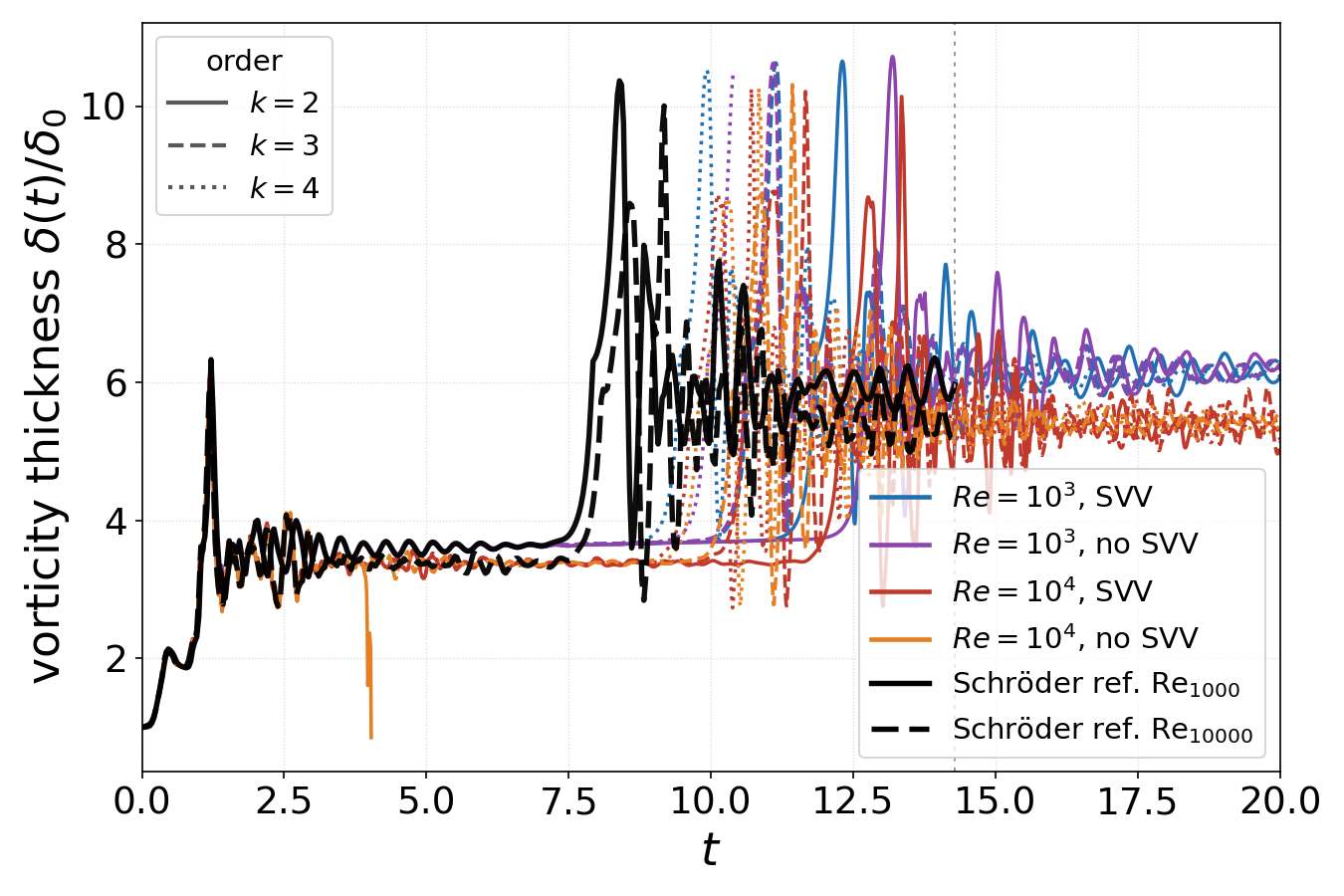}
\caption{Kelvin--Helmholtz integral diagnostics at $N=512$ for $k=2,3,4$ with and without SVV, compared with the reference solutions of \cite{Schroeder2019KH}. Top row: kinetic energy $K$ (left) and enstrophy $E$ (right); bottom row: palinstrophy $P$ (log scale, left) and vorticity thickness $\delta/\delta_0$ (right). Colour encodes the Reynolds number and SVV on/off, line style the order $k$; the reference is drawn in black (solid $\Reyn=10^3$, from Re$_{1000}$; dashed $\Reyn=10^4$, from Re$_{10000}$). Time is physical ($\bar t=28\,t$); the reference ends at $t=14.3$ ($\bar t=400$, vertical line).}
\label{fig:KH_diagnostics}
\end{figure}

\section{Conclusions}
\label{sec:concl}

The BDF--IMEX consistent splitting scheme of Huang and
Shen~\cite{HuangShen2025} for the Navier--Stokes equations is higher-order, but its error estimate carries a constant that
degenerates as $\nu^{-5}$ in the inviscid limit, and in practice the bare scheme blows up at high Reynolds number. To cure this we augment the velocity
update with a symmetric positive-semidefinite spectral-vanishing-viscosity
(SVV) operator $S_N=-\eps_N\diver(\bm{\mathcal Q}_N\nabla)$, built from
the Maday--Kaber--Tadmor kernel~\citep{Tadmor1989,MadayKaberTadmor1993}
applied \emph{directionally}~\citep{SeveracSerre2007,ChenPasquettiXu2021}, so
that each mode is damped along the direction in which it is under-resolved
while the low, resolved modes are left untouched. The modification is
implementationally trivial, a single diagonal correction in the
simultaneous-diagonalization eigenbasis, and preserves the per-step cost of the base scheme. It carries through
unchanged to every order $k=2,3,4$ of the BDF--IMEX family (only the
coefficients $a_{k,k},b_{k,k-1}$ in \eqref{eq:lambda-SVV} change). Because the
added operator is, mode by mode, a scalar multiple of the Laplacian, it leaves
the structure of the error analysis intact.

Analytically (Theorem~\ref{thm:main}), the SVV term contributes an
additional, $\nu$-independent coercive piece
$\eps_N\,\dt\sum\norm{\sqrt{Q_N}\Delta\bu^{n+1}}^2$ on the left of the
energy identity, giving a $\nu$-uniform bound on the high-mode part of the
discrete Laplacian. The right-hand-side constant, however, retains the
$\nu^{-5}$ scaling of the bare scheme: the Stokes-pressure absorption
forces the convection weight $\eps\sim\nu$ on the low-mode subspace,
exactly where the MKT--SVV kernel vanishes. This contrast between the
$\nu$-uniform left-hand side and the unchanged $\nu$-dependent
right-hand side is the analytical fingerprint of the intervention, and
closing the gap between this bound and the robustness observed in practice
remains an open problem.

Three two-dimensional experiments establish that robustness numerically.
In the manufactured-solution convergence test (Example~1), the
SVV scheme reproduces the design temporal order $k=2,3,4$ wherever the
temporal error dominates; the only visible effect of the stabilization is
a mild accuracy floor $\approx10^{-4}$ set by the SVV consistency term
$D_{\rm svv}\le c\,\eps_N T\sup_t\norm{\sqrt{Q_N}\Delta\bu(t)}^2$, which
surfaces only at the smallest time steps for $k=3,4$ and decreases to $1.6\times10^{-5}$ at $N=256$. At $\Reyn=10^4$ the
bare scheme diverges for most time steps under both exact and Richardson--BE
initialization, the error reaching $10^{14}$--$10^{25}$ or overflowing, so the
instability is intrinsic to the scheme at this Reynolds number rather than a start-up
effect. With $N=128$ the SVV scheme stays stable and convergent down to the SVV floor
for every~$\dt$ except the largest step of the fourth-order variant. With $N=256$ both
schemes are subject to a restriction on the time step, which the stabilization relaxes but
does not remove.

On the perturbed Kovasznay flow (Example~2), the bare scheme grows
spuriously at $\Reyn=10^3$ (a minimum near $t\approx40$ followed by growth
to $1.2\times10^{-1}$) and blows up near $t\approx10$ at $\Reyn=10^4$,
whereas the SVV scheme drives the perturbation to exponential decay back to
the steady flow. The decay of the perturbation norm is insensitive to resolution (the
$N=128,256,512,1024$ curves coincide), resolves the thin outflow boundary
layer, converges under mesh refinement, and agrees with an independent
Galerkin--Newton finite-element reference in the position of the layer,
differing only in the peak amplitude that the coarser finite-element mesh
under-resolves.

The Kelvin--Helmholtz problem (Example~3) probes the demanding
high-Reynolds, perturbation-sensitive regime that motivates this study.
With SVV the roll-up of the shear layer is captured cleanly, and the
vorticity fields at $N=128,256,512$ agree to plotting accuracy for every
order $k=2,3,4$ throughout the reliable regime, so the stabilized scheme
is resolution- and order-robust rather than something that must be chased
with mesh refinement. Without SVV the bare scheme is only marginally usable
at $\Reyn=10^3$ and is swamped by grid-scale oscillations at $\Reyn=10^4$
(spurious vorticity $\min\nabla\times\bu\approx-106$, far outside the
physical range) before diverging. Validated against the reference integral
time series of \citet{Schroeder2019KH}, the SVV kinetic energy, enstrophy,
palinstrophy and vorticity thickness are indistinguishable from the
reference through the reliable regime at both Reynolds numbers; beyond the
roll-up the curves separate only in the \emph{timing} of the final,
notoriously perturbation-sensitive vortex pairing, while the kinetic
energy, insensitive to that pairing, continues to track the reference
throughout. Taken together, the three examples show that the spectral
vanishing viscosity acts precisely where it is needed, on the unresolved
high modes, and is what makes the higher-order consistent splitting scheme
dependable at high Reynolds number. The implementation here is
two-dimensional, but the extension to three space dimensions is
straightforward, since the eigenbasis solver carries over directly.

Several questions remain open for future work. On the
theoretical side, the robustness of the error bound in the viscosity could be sharpened: removing the $\nu^{-5}$
factor from Theorem~\ref{thm:main} appears to require a quantitative
low/high-mode splitting of the trilinear convective term, in the spirit of
\cite{MadayKaberTadmor1993} for conservation laws and
\cite{GuoMaTadmor2001} for multi-dimensional spectral viscosity, carried
through with the additional bookkeeping of the BDF--IMEX Taylor-shift
operators. Finally, the SVV accuracy floor, the saturation
$\approx10^{-4}$ observed for $k=3,4$ in Example~1, is governed by the SVV
amplitude $\eps_N=C_{\rm svv}/M$, the kernel cut-off $m_N$, and the algebraic decay of the
solution coefficients in the eigenbasis in which the kernel acts. A sharper,
problem-adapted choice of these parameters, or a defect correction that
removes the SVV consistency error once the flow is resolved, would let the
higher-order variants realize their full temporal accuracy without
sacrificing robustness.

\appendix
\makeatletter
\renewcommand\@seccntformat[1]{Appendix~\csname the#1\endcsname\quad}
\makeatother

\section{Proof of the reduction \textup{(P1)}--\textup{(P2)}}
\label{app:svv-proof}
We prove the reduction (P1)--(P2) of Lemma~\ref{lem:svv-props} for the operator
$S_N$ of Definition~\ref{def:svv}; the remaining properties (P3)--(P6) are the
classical spectral-vanishing-viscosity facts recalled after the lemma. Throughout,
$\bu=\sum_{ij}\widehat u_{ij}\Psi_{ij}$ and
$\bv=\sum_{ij}\widehat v_{ij}\Psi_{ij}$ are in $\bm V_N$, and we use the
orthonormality $(\Psi_{ij},\Psi_{i'j'})=\delta_{ii'}\delta_{jj'}$ together with the
stiffness relations $-\Delta\Psi_{ij}=(\mu_i+\mu_j)\Psi_{ij}$,
$(\partial_x\Psi_{ij},\partial_x\Psi_{i'j'})=\mu_i\delta_{ii'}\delta_{jj'}$, and
$(\partial_y\Psi_{ij},\partial_y\Psi_{i'j'})=\mu_j\delta_{ii'}\delta_{jj'}$.

\medskip\noindent{\bf (P1) Strong form.}\\
Both $S_N$ and $-\eps_N Q_N\Delta$ are diagonal in $\{\Psi_{ij}\}$, so it suffices
to compare them mode by mode. Using $-\Delta\Psi_{ij}=(\mu_i+\mu_j)\Psi_{ij}$,
$Q_N\Psi_{ij}=\widehat Q_{ij}\Psi_{ij}$, and
$\widehat Q_{ij}(\mu_i+\mu_j)=\widehat Q_i\mu_i+\widehat Q_j\mu_j$,
\[
-\eps_N Q_N\Delta\,\Psi_{ij}
=\eps_N\widehat Q_{ij}(\mu_i+\mu_j)\Psi_{ij}
=\eps_N\bigl(\widehat Q_i\mu_i+\widehat Q_j\mu_j\bigr)\Psi_{ij}
=S_N\Psi_{ij}.
\]
Agreement on a basis of $\bm V_N$ gives $S_N=-\eps_N Q_N\Delta$ on $\bm V_N$.

\medskip\noindent{\bf (P2) Directional form.}\\
For $\bu,\bv\in\bm V_N\subset\bm H^1_0(\Om)$, integration by parts (the boundary
term vanishes) gives
\[
-\eps_N\bigl(\diver(\bm{\mathcal Q}_N\nabla\bu),\bv\bigr)
=\eps_N\bigl(\bm{\mathcal Q}_N\nabla\bu,\nabla\bv\bigr)
=\eps_N\Bigl[(\widehat Q^{\,x}\partial_x\bu,\partial_x\bv)
            +(\widehat Q^{\,y}\partial_y\bu,\partial_y\bv)\Bigr].
\]
Here $\widehat Q^{\,x}$ and $\widehat Q^{\,y}$ act on the mode indices, that is,
$\widehat Q^{\,x}\partial_x\Psi_{ij}:=\widehat Q_i\,\partial_x\Psi_{ij}$ and
$\widehat Q^{\,y}\partial_y\Psi_{ij}:=\widehat Q_j\,\partial_y\Psi_{ij}$.
With $\bu=\Psi_{ij}$ and $\bv=\Psi_{i'j'}$,
\[
(\widehat Q^{\,x}\partial_x\Psi_{ij},\partial_x\Psi_{i'j'})
=\widehat Q_i\mu_i\,\delta_{ii'}\delta_{jj'},\qquad
(\widehat Q^{\,y}\partial_y\Psi_{ij},\partial_y\Psi_{i'j'})
=\widehat Q_j\mu_j\,\delta_{ii'}\delta_{jj'}.
\]
Summing, $-\eps_N(\diver(\bm{\mathcal Q}_N\nabla\Psi_{ij}),\Psi_{i'j'})
=\eps_N(\widehat Q_i\mu_i+\widehat Q_j\mu_j)\delta_{ii'}\delta_{jj'}
=(S_N\Psi_{ij},\Psi_{i'j'})$. As this holds for every basis pair, the two operators
coincide on $\bm V_N$. \qed

\section{The Fourier--cosine/sine realization for the Kelvin--Helmholtz problem}
\label{app:fourier-trig}
We describe here the spectral Fourier($x$)$\times$cosine/sine($y$) realization for the Kelvin--Helmholtz problem in Section\,\ref{sec:KH}. The time discretization $\Ak,\Bk,\Ck$ and the SVV operator $S_N=-\eps_N Q_N\Delta$ are unchanged. Only the spatial basis
differs, and it enters the solver only through the eigenpairs of the
one-dimensional operators.

\medskip
\noindent{\bf Fourier basis in $x$.}\\
\noindent
Because the flow is periodic in $x$, each field is expanded in the Fourier basis
$e_m(x)=e^{2\pi\mathrm{i} m x}$, $m=-N_x/2+1,\dots,N_x/2$, on the uniform grid
$x_l=l/N_x$, $l=0,\dots,N_x-1$. Each $e_m$ is an eigenfunction of
$-\partial_{xx}$ with eigenvalue $k_m^2=(2\pi m)^2$, and $\partial_x$ acts as the
multiplier $2\pi\mathrm{i} m$. The Fourier pair $(e_m,k_m^2)$ plays the role that the
one-dimensional simultaneous-diagonalization pair $(\psi_i,\mu_i)$ of $X_N$ plays
in $x$ in Section~\ref{sec:scheme}.

\medskip
\noindent{\bf Cosine/sine basis in $y$.}\\
\noindent
This is the $y$-space used for the Kelvin--Helmholtz computations. The tangential velocity $u_1$ and the pressure
satisfy the Neumann condition $\partial_y(\cdot)=0$ at $y=0,1$ and are expanded in
the cosine family
\begin{equation}
c_m(y)=\cos(m\pi y),\qquad m=0,\dots,N_y-1,
\end{equation}
while the normal velocity $u_2$ satisfies the Dirichlet condition $u_2=0$ at
$y=0,1$ and is expanded in the sine family
\begin{equation}
s_m(y)=\sin(m\pi y),\qquad m=1,\dots,N_y.
\end{equation}
Both families are collocated on the cell-centered uniform grid
$y_j=(j+\tfrac12)/N_y$, $j=0,\dots,N_y-1$, which gives uniform resolution across
the shear layer. Each $c_m$ and $s_m$ is an eigenfunction of $-\partial_{yy}$ with
eigenvalue $(m\pi)^2$, so the cosine pair $\bigl(c_m,(m\pi)^2\bigr)$ and the sine
pair $\bigl(s_m,(m\pi)^2\bigr)$ are the $y$-eigenpairs used by the solver.

A field is represented in the tensor basis $\Psi_{mn}=e_m(x)\,c_n(y)$ for a cosine
field, that is $u_1$ and the pressure, and $e_m(x)\,s_n(y)$ for a sine field, that
is $u_2$. The Laplacian is diagonal,
\begin{equation}
-\Delta\,\Psi_{mn}=\bigl(k_m^2+(n\pi)^2\bigr)\,\Psi_{mn},
\end{equation}
so the Helmholtz operator of each velocity component is diagonal in the index
pair $(m,n)$.

\medskip
\noindent{\bf SVV kernel.}\\
\noindent
The Maday--Kaber--Tadmor kernel \eqref{eq:MKTkernel} is applied
\emph{directionally}, exactly as in \eqref{eq:dirkernel}. In $x$ it is indexed by
the Fourier wavenumber magnitude $|m|$, with cut-off $M_x=N_x/2$ and threshold
$m_{N,x}=\lceil\sqrt{M_x}\,\rceil$, giving $\widehat Q^{x}_{|m|}$. In $y$ it is
indexed by the mode number $n=0,\dots,N_y-1$, with cut-off $M_y=N_y$ and threshold
$m_{N,y}=\lceil\sqrt{M_y}\,\rceil$, giving $\widehat Q^{y}_n$. Both velocity
components use the same one-dimensional kernels. Writing $\mu^x_m=k_m^2$ and
$\mu^y_n=(n\pi)^2$ for the Laplacian eigenvalues in each direction, the operator
is
\begin{equation}\label{eq:SN-fourier}
\begin{aligned}
S_N\,\Psi_{mn}
&=\eps_N\bigl(\widehat Q^{x}_{|m|}\,\mu^x_m+\widehat Q^{y}_n\,\mu^y_n\bigr)\Psi_{mn}
=\eps_N\,\widehat Q_{mn}\,\bigl(\mu^x_m+\mu^y_n\bigr)\,\Psi_{mn},\\
\widehat Q_{mn}&:=\frac{\widehat Q^{x}_{|m|}\mu^x_m+\widehat Q^{y}_n\mu^y_n}
                       {\mu^x_m+\mu^y_n},
\qquad
\eps_N=\frac{C_{\rm svv}}{\max(M_x,M_y)},
\end{aligned}
\end{equation}
so that, as in \eqref{eq:Qtilde}, $S_N=-\eps_N Q_N\Delta$ with the scalar
multiplier $\widehat Q_{mn}\in[0,1]$.

The directional form matters here for a concrete reason. In the Kelvin--Helmholtz
flow the under-resolved structures are the braids and vortex sheets: thin in one
direction and long in the other. Because $\bm{\mathcal Q}_N$ judges each direction
separately, such a mode is damped according to the direction in which it is
unresolved, however smooth it may be in the other.

\medskip
\noindent{\bf Diagonal velocity update.}\\
\noindent
Let $a^\star_k$ and $b^\star_k$ denote the coefficients of $\bu^{n+1}$ in $\Ak$
and $\Bk$. Substituting the tensor basis into the SVV-stabilized momentum
equation \eqref{eq:SVV-vel} makes each velocity component a diagonal solve,
\begin{equation}
\Bigl[\frac{a^\star_k}{\dt}
+ b^\star_k\Bigl(\nu\bigl(k_m^2+(n\pi)^2\bigr)
  + \eps_N\bigl(\widehat Q^{x}_{|m|}k_m^2+\widehat Q^{y}_n(n\pi)^2\bigr)\Bigr)
\Bigr]\,\widehat u_{mn}^{\,n+1}
= \widehat R_{mn},
\end{equation}
where the right-hand side $\widehat R_{mn}$ collects the older $\Ak$ terms, the
extrapolated convection $\Ck(\bu^n)\!\cdot\!\nabla\Ck(\bu^n)$, the extrapolated
pressure gradient $\nabla\Ck(p^n)$, and the explicit part of
$(\nu+\eps_N Q_N)\Delta\Bk(\bu^{n+1})$, all in the tensor basis. The cosine
transform is used for $u_1$ and the sine transform for $u_2$. The stabilizing
operator is the same $-\eps_N Q_N\Delta\Bk(\bu^{n+1})=S_N\Bk(\bu^{n+1})$ as in
\eqref{eq:SVV-vel}.

\begin{remark}
The stabilization is identical to that of the main text. The Legendre eigenpairs
$(\psi_i,\mu_i)$ are replaced by the Fourier pair $(e_m,k_m^2)$ in $x$ and by the
cosine and sine pairs $\bigl(c_n,(n\pi)^2\bigr)$ and $\bigl(s_n,(n\pi)^2\bigr)$ in
$y$, so the Laplacian eigenvalues $(\mu_i+\mu_j)$ of Section~\ref{sec:scheme}
become $(k_m^2+(n\pi)^2)$. The algebraic part of the analysis, namely the properties of
Lemma~\ref{lem:svv-props} and the SVV coercivity of Lemma~\ref{lem:SVVcoerc}, is stated in terms of
the operators $S_N$, $Q_N$ and $\nabla$ and carries over unchanged to this realization.
Theorem~\ref{thm:main} itself is proved for the no-slip boundary condition.
\end{remark}

\section*{Acknowledgments}
J.~Wu was partially supported by the National Science Foundation of
the United States (Grant No.~DMS-2104682 and DMS-2309748). X.~Zheng was
partially supported by NSF grant DMS-2309747. This work was supported in part by
the computational resources and services provided by the High Performance
Computing Center (HPCC) of the Institute for Cyber-Enabled Research at Michigan
State University.

\bibliographystyle{plainnat}
\bibliography{ref_HS2025_SVV_paper}

\end{document}